\documentclass[11pt, a4paper]{amsart}

\usepackage[T1]{fontenc}
\usepackage{microtype}

\usepackage{amsmath}								
\usepackage{amssymb}
\usepackage{amsthm}
\usepackage{amscd}
\usepackage{amsfonts}
\usepackage{fontawesome5}

\usepackage{stmaryrd}

\usepackage{extarrows}

\usepackage[bookmarks=true, colorlinks=true, linkcolor=blue!50!black,
citecolor=orange!50!black, urlcolor=orange!50!black, pdfencoding=unicode]{hyperref}
\usepackage{color}

\usepackage{tikz}									
\usetikzlibrary{matrix}
\usetikzlibrary{patterns}
\usetikzlibrary{matrix}
\usetikzlibrary{positioning}
\usetikzlibrary{decorations.pathmorphing}
\usetikzlibrary{cd}
\usepackage{caption}
\usepackage{subcaption}

\usepackage{fullpage}
\usepackage{enumerate}
\usepackage{mathtools}

\newtheorem*{theorem*}{Theorem}

\newtheorem{maintheorem}{Theorem}[section]
\newtheorem{maincorollary}[maintheorem]{Corollary}

\newtheorem{manualtheoreminner}{Theorem}
\newenvironment{manualtheorem}[1]{%
  \IfBlankTF{#1}
    {\renewcommand{\themanualtheoreminner}{\unskip}}
    {\renewcommand\themanualtheoreminner{#1}}%
  \manualtheoreminner
}{\endmanualtheoreminner}

\newtheorem{theorem}{Theorem}[section]
\newtheorem{lemma}[theorem]{Lemma}
\newtheorem{proposition}[theorem]{Proposition}
\newtheorem{corollary}[theorem]{Corollary} 

\theoremstyle{definition}
\newtheorem{definition}[theorem]{Definition}
\newtheorem{example}[theorem]{Example}
\newtheorem{examples}[theorem]{Examples}

\newtheorem{remark}[theorem]{Remark}
\newtheorem*{remark*}{Remark}

\newtheoremstyle{myitemstyle}						
	{}			
	{}			
	{}			
	{}			
	{}			
	{.}			
	{ }			
	{}			
\theoremstyle{myitemstyle}
\newtheorem{myitemthm}{}

\newcommand{\hooklongrightarrow}{\lhook\joinrel\longrightarrow}

\newcommand{\R}{\mathbb{R}}
\newcommand{\Rbar}{\overline{\mathbb{R}}}
\newcommand{\Z}{\mathbb{Z}}
\newcommand{\ZZ}{\mathbb{Z}}

\newcommand{\C}{\mathbb{C}}

\newcommand{\N}{\mathbb{N}}
\renewcommand{\P}{\mathbb{P}}

\newcommand{\A}{\mathbb{A}}
\newcommand{\G}{\mathbb{G}}
\newcommand{\GG}{\mathbb{G}}

\newcommand{\bbX}{\mathbb{X}}
\newcommand{\bbY}{\mathbb{Y}}

\newcommand{\bbG}{\mathbb{G}}
\newcommand{\bbH}{\mathbb{H}}

\newcommand{\bft}{t}
\newcommand{\bfg}{g_0}

\newcommand{\bfl}{\ell_0}

\newcommand{\bfh}{h_0}

\newcommand{\Deltabar}{\overline{\Sigma}}

\newcommand{\Sigmabar}{\overline{\Sigma}}
\newcommand{\sigmabar}{\overline{\sigma}}

\newcommand{\calA}{\mathcal{A}}
\newcommand{\calB}{\mathcal{B}}

\newcommand{\calD}{\mathcal{D}}

\newcommand{\calE}{\mathcal{E}}

\newcommand{\calF}{\mathcal{F}}

\newcommand{\calH}{\mathcal{H}}

\newcommand{\calM}{\mathcal{M}}
\newcommand{\calN}{\mathcal{N}}
\newcommand{\calO}{\mathcal{O}}
\newcommand{\calP}{\mathcal{P}}

\newcommand{\calS}{\mathcal{S}}

\newcommand{\calU}{\mathcal{U}}

\newcommand{\calV}{\mathcal{V}}

\newcommand{\calX}{\mathcal{X}}

\newcommand{\frakS}{\mathfrak{S}}
\newcommand{\frakF}{\mathfrak{F}}
\newcommand{\fD}{\mathfrak{D}}

\newcommand{\bfp}{\mathbf{p}}
\newcommand{\bfq}{\mathbf{q}}

\newcommand{\gnor}{\vert \! \vert \cdot\vert\! \vert}
\newcommand{\nor}[1]{\vert\! \vert {#1} \vert \! \vert}

\DeclareMathOperator{\Spec}{Spec}
\DeclareMathOperator{\Hom}{Hom}

\DeclareMathOperator{\val}{val}
\DeclareMathOperator{\won}{won}

\DeclareMathOperator{\Trop}{\trop}

\DeclareMathOperator{\id}{id}

\DeclareMathOperator{\Span}{Span}

\DeclareMathOperator{\GL}{GL}

\DeclareMathOperator{\PGL}{PGL}
\DeclareMathOperator{\trop}{trop}

\DeclareMathOperator{\SL}{SL}
\DeclareMathOperator{\coker}{coker}
\DeclareMathOperator{\ev}{ev}
\DeclareMathOperator{\build}{\textrm{\faHome}}
\DeclareMathOperator{\vspan}{span}
\DeclareMathOperator{\Star}{star}
\DeclareMathOperator{\an}{an}
\DeclareMathOperator{\swap}{swap}
\DeclareMathOperator{\opp}{opp}
\DeclareMathOperator{\gp}{gp}
\DeclareMathOperator{\quot}{quot}
\DeclareMathOperator{\Berk}{Berk}
\DeclareMathOperator{\weak}{weak}
\DeclareMathOperator{\rig}{rig}

\title{Towards the tropicalization of reductive groups}%

\date{}

\author{Desmond Coles}
\address{The University of Texas at Austin,
Department of Mathematics,
Austin, TX 78712,
USA
}
\email{dcoles@utexas.edu}

\author{Martin Ulirsch}
\address{Institut f\"ur Mathematik, Goethe--Universit\"at Frankfurt,
60325 Frankfurt am Main, Germany}
\email{ulirsch@math.uni-frankfurt.de}

\begin{document}

\begin{abstract}
Let $G$ be a connected reductive algebraic group over an algebraically closed field of characteristic zero carrying the trivial valuation. In this article we discuss two candidates for what could be the tropicalization of $G$. 

Our first suggestion is the extended affine building associated to $G$. This perspective makes makes use of Berkovich's embedding of the extended affine building into the Berkovich analytic space $G^{\an}$ and expands on work of Mumford by associating a toroidal bordification of $G$ to the choice of stacky fan in the building. We show that the natural retraction onto the building is compatible with the tropicalization map associated to a toroidal bordification. 

Our second suggestion is a Weyl chamber of $G$, a special instance of spherical tropicalization, where we think of $G$ as a spherical $G\times G$-variety with respect to left-right-multiplication. We show that the spherical tropicalization map may be identified with the toroidal tropicalization map associated to a wonderful compactification of $G$. This map also has a moduli-theoretic interpretation expanding on the compactifications of $G$ as moduli spaces of framed $\mathbb{G}_m$-equivariant principal bundles on chains of projective lines introduced by Martens and Thaddeus. 
\end{abstract}

\maketitle

\setcounter{tocdepth}{1}
\tableofcontents


\newpage

\section*{Introduction}

Let $k$ be an algebraically closed field of characteristic 0 equipped with the trivial valuation. Let $T$ be an algebraic torus over $k$ and $N$ its cocharacter lattice. We denote by $T^{\trop}=N_\R=N\otimes \R$ the \textit{tropicalization} of $T$. Let $M=\Hom(N,\Z)$ be the character lattice of $T$. 
Denote by $T^{\an}$ the Berkovich analytification of $T$. There is a natural continuous and proper surjective tropicalization map $\trop\colon T^{\an}\rightarrow T^{\trop}$ that restricts to 
\begin{equation*}
T(K)\ni \gamma\longmapsto \big(m\mapsto -\log \nor{\chi^m(\gamma)}_K\big)\in \Hom(M,\R)=T^{\trop}
\end{equation*}
for every non-Archimedean extension $K$ of $k$. 

The tropicalization map admits a canonical section $T^{\trop}\hookrightarrow T^{\an}$ and  
can be extended to toric varieties. This is a starting point for what one might refer to as \emph{embedded tropical geometry}, where one studies the geometry of projections of closed subvarieties of $T$ under the tropicalization map (see \cite{MaclaganSturmfels} and \cite{Gubler_guide} for introductions to (embedded) tropical geometry).

Algebraic tori are a particular example of reductive groups, which leads us to ask the following question: Given a (connected) reductive algebraic group $G$, what does it mean to tropicalize $G$? 

This article aims to give two (partial) answers to this question. Our first answer procceds by exploring the connections between a generalization of Mumford's toroidal bordifications of $G$ coming from \cite[Section 4.2]{KKMSD_toroidal} and their associated toroidal tropicalization maps on the one side as well as the geometry of Bruhat-Tits buildings and its realization within Berkovich analytic geometry on the other side (see e.g. \cite[Section 5.4]{Berkovich_book} as well as  \cite{RemyThuillierWernerI, RemyThuillierWernerII, RemyThuillierWerner_wonderful, RemyThuillierWener_Intrinsic, BKKUV} for more details on the latter). Our second answer focuses on the interaction between the spherical tropicalization of $G$ (as e.g.\ in \cite{TevelevVogiannou, Nashpreprint, Coles}) and the wonderful compactification of $G$ (see \cite{deConciniProcesi, BrionKumar}).

\subsection*{The extended affine building} Let $G$ be a connected reductive group over $k$. Then we define the \emph{(extended affine) building} of $G$ to be the colimit of sets
\begin{equation}\label{eq_Gbuild}
G^{\build}:=\varinjlim_{T\le G} T^{\trop}.
\end{equation}
The building $G^{\build}$ can be endowed with two a priori different topologies. The first is the \emph{weak topology}: A subset is closed if and only if its intersection with every $T^{\trop}$ is closed. We write $G^{\build}_{\weak}$ to denote $G^{\build}$ with the weak topology. The second topology is a bit more coarse. As explained in \cite[Section 5.4]{Berkovich_book} and Section \ref{subsec_BTbuildings} below, the extended affine building $G^{\build}$ admits a natural injection into the Berkovich analytification $G^{\an}$ of $G$ (see Section \ref{subsec_BTbuildings} below for details). The \emph{Berkovich topology} on $G^{\an}$ is induced from the one on $G^{\an}$. We write $G^{\build}_{\Berk}$ for the resulting topological space. In this notation the identity defines a continuous bijection 
\begin{equation*}
G^{\build}_{\weak}\xrightarrow{\id_{G^{\build}}} G^{\build}_{\Berk}.
\end{equation*} Note that this map is not in general a homeomorphism.

\begin{remark*}\label{Remark_DefinitionOfTheBuilding}
    Note that our definition of the extended affine building is slightly different from what is usually found in the literature on Bruhat-Tits buildings. 
    \begin{itemize}
    \item When $G$ is semisimple, one usually considers the  \emph{spherical building} associated to $G$. In this case, the space $G^{\build}$ is the open cone over the spherical building. 
    \item For a more general reductive group, many authors often only consider the building associated its semisimplification $G/Z(G)$. In this convention the building associated to an algebraic torus would simply be a point.  
    \end{itemize}
We decided against both of these conventions and in favor of our definition \eqref{eq_Gbuild} in order to emphasize the similarity to the well-known tropicalization of algebraic tori. Our convention aligns with the perspective taken in \cite{KMBundles} and \cite{KSBundles} where toric principal bundles are classified in terms of piecewise linear maps from fans to extended affine buildings.
\end{remark*}

When $G=\textrm{GL}_n$, we may identify $G^{\build}$ with the collection of all norms on the vector space $k^n$. When $G=\textrm{PGL}_n$, then $G^{\build}$ is the collection of norms on $k^{n+1}$, equivalent up to homothety. The natural topology on these spaces is the coarsest topology such that evaluation at any element of $V$ is continuous, which is even coarser than the Berkovich topology. We refer the reader to Section \ref{section_examples} below for further details as well as to \cite{BKKUV} and \cite{DressTerhalle} for a recent and a not so recent exposition of this story and its connection to tropical linear spaces/valuated matroids. We expect that a generalization of this perspective beyond the $A_n$ case can be obtained by employing the Tannakian formalism  introduced in \cite{Ziegler_Tannakianbuilding}.

\subsection*{Extended buildings and toroidal bordifications} The central fact of toric geometry (as e.g. introduced in \cite[Section I]{KKMSD_toroidal}) is that to every rational polyhedral fan $\Delta$ in $N_\R$ we may associate a toric variety $X=X(\Delta)$ whose $T$-orbits are in natural order-reversing one-to-one correspondence with the cones in $\Delta$ (and that, vice versa, every toric variety arises in this fashion). This construction generalizes to \emph{toric stacks} that are constructed from so-called \emph{stacky fans} (see \cite{GeraschenkoSatrianoI, GeraschenkoSatrianoII} for an overview on the different approaches to this construction).

 In Section \ref{section_stackyfans} below we shall see how to define a notion of a \emph{stacky fan} $(\Delta,\Phi)$ in $G^{\build}$, consisting of a rational polyhedral fan $\Delta$ in $G^{\build}$ and a compatible collection of Kummer homomorphisms $\phi_\sigma\colon S_\sigma\rightarrow \widetilde{S}_\sigma$ for all $\sigma\in\Delta$. In Section \ref{section_toroidalbord} we expand on \cite[Section IV.2]{KKMSD_toroidal} and explain how to associate to $(\Delta,\Phi)$ a natural toroidal bordification $\calX_G(\Delta,\Phi)$ of $G$ that generalizes the construction of a toric variety (or more generally a toric stack) from its (stacky) rational polyhedral fan. 

\begin{maintheorem}\label{mainthm_stackyfan=toroidalbordification}
Let $G$ be a reductive group. For every stacky rational polyhedral fan $(\Delta,\Phi)$ in $G^{\build}$ there is a separated Deligne-Mumford stack $\calX_G(\Delta)$ locally of finite type over $k$ that contains $G$ as a dense open subset such that the following properties hold:
\begin{enumerate}[(i)]
\item The embedding $G\hookrightarrow \calX_G(\Delta)$ is a toroidal embedding;
\item there is a natural order-reversing one-to-one correspondence between the toroidal strata of $\calX_G(\Delta)$ and the cones in $\Delta$;
\item the Deligne-Mumford stack $\calX_G(\Delta)$ is smooth if and only if every cone in $\Delta$ is generated by a basis of in the dual lattice to $\widetilde{S}_\sigma$; 
\item for a one-parameter subgroup $u\colon \GG_m\hookrightarrow G$ we have
\begin{equation*}
\lim_{t\rightarrow 0}u(t)\in \calX_G(\Delta)
\end{equation*}
if and only if the $u\in \Delta$;
\item the multiplication $G\times G\rightarrow G$ extends to a natural left operation of $G$ on $\calX_G(\Delta)$; and,
\item if the operation of $G$ on $G^{\build}$ lifts to an operation of $G$ on the stacky fan $\Delta$, then the conjugation operation of $G$ on itself extends to a set-theoretic operation of $G(k)$ on $\calX_G(\Delta)$.
\end{enumerate}
\end{maintheorem}

Theorem \ref{mainthm_stackyfan=toroidalbordification} is a generalization of a construction of Mumford in \cite[Section IV.2]{KKMSD_toroidal} that proposes a (non-stacky) toroidal bordification of semisimple algebraic groups (in an effort to find an algebro-geometric interpretation of spherical buildings). The central idea of his and subsequently our construction is to first $T$-equivariantly bordify each torus $T$ of $G$ separately by the $T$-toric stack $\calX_T(\Delta_T)$ associated to the stacky fan $\Delta_T=\Delta\cap T^{\trop}$ in $T^{\trop}$. We then lift this bordification to a bordification $\calX_T(\Delta_T)\times^T G$ of $G$ and glue by forming the colimit
\begin{equation*}
\calX_G(\Delta)=\varinjlim_{T\le G} \big(\calX_T(\Delta_T)\times^T G\big) \ .
\end{equation*}

Our insistence on including the stacky perspective comes from the fact that for example in the case $G=\SL_n$ there is a minimal choice for $\Delta$ which is only a simplicial fan, but not a smooth one when $n\geq 3$. The resulting bordification will only be at the same time smooth and canonical if we allow for a stacky resolution of the resulting quotient singularities (see Section \ref{subsection_PGLSL} and, in particular, Figure \ref{fig_thefigure} below). Our approach via stacky resolutions is inspired by the analogous problem in the case of  wonderful compactifications of $G$ that was observed and resolved by Martens and Thaddeus in \cite{MartensThaddeus}.

\subsection*{Toroidal bordifications and tropicalization} There are two ways of associating a natural tropicalization map to $G$ that lands in the affine extended building $G^{\build}$.

For the \textbf{first approach} we recall that, for any torus $T$ in $G$, there is a continuous map  $\trop_T\colon T^{\an}\rightarrow T^{\trop}\subseteq G^{\build}_{\Berk}$ given by the tropicalization map for $T$. Thus, we have a continuous tropicalization map
\[
\trop_{\build}\colon \bigcup_{T\leq G}T^{\an}\rightarrow G^{\build}_{\Berk}.
\]
We may enlarge the domain of this map as follows: 
\begin{itemize}
\item Any two maximal tori in $G$ are conjugate by some element of $G(k)$, and thus if we fix a maximal torus $T$ we can represent a point in $G^{\build}$ as a pair $(g,\lambda)$ where $\lambda \in T^{\trop}$ and $g\in G(k)$. 
\item Denote by  $G^{\beth}$ Thuillier's generic fiber of $G$ (in the sense of \cite[Prop.\ et D\'ef.\ 1.3]{Thuillier_toroidal}), which, in this case, is an analytic subdomain of $G^{\an}$ by \cite[Prop. 1.10]{Thuillier_toroidal}. Then $G^{\beth}$ is a maximal compact subgroup of $G^{\an}$ and we have that $G^{\an}=G^{\beth}T^{\an}G^{\beth}$ (see Proposition \ref{Prop_CartanDecomp} below). 
\item Consider now the space $\fD:=G(k)T^{\an}G^{\beth}$. Note that, because all maximal tori are conjugate, and $G(k)\subseteq G^{\beth}$, this space is independent of the choice of $T$ and it contains $S^{\an}$ for any torus $S$.
For a point $p\in \fD$, with $p$ of the form $gth$, we can define $\trop_{\build}(p)=(g,\trop(t))$. 
\end{itemize}
This defines a map from $\trop_{\build}\colon \fD\rightarrow G^{\build}_{\Berk}$ and in Section \ref{Sec_SkeletonAndBuilding} below we show that this map is well-defined and continuous. 

In \cite[Chapter 5.5]{Berkovich_book}, Berkovich defines a map $\Theta\colon G^{\build}_{\Berk}\rightarrow G^{\an}$ that lands in $\fD$ and is a continuous section of $\trop_{\build}$. The map is given by $(g,\lambda)\mapsto g\bft_{\lambda}\ast \bfg$ where $\bfg$ is the Shilov boundary point of $G^{\beth}$. Furthermore, there is a deformation retraction $\bfq\colon \fD\rightarrow \Theta(G^{\build}_{\Berk})$ given by $p\mapsto p\ast\bfg$ (see Section \cite[Corollary 6.1.2]{Berkovich_book}), thereby making the diagram
\begin{equation*}\begin{tikzcd}
 & \fD \arrow[rd, "\bfq"]\arrow[ld,"\trop_{{\build}}"']& \\
G^{\build}_{\Berk}\arrow[rr, "\Theta"'] & & \Theta(G^{\build}_{\Berk}) 
\end{tikzcd}\end{equation*} 
commute.  

For the \textbf{second approach} towards a tropicalization map landing in $G^{\build}$ we consider a stacky fan $\Delta$ in $G^{\build}$ as well as the associated toroidal bordification $\calX_G(\Delta)$ as in Theorem \ref{mainthm_stackyfan=toroidalbordification}. Denote by $\Sigma_\Delta$ the (rational polyhedral) cone complex associated to $\Delta$. We note that the stack $\calX_G(\Delta)$ is not quasi-paracompact when $\Delta$ contains cones in infinitely many distinct maximal tori, so $\calX_G(\Delta)^{\beth}$ does not exist as a Berkovich space. But we can construct $\calX_G(\Delta)^{\beth}$ in the category of topological spaces as a colimit of the affinoid spaces $\calX_G(\sigma)$ where $\sigma$ is a cone in $\Delta$ (see Section \ref{Sec_SkeletonAndBuilding} below for further details). The space $\calX_G(\Delta)^{\an}$ does exist and by the universal property of the colimit there is a continuous map $c\colon \calX_G(\Delta)^{\beth}\rightarrow \calX_G(\Delta)^{\an}$.

By \cite{Ulirsch_functroplogsch, Ulirsch_nonArchArtin} there is a natural continuous tropicalization map $\trop_{\calX_G(\Delta)}\colon \calX_G(\Delta)^{\beth}\rightarrow \Sigmabar_{\Delta}$ that, using the work of \cite{Thuillier_toroidal}, admits a continuous section $J=J_{\calX_G(\Delta)}\colon \Sigmabar_\Delta\rightarrow \calX_G(\Delta)^{\beth}$ making the composition $\bfp_{\calX_G(\Delta)}=J_{\calX_G(\Delta)}\circ \trop_{\calX_G(\Delta)}$ into a strong deformation retraction onto the \emph{toroidal skeleton} of $\calX_G(\Delta)^\beth$. The (non-extended) cone complex $\Sigma_{\Delta}$ admits an embedding into $G^{\build}_{\weak}$ denoted $\iota_{\Delta\subseteq G^{\build}}$.

\begin{maintheorem}\label{mainthm_skeletonbuilding}
Let $\Delta$ be a stacky fan in $G^{\build}$. 
\begin{enumerate}[(i)]
\item\label{item_mainthmb1} The two tropicalization maps $\trop_{\build}$ and $\trop_{\calX_G(\Delta)}$ as well as the two retractions $\bfp_{\calX_G(\Delta)}$ and $\bfq$ are equal. To be precise; we have that $\trop_{\calX_G(\Delta)}^{-1}(\Sigma_{\Delta})$ is equal to $c^{-1}(G^{\an})$, the continuous map $c\colon \calX_G(\Delta)^{\beth}\rightarrow \calX_G(\Delta)^{\an}$ sends $c^{-1}(G^{\an})$ into $\fD$, and
 the diagram 
\begin{equation*}
    \begin{tikzcd}
        c^{-1}(G^{\an})\arrow[dd,"\trop_{\calX_G(\Delta)}"]\arrow[rr,"c"]& &  \fD \arrow[dd,"\trop_{\build}"'] \\ \\\Sigma_{\Delta}\arrow[r,hookrightarrow,"i_{\Delta\subseteq G^{\build}}"]\arrow[uu,bend left,"J"] &  G_{\weak}^{\build} \arrow[r,"\id_{G^{\build}}"] & G^{\build}_{\Berk} \arrow[uu,bend right,"\Theta"'] 
    \end{tikzcd}
\end{equation*}
commutes. 
\item\label{item_mainthmb2} When the support of the stacky fan $\Delta$ covers all of $G^{\build}$, the map $i_{\Delta\subseteq G^{\build}}\colon \Sigma_{\Delta}\hookrightarrow G^{\build}_{\weak}$ is a homeomorphism and $c$ defines a continuous bijection between $c^{-1}(G^{\an})$  and $\fD$.
\end{enumerate}
\end{maintheorem}

Theorem \ref{mainthm_skeletonbuilding} exhibits a close relationship between the extended building, embedded into $G^{\an}$, and the toroidal skeleton associated to the toroidal bordification. It would be very interesting to understand the precise relationship between the embedded building in the case of non-constant coefficients and toroidal bordifications over valuation rings in the sense of \cite[Section IV.3]{KKMSD_toroidal}. We hope to return to this question in a future article.

\subsection*{Spherical tropicalization and toroidal tropicalization}

Recall that a $G$-variety $X$ is said to be a spherical variety if it is normal and there is a Borel subgroup $B\subseteq G$ with an open orbit in $G$. Spherical varieties are nonabelian analogues of toric varieties, and they can also be studied using combinatorial data similiar to that of toric varieties as originally described in \cite{LunaVust}.

Let $X$ be a spherical $G$-variety. As introduced in \cite{TevelevVogiannou}, \cite{Nashpreprint}, and \cite{Coles}, one can define a \textit{spherical tropicalization} $X^{\trop_G}$ of $X$, which is the target of a spherical tropicalization map $\trop_{G}\colon X^{\an}\rightarrow X^{\trop_G}$ (see \cite[Section 2]{NashExtendedVIaToric} or \cite[Subsection 5.2]{Coles} for details). As in the toric case the spherical tropicalization map admits section $J_G\colon X^{\trop_{G}}\rightarrow X^{\an}$ and the composition $\bfp_G=J_G\circ\trop_G$ defines a retraction from $X^{\an}$ onto a closed subset of $X^{\an}$. 

Again in a similar manner to toric varieties we can restrict the spherical tropicalization map to the space $X^{\beth}$, we will denote the image by $\Sigmabar_X^G$. The retraction $\bfp_G$ defines a strong deformation retraction from $X^{\beth}$ onto a closed subset that is equal to $J_G(\Sigmabar_X^G)$.
Suppose that the embedding $G/H\hookrightarrow X$ is toroidal. Then, as above, there is a tropicalization map $\trop_X\colon X^\beth\rightarrow \Sigmabar_X$ associated to $X$ that admits a continuous section $J$ making the composition $\bfp$ into a strong deformation retraction onto $J(\Sigmabar_X)$. The following Theorem shows that the two tropicalizations, $\Sigmabar_X$ and $\Sigmabar_X^G$, are canonically homeomorphic and that the tropicalization maps agree.

\begin{maintheorem}\label{mainthm_sphertroplogtrop}
Let $G$ be a connected reductive group and let $X$ be a spherical $G$-variety with open $G$-orbit $G/H$. Assume the embedding $G/H\hookrightarrow X$ is a toroidal embedding. Then we have that the spherical tropicalization map and the tropicalization map for toroidal embeddings agree. To be precise, we have $J_G(\Sigmabar_X^G)=J(\Sigmabar_X)$ and the two retraction maps $\bfp$ and $\bfp_G$  are the same.
\end{maintheorem}

Theorem \ref{mainthm_sphertroplogtrop} is of independent interest, since it fills a significant gap in the literature on the tropical geometry of spherical varieties. In this article we mostly care about the case when $G$ itself is considered as a spherical variety. 

\subsection*{Affine buildings and spherical tropicalization}

The group $G$ is acted on by $G\times G$ via $(g,h)\cdot x=gxh^{-1}$ which gives $G$ the structure of a spherical $G\times G$-variety. Here we have that the spherical tropicalization $G^{\trop_{G\times G}}$, which we will denote $\calV$, can be identified with $N_{\R}/W$, where $N$ is the cocharacter lattice of some maximal torus in $T$ and $W$ is the Weyl group associated to $T$. Given another maximal torus $S=gTg^{-1}$ in $G$, we can define a map from $S^{\trop}$ to $\calV$ via $x\mapsto g^{-1}xgW$. This way we obtain a continuous map $\pi \colon G^{\build}_{\Berk}\rightarrow \calV$.

\begin{maintheorem}
The map $\pi\colon G^{\build}\rightarrow \calV$ is well-defined, independent of the choice of $T$, and its restriction to the Weyl chambers in $G^{\build}$ is a linear isomorphism. Moreover, the diagram
        \begin{equation}\label{eq_factorization}\begin{tikzcd}
            & \fD \arrow[ld, "\trop_{\build}"'] \arrow[r,"\subseteq"]&G^{\an}\arrow[rd,"\trop_{G\times G}"]& \\
            G^{\build}_{\Berk}\arrow[rrr, "\pi"]  & && \calV 
\end{tikzcd}\end{equation} 
commutes and the factorization \eqref{eq_factorization} is functorial with respect to homomorphisms of reductive groups. To be precise, given $\xi: G_1\rightarrow G_2$ be a morphism of algebraic groups then the following diagram commutes:
     \begin{equation*}
\begin{tikzcd}
G_1^{\build }\arrow[d,"\xi^{\build}"'] \arrow[rr,"\pi"] & & \calV_1\arrow[d,"\trop(\xi)"]\\
G_2^{\build} \arrow[rr,"\pi"] & & \calV_2.
\end{tikzcd}
\end{equation*}
\end{maintheorem}

\subsection*{Wonderful compactifications and toroidal tropicalization}

The map $\pi$ given above has a direct connection with a well-known class of equivariant compacitifications of $G$, which, due to their particularly nice properties, are known as the \emph{wonderful compactifications} of $G$. We consider the group $G\times G$ acting on $G$ via $(g,h)\cdot x=gxh^{-1}$. From the perspective of spherical varieties, a wonderful compactification of $G$ is a $G\times G$-equivariant compactification of $G$ such that the inclusion of $G$ into the compactificaiton is a toroidal embedding.
Contrary to our construction of a toroidal bordification of $G$ above, a wonderful compactification depends on the choice of a maximal torus $T$ of $G$ (note however that isomorphism classes of wonderful compactifications do \textit{not} depend on a choice of torus).

The construction of wonderful compactifications goes back to de Concini and Procesi \cite{deConciniProcesi} in the case of a semisimple algebraic group of adjoint type over the complex numbers. When the group $G$ is semi-simple there is a unique minimal wonderful compactification, it is minimal in the sense that it has the smallest possible number of $G\times G$-orbits, furthermore every other wonderful compactification of $G$ admits a proper $G\times G$-equivariant morphism to it.
The construction has been extended to reductive groups by Brion and Kumar \cite{BrionKumar}. If $X$ is the wonderful compactification of $G$, then the open embedding $G\hookrightarrow X$ is a toroidal embedding, and thus there is a tropicalization map $\trop_{\won} \colon X^{\beth}\rightarrow \overline{\Sigma}_X$. Furthermore, $X$ is a spherical $G\times G$-variety, and the action of $G\times G$ on $X$ extends the action of $G\times G$ on $G$ described above. The wonderful compactification of $G$ is complete, so $X^{\beth}=X^{\an}$, and, thus, Theorem \ref{mainthm_sphertroplogtrop} implies:

\begin{maincorollary}
We have $J_{G\times G}(\calV)=J(\Sigmabar_X)$ and the retraction maps $\bfp_{G\times G} $ and $\bfp$, from $X^{\an}$ onto $J(\Sigmabar_X)$ are equal. In this sense the maps $\trop_{G\times G}$ and $\trop_{\won}$ agree.
\end{maincorollary}

The intimate connection between wonderful compactifications and spherical tropicalizations explains why we sometimes like to refer to the spherical tropicalization of $G$ as the \emph{wonderful tropicalization}.

In \cite{MartensThaddeus}, the authors explain how the wonderful compactification(s) can be seen as special cases of a family of (partial) compactifications of $G$ as moduli spaces of $\G_m$-equivariant principal $G$-bundles on chains of projective lines. In Section \ref{section_bundles} below we use this perspective to describe the tropicalization of $G$ in terms of a suitable map of moduli spaces. When $G=\GL_n$, this may be seen as an instance of the geometry of principal bundles on metric graphs developed in \cite{GrossUlirschZakharov}. This also gives a --as far as we are aware-- first instance of spherical tropicalization that can be described using a moduli space of tropical objects.

\subsection*{Funding} This project has received funding from, the Deutsche Forschungsgemeinschaft (DFG, German Research Foundation) TRR 326 \emph{Geometry and Arithmetic of Uniformized Structures}, project number 444845124, from the DFG Sachbeihilfe \emph{From Riemann surfaces to tropical curves (and back again)}, project number 456557832, as well as the DFG Sachbeihilfe \emph{Rethinking tropical linear algebra: Buildings, bimatroids, and applications}, project number 539867663, within the  
SPP 2458 \emph{Combinatorial Synergies}. This project has also received funding from the U.S. National Science Foundation's Division of Mathematical Sciences grants, award numbers 2302475 and 2053261.

\subsection*{Acknowledgements} We thank Daniel Allcock, H\"ulya Arg\"uz, David Ben-Zvi, Lorenzo Fantini, Tom Gannon, Gary Kennedy, Chris Manon, Diane Maclagan, Steffen Marcus, Evan Nash, Sam Payne, Dhruv Ranganathan, Alberto San Miguel Malaney, Matt Satriano, Michael Temkin, Jenia Tevelev, Amaury Thuillier, Tassos Vogiannou, Annette Werner, Dmitry Zakharov, and Paul Ziegler for useful conversations en route to this article.

\bigskip

\part{Berkovich spaces and extended affine buildings}
\bigskip

\section{Berkovich analytic geometry}\label{Sec_BerkovichGeometry}

\subsection{Analytification}

Let $X$ be a scheme that is locally finite type over $k$.  Then the Berkovich analytification $X^{\an}$ is a topological space, which can be equipped with a sheaf of analytic functions and then called a $k$-analytic space, which allows us to study the global non-Archimedean analytic geometry of $X$ (see \cite{Berkovich_book, BerkEtale} for a detailed exposition). In the remainder of this article, we will be primarily concerned with the underlying set and topology of these spaces. Notice that the valuation on $k$ induces a \emph{non-Archmiedean norm} on $k$ given by $\nor{\gamma}_k= e^{-\val_k(\gamma)}$ if $\val_k$ is the valuation on $k$.
Let $A$ be a any $k$-algebra then we say a \textit{(multiplicative) seminorm} on $A$ is a map 
 \begin{equation*}
\gnor \colon A\rightarrow{\R} 
 \end{equation*}
 subject to the following axioms:
 
 \begin{enumerate}[\ \ \ \ (a)]
 \item $\nor{\gamma}=\nor{\gamma}_k$ for any $\gamma \in k$ 
 \item $\nor{f}\geq  0$ for all $f\in A$;
 \item $\nor{f+g}\leq \max\{ \nor{f},\nor{g}\}$ for all $f,g\in A$. 
 \item $\nor{fg}=\nor{f}\cdot \nor{g}$ for all $f,g \in A$.
 \end{enumerate}

 \begin{remark}
    If $K$ is any valued field then there is a norm on $K$ given by $e^{-\val_K(\cdot)}$ and conversely if $K$ has a non-Archimedean norm then there is a valuation on $K$ given by $-\log(\gnor_K)$. In this sense non-Archimedean norms and valuations are `equivalent data'. In particular, because any normed extension of a non-Archimedean normed field is automatically non-Archimedean, normed extensions of the ground field $k$ are equivalent to valued extensions of $k$.
 \end{remark}
 
We say that $\gnor$ is \textit{norm} if $\nor{f}=0$ already implies $f=0$.
We may define the set of points of $X^{\an}$ as
\[
X^{\an}=\big\{(x,\gnor) \mid \textrm{ $\gnor$ is a norm on the residue field $k(x)$.}\big\}.
\]
Consider the structure morphism $\rho:X^{\an}\rightarrow X$ given by $(x,\gnor)\mapsto x$. The topology on $X^{\an}$ is the coarest such that $\rho$ is continuous, and for each $U\subseteq X$ a Zariski open and $f\in \calO_X(U)$ the \emph{evaluation map} $\textrm{ev}_f:\rho^{-1}(U)\rightarrow \R$ given by
\[
\gnor\longmapsto \nor{f(x)}
\]
is continuous. In particular, if $X=\Spec A$ then $X^{\an}$ is the set of all seminorms on $A$ and the topology is the coarsest such that for each $f\in A$ the evaluation map $\gnor\mapsto \nor{f}$ is continuous.  The Berkovich analytic space $X^{\an}$ is a locally path connected and locally compact topological space that is Hausdorff if and only if $X$ is separated, and compact if and only if $X$ is proper over $k$ (see \cite[Theorem 3.5.3]{Berkovich_book}).

Given a morphism $f:X\rightarrow Y $ of schemes that are locally of finite type over $k$, there is an induced morphism of $k$-analytic spaces $f^{\an}:X^{\an}\rightarrow Y^{\an}$. On the level of topological spaces it is given by $(x,\gnor)\mapsto (f(x), f_*\gnor)$ where $f_*\nor{h}=\nor{f_x(h)}$ for $h\in k(f(x))$ and $f_x$ is the morphism of reside fields $k(f(x))\rightarrow k(x)$. 

Given any valued extension $K/k$ there is a map $X(K)\rightarrow X^{\an}$ defined by $(\Spec K \rightarrow X)\mapsto (x,\gnor_K)$ where $x$ is the image of the unique point in $\Spec K$ and $\gnor_K$ is the restriction of the norm on $K$ to the residue field $k(x)$. The norm on $K$ induces a topology on $X(K)$, and the map $X(K)\rightarrow X^{\an}$ is continuous with respect to this topology. We can express the topological space of $X^{\an}$ as
\[
X^{\an}=\bigsqcup_{K/k} X(K)/ \sim
\]
where $\sim$ is the equivalence relation defined by $p\sim q$ if $p$ and $q$ define the same pair in $X^{\an}$.

\subsection{The $(.)^\beth$ analytification}\label{Subsec_} Recall that a scheme $X$ that is locally of finite type over $k$ is called \emph{quasi-paracompact} if $X$ admits a locally finite open covering by open affine schemes that are of finite type over $k$. In \cite{Thuillier_toroidal} another type of $k$-analytic space is associated to a quasi-paracompact locally finite type scheme $X$, which is denoted $X^{\beth}$. If $X=\Spec A$ then as a set it is given by
\[
X^{\beth}:=\big\{ \gnor \in X^{\an}\mid  \nor{f}\leq 1 \textrm{ for all $f\in A$} \big\}
\]
and its topology is the coarsest such that the evaluation maps are continuous, i.e. the topology of $X^{\beth}$ is the subspace topology given by viewing $X^{\beth}$ as a subset of $X^{\an}$. One may construct $X^{\beth}$ for an arbitrary locally finite type quasi-paracompact scheme by taking an affine open cover $\{X_i\}_{i\in I}$ and gluing the spaces $X_i^{\beth}$ along the the closed subdomains $(X_i\cap X_j)^{\beth}$ using \cite[Proposition 1.3.3]{BerkEtale}. Similar to above, any morphism $f:X\rightarrow Y$ of quasi-paracompact schemes locally of finte type over $k$ induces a map of analytic spaces $f^{\beth}:X^{\beth}\rightarrow Y^{\beth}$.

By \cite[Section 1]{Thuillier_toroidal} the following properties hold:
\begin{itemize}
    \item If $X$ is finite type, then $X^{\beth}$ is compact.
    \item If $X$ is separable then $X^{\beth}$ the natural morphism $X^\beth\rightarrow X^{\an}$ identifies $X^\beth$ with an analytic subdomain of $X^{\an}$.
    \item If $X$ is of finite type and separated, then we have that $X^{\beth}=X^{\an}$ if an only if $X$ is proper over $k$.
\end{itemize}
Furthermore, given any valued extension $K/k$, the intersection $X(K)\cap X^{\beth}$ is given by $X(K^{\circ})$ where $K^{\circ}$ is the valuation ring of $K$, i.e. those elements $\lambda \in K$ with $\nor{\lambda}_K\leq 1$. We can similarly write:
\[
X^{\beth}=\bigsqcup_{K/k} X(K^{\circ})/ \sim.
\]

\begin{remark}\label{remark_bethvsquasiparacompact}
In the remainder of this article we will also want to apply Thuillier's $\beth$-functor to $\calX_G(\Delta)$, which, in general, is a scheme that is locally of finite type over $k$ but may not be quasi-paracompact. 

Let $X$ be a possibly not quasi-paracompact scheme that is locally of finite type over $k$. Choose a cover by open affine subschemes $U_i$ of finite type over $k$. Then the $U_i^\beth$ are affinoid spaces and we may glue them in the category of rigid analytic spaces giving rise to a rigid analytic space $X^{\rig}$, but not in the category of Berkovich analytic spaces (see \cite[Proposition 1.3.3]{BerkEtale} for a discussion on gluing of Berkovich analytic spaces). We may nevertheless associate a topological space of "Berkovich points" to $X^{\rig}$ by gluing all $U_i^\beth$ over their intersections. 

In what follows, we are only interested in the topological properties of the space of Berkovich points of $X^{\rig}$. In order to avoid unnecessary and unenlightening complications, we choose to ignore this technical peculiarity and simply denote this topological space by $X^\beth$ as well.
\end{remark}

\subsection{The Berkovich spectrum} The last type of $k$-analytic space we will encounter is the \textit{Berkovich spectrum} of a Banach $k$-algebra $A$ (i.e. a $k$-algebra with a norm that is complete with respect to that norm). Define the Berkovich spectrum $\calM(A)$ of $A$ to be the set of seminorms on $A$ bounded by the norm on $A$. Endow $\calM(A)$ with the coarsest topology such that evaluating at an element of $A$ is always continuous. As a topological space $\calM(A)$ is a compact Hausdorff space \cite[Theorem 1.2.1]{Berkovich_book}. This construction is an analogue of $\Spec$ in algebraic geometry. Note that, if $A$ is a finitely generated $k$-algebra equipped with the trivial valuation, then $\calM(A)=(\Spec A)^{\beth}$. Let $p=(x,\gnor)\in X^{\an}$ and let $\calH(p)$ be the completion of $k(X)$ with respect to $\gnor$; this is referred to as \textit{the complete residue field at $p$}. The space $\calM(\calH(p))$ is a one point space and there is a canonical inclusion $\calM\big(\calH(p)\big)\hookrightarrow X^{\an}$. 

\subsection{Berkovich's $\ast$-multiplication}\label{subsubsec_starmult}

One construction that we will repeatedly make use of is $\ast$-multiplicative in the sense of \cite[Chapter 5]{Berkovich_book}. If $G$ is an algebraic group then $G^{\an}$ and $G^{\beth}$ are $k$-analytic groups. If $G$ acts on a variety $X$ then $G^{\an}$ acts on $X^{\an}$ (and similarly $G^{\beth}$ acts on $X^{\an}$ and $X^{\beth}$). Let $\bbG$ be a $k$-analytic group. The underlying set of points of $\bbG$ does not in general form a group (although $\bbG(k)$ is a group). If $\bbG$ acts on a $k$-analytic space $\bbX$ (such as $G^{\an}$ acting on $X^{\an}$ when $G$ acts on $X$) then $\ast$-mutliplication extends the operation on $\bbX$ of left multiplication by an element of $\bbG(k)$ to any element of $\bbG$. We refer the reader to \cite[Chapter 5]{Berkovich_book}, and \cite[Section 3]{Coles} for more details.

\begin{remark}
    Let $g\in \bbG$ and $x\in \bbX$, and denote the multiplication map by $m\colon \bbG\times \bbX \rightarrow \bbX$.  When we write $gx$ we mean the set 
    \begin{equation*}
    \big\{p \in \bbX \mid p=m(y) \textrm{ where $y\in \bbG\times \bbX$ and $p_1(y)=g$ and $p_2(y)=x$}\big\}.
    \end{equation*}
    When $g\in \bbG(k)$ the set $gx$ is a point, but this is not  a priori true for general $g$. When we say a `point of the form' $gx$ we mean a point $p\in gx$.
\end{remark}

Let $A$ and $B$ be Banach rings over $k$ then we can define a  norm on the tensor product $A\otimes_k B$ by
\[
\nor{f}_{A\otimes_k B}:=\inf\big\{\max\nor{a_i}\nor{b_i} \mid f=\sum\limits_i a_i\otimes b_i \big\}
\]
for $f\in A\otimes B$. We must remark that without assumptions on $k$ the seminorm $\nor{f}_{A\otimes_k B}$ is \textit{not} in general multiplicative, i.e. it is not true that, given any $f,g\in A\otimes_k B$, we have $\nor{fg}_{A\otimes_k B}=\nor{f}_{A\otimes_k B}\nor{g}_{A\otimes_k B}$. It is, however, submutliplicative, i.e. we have $\nor{fg}_{A\otimes_k B}\leq \nor{f}_{A\otimes_k B}\nor{g}_{A\otimes_k B}$ for all $f,g\in A\otimes_k B$. Define $A\widehat{\otimes}_k B$ to be the completion of $A\otimes B$ with respect to $\gnor_{A\otimes_k B}$. In the category of $k$-analytic spaces the product $\calM(A)\times\calM(B)$ is given by $\calM(A\hat{\otimes}_k B)$. If $K/k$ is a valued field extension and we have that for any other valued field extension $L/k$ the norm $\gnor_{K\otimes L}$ is multiplicative then we say $K$ is \emph{peaked}. If $p$ is a point of a Berkovich space and $\calH(p)$ is peaked then we say $p$ is \emph{peaked}.
It is shown in \cite[Corollaire 3.14]{Poineau} that, for any point in $p=(x,\gnor)\in X^{\an}$ and any normed field extension $K/k$, the norm $\gnor_{\calH(p)\otimes_k K}$ is multiplicative, i.e. all points are peaked, and therefore the norm on $\calH(p)\hat{\otimes}_k K$ defines a point in $\calM(\calH(p))\times \calM(K)$. Let $G$ be algebraic group acting on $X$, let $g\in G^{\an}$, and let $p\in X^{\an}$. Then there is a point in $\omega_{g,p}\in \calM\big(\calH(g)\big)\times \calM\big(\calH(p)\big)$ corresponding to the norm on $\calH(p)\hat{\otimes}_k\calH(p)$. Let $g\ast p$ be the image of $\omega_{g,p}$ under the composition:
\[
\calM\big(\calH(g)\big)\times \calM\big(\calH(p)\big) \hooklongrightarrow G^{\an}\times X^{\an}\rightarrow X^{\an}.
\]
\begin{examples}
  Here are several examples of $\ast$-multiplication:
  \begin{itemize}
      \item Let $g\in G(k)$, then for $p\in X$ we have $g\ast p=gp$ and in particular $e\ast p=p$.
      \item Let $G$ be a linear algebraic group, and let $g_0$ be the point of $G^{\an}$ corresponding to the trivial norm on $k[G]$. Then for any $h\in G^{\beth}$ we have $g_0\ast h= g_0$.
      \item Let $T$ be an $n$-dimensional torus and let $U_{\sigma}$ be the affine toric variety with big torus $T$ defined by the cone $\sigma$. Let $\bft$ be the point in $T^{\an}$ given by the trivial norm on $k[T]$. Then for any $p\in U_{\sigma}^{\an}$ the $\ast$-product $\bft\ast p$ is the norm given by:
      \[
      \Big\vert \!\Big\vert\sum\limits_{m\in \sigma^{\vee}}c_{m} \chi^m \Big\vert\!\Big\vert_{t\ast p}=\max_{m \in \sigma^{\vee}}\Big(\nor{c_m}\cdot \nor{\chi^m}_p\Big)
      \].
  \end{itemize}
\end{examples}

\begin{proposition}\label{prop_StarMultBySBPOfDifGroups}
    Let $\bbG$ and $\bbH$ be $k$-analytic groups acting on the $k$-analytic spaces $\bbX$ and $\bbY$, respectively. Let $\phi:\bbG\rightarrow \bbH$ be a morphism of $k$-analytic groups and let $\psi: \bbX\rightarrow \bbY$ be a morphism of $k$-analytic spaces. Then we have the following.
     \begin{enumerate}[(1)]
 \item For any peaked point $g\in \bbG$, the map $\bbX\rightarrow \bbX$ given by $p\mapsto g\ast p$ is continuous.
 \item Given $g\in \bbG$ and $h\in \bbH$, both peaked, and $p\in \bbX$ we have that $(g\ast h)\ast p= g\ast (h\ast p)$, i.e. $\ast$-multiplication is associative.
 \item Given a commutative diagram\begin{equation*}\begin{tikzcd}
\bbG\times \bbX\arrow[d,"\phi\times \psi"'] \arrow[rr,] & & \bbX\arrow[d, "\psi"]\\
\bbH\times \bbY \arrow[rr,] & &\bbY
\end{tikzcd}\end{equation*}
we have that $\psi(g)$ is peaked and $\psi(g\ast p)=\phi(g)\ast \psi(p)$.
 \item If $\bbG$ is commutative then for any $g,h\in \bbG$, one of which is peaked, we have $g\ast h=h\ast g$.
 \end{enumerate}
\end{proposition}

In particular, all of the above hold when $\bbG$, $\bbH$, $\bbX$, and $\bbY$ are given by the Berkovich analytification or $\beth$-space associated to a variety.

\begin{proof}
    The first three items are the content of \cite[Proposition 5.2.8]{Berkovich_book}. For the last point assume $\bbG$ is commutative and let $g,h\in \bbG$. By assumption we have a commutative diagram:
    \begin{equation*}
\begin{tikzcd}
\calM\big(\calH(g)\big)\times\calM\big(\calH(h)\big) \arrow[r,hookrightarrow] &  \bbG\times \bbG\arrow[d,"\swap"] \arrow[rr,] & & \bbG\arrow[d, "\textrm{id}"]\\
\calM\big(\calH(h)\big)\times\calM\big(\calH(g)\big) \arrow[r,hookrightarrow] & \bbG\times \bbG \arrow[rr,] & & \bbG
\end{tikzcd}
\end{equation*}
    where $\swap$ is the transposition morphism $(g,h)\mapsto (h,g)$. We then see that the images of $\omega_{g,h}$ and $\omega_{h,g}$ are equal in $\bbG$.
\end{proof}

One particular instance of $\ast$-multiplication that is of particular interest is the following. Let $G$ be a connected linear algebraic group, acting on a variety $X$. Let $\bfg\in G^{\an}$ be the point corresponding to the trivial norm on the coordinate ring $k[G]$; this point is the Shilov boundary of $G^{\beth}$ in the sense of \cite[Section 2.4]{Berkovich_book}. Let $m_{\bfg}:X^{\an}\rightarrow X^{\an}$ be given by $p\mapsto \bfg\ast p$. Maps of this form will appear when constructing embeddings of the building into $G^{\an}$ and discussing connections with tropicalization.

\begin{proposition}\label{prop_PropertiesOfStarMultBySBP}
 Let $p=(x,\gnor)\in X^{\an}$. 
 \begin{enumerate}[(1)]
 \item\label{Item1_prop_properties} We have $m_{\bfg}(p)=\big(\eta_{Gx},\overline{\gnor}\big)$ where $\eta_{Gx}$ is the generic point of the $G$-orbit of $x$ and $\overline{\gnor}$ is a $G(k)$-invariant norm on the function field $k(Gx)$.
 \item\label{Item2_prop_properties} The image of $m_{\bfg}$ is all such points $\big(\eta_{Gx},\overline{\gnor}\big)$ where $\overline{\gnor}$ is a $G(k)$-invariant norm.
 \item\label{Item4_prop_properties} Multiplication by $\bfg$ is idempotent, i.e. $m_{\bfg}$ is a retraction of topological spaces.
 \item\label{Item5_prop_properties} If $H$ and $L$ are closed subgroups of $G$ such that $HL=G$ and $\bfh$ and $\bfl$ are the points of $G^{\an}$ given by the trivial norm on $k[H]$ and $k[L]$ then $\bfh\ast\bfl =\bfg$.
 \end{enumerate}
\end{proposition}

\begin{proof}
    For the first three items we refer the reader to \cite[Section 3]{Coles}. For the last result observe that norm on $\calH(\bfh)\hat{\otimes}_k \calH(\bfl)$ is determined by its values on $k[H]\otimes_kk[L]$ and for $f\in k[H]\otimes_kk[L]$ we have that
   \[
\nor{f}_{\calH(\bfh)\otimes_k \calH(\bfl)}:=\inf\Big\{\max\big\{\nor{h_i}\cdot\nor{l_i}\big\} \mid f=\sum\limits_i h_i\otimes l_i \Big\}.
    \]
    Since $\bfh$ and $\bfl$ are the trivial norm on their respective subrings, the norm $\gnor_{\calH(\bfh)\otimes_k \calH(\bfl)}$ is the trivial norm. Because $HL=G$, the $\bfh\ast \bfl$ is determined by its restriction from $k[H]\otimes_k k[L]$ to $k[G]$, and thus $\bfh\ast \bfl =\bfg$.
\end{proof}

We would like to make the connection between the map $m_{\bfg}$ and actions of the group $G^{\beth}$. The point $\bfg$ acts as a generic point of $G^{\beth}$ in the following sense. If we define $\{\bfg\}^h$ to be the collection of norms on $k[G]$ bounded by $\bfg$ then $\{\bfg\}^h=G^{\beth}$, and $\bfg$ can be recovered as the unique maximal point of $G^{\beth}$\footnote{We call the space $\{\bfg\}^h$ the holomorphically convex envelope of $\bfg$, see \cite[Section 2.6]{Berkovich_book} for more details.}. When an $k$-analytic group $\bbG$ acts on a $k$-analytic space $\bbX$ we can define the orbit of a point $x\in\bbX$ to be the image in $\bbX$ of all points $p\in \bbG x$ such that the projection of $p$ onto $\bbX$ is $x$. The orbits of $\bbG$ partition the underlying topological space of $\bbX$ and thus we can define a topological quotient $\bbG \backslash\bbX$ (see \cite[Section 5.1]{Berkovich_book} for more details).

\begin{proposition}\label{prop_StarMultGivesQuotient}
    Let $G$ act on an algebraic variety $X$, and let $\bbX$ be of the form $X^{\an}$ or $X^{\beth}$. Then the image of $m_{\bfg}$ is homeomorphic to $G^{\beth}\backslash \bbX$  via $\bfg\ast x\mapsto G^{\beth}x$.
\end{proposition}

\begin{proof}
This amounts to showing that for $x,y\in \bbX$ we have $G^{\beth}x=G^{\beth}y$ if and only if $\bfg \ast x= \bfg \ast y$. The case when $\bbX=X^{\beth}$ follows from the case when $\bbX=X^{\an}$.
Notice that $\bfg \ast x$ is in $G^{\beth}x$. So, if $\bfg\ast x = \bfg \ast y$, then $G^{\beth}x=G^{\beth}y$.
On the other hand, assume $G^{\beth}x=G^{\beth}y$. First we claim that the generic points in $X$ of $G\rho(x)$ and $G\rho(y)$ are equal, this follows because we have a commutative diagram
\[
\begin{tikzcd}
G^{\beth}\times X^{\an} \arrow{d}{m^{\an}} \arrow{r}{\rho} & G\times X \arrow{d}{m}\\
X^{\an} \arrow{r}{\rho} & X.
\end{tikzcd}
\]
Thus we have $\rho(\bfg\ast y)=\eta_{G\rho(y)}=\eta_{G\rho(x)}=\rho(\bfg\ast x)$. Now let $U$ be any affine open subset in $X$ containing this scheme-theoretic point. Let $y\in hx$ where $h\in G^{\beth}$ and let $f\in \calO(U)$. We then compute that
\[
\nor{f}_{\bfg\ast y}=\nor{m^*f}_{\omega_{\bfg,y}}\leq \nor{(1\times m)^*\circ m^* (f)}_{\omega_{\bfg,h,x}}=\nor{f}_{\bfg \ast h \ast x}= \nor{f}_{\bfg\ast x}.
\]
So for any $f$ we have that $\nor{f}_{\bfg \ast y}\leq \nor{f}_{\bfg \ast x}$. But the above computation is symmetric in $x$ and $y$. So we have that  $\nor{f}_{\bfg \ast y}\geq \nor{f}_{\bfg \ast x}$ and thus $\bfg \ast x=\bfg \ast y$.
\end{proof}

\begin{remark}
    In \cite[Chapter 5]{Berkovich_book} the quotient space $G^{\beth}\backslash \bbX$ is only given as a topological space though it can also be given the structure of an nonarchimedean analytic stack in the sense of \cite{Ulirsch_StackQuotient}. The quotient of a reduced affinoid group by a closed reduced affinoid subgroup is in fact a $k$-analytic space (see \cite[Theorem 6.7]{Bosch_Quotients} for a precise statement).
\end{remark}

Lastly, we remark there is a Cartan decomposition for $G^{\an}$. This can be found in \cite[Section 5.3]{Berkovich_book} where it is presented as a corollary of \cite[Theorem 7.3.4]{BT1}, but we include a proof here for completeness.

\begin{proposition}\label{Prop_CartanDecomp}
    Let $G$ be a reductive group over $k$, where $k$ is trivially valued and algebraically closed. Let $T$ be a maximal torus of $G$, then we have that:
    \[
    G^{\an}=G^{\beth}T^{\an}G^{\beth}
    \] and the quotient $G^{\beth}\backslash G^{\an} / G^{\beth}$ can be identified with a Weyl chamber in $T^{\trop}$.
\end{proposition}
\begin{proof}
    Let $K_1=k(\!(t)\!)$ be the field of formal Laurent series equipped with the $t$-adic valuation. Let $K_n=k(\!(t^{1/n})\!)$ be the degree $n$ field extension of $K_1$ whose valuation is the unique one extending the valuation on $K_1$. Let $R_n=k[[t^{1/n}]]$ be the valuation ring of $K_n$. We have the Cartan decompositon for $G(K_1)$, that is $G(K_1)=G(R_1)T(K_1)G(R_1)$, and because for each $n$ $K_n$ is isomorphic to $K_1$ as a valued field we have the Cartan decomposition for $G(K_n)$ for any $n$. If $K$ is the field of Puiseux series in $t$, then $K$ is increasing union of the fields $K_n$ and therefore any $g\in G(K)$ is contained in $G(K_n)$ for some $n$. So we have that $G(K)=G(R)T(K)G(R)$ (where $R$ is the valuation ring of $K$). But the rational points over the field of Puiseux series are dense in $G^{\an}$ as a corollary of \cite[Proposition 2.1.15]{Berkovich_book}. Because $G^{\beth}T^{\an}G^{\beth}\supseteq G(R)T(K)G(R)$ and $G(R)T(K)G(R)$ is dense in $G^{\an}$ we have that $G^{\beth}T^{\an}G^{\beth}$ is dense in $G^{\an}$.  But, by \cite[Prosition 5.1.1]{Berkovich_book}, we have that $G^{\beth}T^{\an}G^{\beth}$ is closed and so $G^{\an}=G^{\beth}T^{\an}G^{\beth}$. Similarly, the quotient can be identified with a Weyl chamber in $T^{\trop}$ because for each $n$ the quotient $G(R_n)\backslash T(K_n) /G(R_n)$ is identified with points in a Weyl chamber in $N(T)\otimes_{\Z} \Z[1/n]$ (given a choice of Weyl chamber/Borel) and these define points in $T^{\trop}$.
\end{proof}


\section{Reductive groups and extended buildings}\label{sec_buildings}

\subsection{Structure and combinatorics of reductive groups}\label{subsec_RedGroups}

Let $G$ be a connected reductive algebraic group over $k$. Much of the structure of $G$ can be described using the notion of a root datum. We will follow the construction of a root datum associated to a reductive group as in \cite[Reductive groups]{Springer}. For the combinatorics of root systems/root data we refer the reader to \cite{Humphreys} and for a comprehensive study of reductive groups that covers this material to \cite[Expos\'{e} XXII]{SGA3} and \cite{Springer_linearalgebraicgroups}.

A \textit{root datum} is a tuple $(M,\Phi,N,\Phi^*)$, where $M$ is a finite rank free abelian group and $N$ is the dual group, $\Phi\subseteq M$ is a finite subset called the \textit{roots}, $\Phi^*\subseteq N$ is also a finite subset called the \textit{coroots}, and there is a bijection $\Phi\rightarrow \Phi^*$ given by $\alpha\mapsto \alpha^*$. For each element $\alpha \in \Phi$ there is endomorphism $s_{\alpha}$ of $M$ given by
\[
s_{\alpha}(x)=x-\langle \alpha^*,x  \rangle \alpha
\]
and similarly we have an automorphism of $s_{\alpha^*}$ given by
\[
s_{\alpha^*}(x)=x-\langle x,\alpha  \rangle\alpha^*\ .
\]
Hereby, the following two axioms must be satisfied:
\begin{itemize}
    \item for each $\alpha \in \Phi$ we have that $\langle \alpha,\alpha  \rangle=2$
    \item for each $\alpha$ $s_{\alpha}(\Phi)\subseteq \Phi$ and $s_{\alpha^*}(\Phi^*)\subseteq \Phi^*$.
\end{itemize}
Given a choice of maximal torus $\subseteq G$ we can associate a \textit{root datum} to $G$ as follows. Let $M$ be the character lattice of $T$
and let $N$ the cocharacter lattice\footnote{Often in representation theoretic contexts a torus $T$ has been fixed and one refers to a character of $T$ as a weight of $G$, and a cocharacter as a coweight of $G$.}. Let $\mathfrak{g}$ be the Lie algebra of $G$, then $G$ acts on $\mathfrak{g}$ by conjugation. Then $\mathfrak{g}$ can be written as a direct sum of subrepresentations:
\[
\mathfrak{g}=\mathfrak{t}\oplus \left(\bigoplus_{\alpha \in \Phi} \mathfrak{g_{\alpha}} \right)
\]
where $\mathfrak{t}$ is the Lie algebra of $T$ and each representation $\mathfrak{g}_{\alpha}$ is a 1-dimensional subspace of $\mathfrak{g}$, on which $G$ acts by the character $\alpha\in M$, i.e. for each $x\in \mathfrak{g}_{\alpha}$ and $t\in T$ we have $t\cdot x=\alpha(t)x$. The set $\Phi$ will the be roots of the root datum defined by $G$ and $T$, and the pair $(M_,\Phi)$ is nothing but the root system associated to $G$ and $T$ (or equivalently the root system associated to $\mathfrak{g}$ and $\mathfrak{t}$). The pair $(N,\Phi^*)$ will be a root system associated to the Langland's dual of $G$. Now we will describe the coroots $\Phi^*$. Let $T_{\alpha}$ be the codimension 1 subtorus of $T$ given by connected component of the identity in the kernel of $\alpha$. Let $Z(T_{\alpha})$ be the centralizer of $T_{\alpha}$. The group $Z(T_{\alpha})$ is a connected reductive group with maximal torus $T_{\alpha}$ and the derived subgroup $G_{\alpha}$ of $Z(T_{\alpha})$ is a semisimple group of rank 1. There will be a unique cocharacter of $T$, such that $\textrm{im}(\alpha^*)T_{\alpha}=T$ and $\alpha^*(\alpha)=2$. Then $\Phi^*$ is given by all the cocharacters $\alpha^*$.

The vector space $N_{\R}$ can be decomposed into convex cones called \emph{Weyl chambers} that describe intersection of tori and parabolic subgroups of $G$ that contain $T$. The maps $s_{\alpha*}:N\rightarrow N$ in fact define automorphisms of $N$ and thus the vector space $N_{\R}$. In fact they are reflections about the hyperplane $H_{\alpha^*}=\Span(\alpha^*)^{\perp}.$ The group $W(T)$ generated by these reflections is a finite group called the \textit{Weyl group}.
Let $\textrm{Norm}(T)_G$ be the normalizer of $T$. Conjugation by an element of $\textrm{Norm}(T)_G$ on $T$ induces a linear automorphism of $N_{\R}$, and this induces an isomorphism $\textrm{Norm}_G(T)/T\rightarrow W$.  Let $C_0$ be a connected component of the compliment of the union of these hyperplanes $H_{\alpha^*}$. Then the closure $C$ of $C_0$ is a convex cone called the \textit{Weyl Chamber}. 

We can use the root datum and Weyl chambers to describe parabolic subgroups of $G$.
We say a subset $\Phi_+\subseteq \Phi$ is a set of \textit{positive roots} if the following holds:
\begin{itemize}
    \item given $\alpha \in \Phi$ eaxctly one of $\alpha$ or $-\alpha$ is in $\Phi^+$
    \item if $\alpha,\beta \in \Phi^+$ and $\alpha+\beta\in\Phi$ then $\alpha+\beta\in \Phi^+$.
\end{itemize} 
We say the \textit{negative roots} are $\Phi^-=-\Phi^+=\Phi\setminus \Phi^+$. Take any Weyl chamber, then set of roots $\alpha$ such that for any $u \in C$ $\langle u,\alpha \rangle \geq 0$ is a system of positive roots.  Equivalently given a system of positive roots $\Phi^+$ the set of $u \in N_{\R}$ such that $\langle u,\alpha \rangle \geq 0 $ for any $\alpha \in \Phi^+$ is a Weyl chamber. In this way we have a correspondence between a choice of positive roots and a choice of Weyl chamber in $N_{\R}$.

We can now describe the parabolic subgroups containing $T$ as follows, these are algebraic subgroups of $G$ that contain a Borel subgroup. 
For each root $\alpha$ there is a unique homomorphism of groups $\iota_{\alpha}:\GG_{a}\rightarrow G_{\alpha}$ such that for any $t\in T(k)$ we have $g\in \GG_{a}$ as well as  $t\iota_{\alpha}(g)t^{-1}=\iota_{\alpha}(\alpha(t)g)$. The map $\iota_{\alpha}$ is an isomorphism onto its image and we denote the image in $G$ by $U_{\alpha}$. So $U_{\alpha}$ is a subgroup of $G$ that is normalized by $T$, in particular $T$ acts on $U_{\alpha}$ by the character $\alpha$. We write $U_{\alpha}=\Spec(k[X_{\alpha}])$ to put coordinates on $U_{\alpha}$. 
Let $\Phi^+$ be a set of positive roots. Then the group generated by $T$ and $U_{\alpha}$ for $\alpha \in \Phi^+$ is a Borel subgroup $B$ of $G$, which contains $T$. The unipotent radical $B_u$ of $B$ is generated by the groups $U_{\alpha}$. The group $B_u$ is isomorphic to affine space as a variety. On the other hand, given a Borel subgroup, the set of roots $\alpha$ such that $U_{\alpha}$ is contained in the unipotent radical of the Borel is a system of positive roots. In this way we have that a choice of Weyl chamber, a choice of positive roots, and a choice of Borel containing $T$, are all equivalent. Given a Weyl chamber $C$ and let $u\in C$. Let $P(u)$ be the group generated by $T$ and all the subgroups $U_{\alpha}$ such that $\langle u,\alpha \rangle \geq 0.$ Then $P(u)$ is a parabolic subgroup of $G$ that contains $B$. Furthermore $P(u)=P(u')$ if and only if $u$ and $u'$ are contained in the relative interior of the same face of $C$. Faces of $C$ then correspond to parabolics containing $B$.

Given a Borel $B$ and corresponding set of positive roots $\Phi^+$ we say the \textit{opposite Borel subgroup}, $B^{\opp}$ is the group generated by $T$ and the groups $U_{\alpha}$ for $\alpha\in \Phi^-$. The group operation gives a map:
\[
B_u^{\opp}\times T \times B_u \rightarrow G
\]
given by $(c,t,b)\mapsto ctb$. This map is an open immersion of algebraic varieties and is called the \textit{big Bruhat cell}.

\begin{example}\label{Example_SLnRootDatum}
    Let $V=k^n$ and let $G=\SL(V)\subseteq \GL(V)$. For a matrix $g\in G$ $g_{ij}$ denotes the coefficient in the $i$th row and $j$th column. Let $T$ be the diagonal torus of $G$ and let $e_i$ be the character given by $e_i(t)=t_{ii}$. The characters of $\GL(V)$ are spanned by $e_i$ and the cocharacters by $e^*_i$ where $e_i^*$ is such that $e_i^*(e_j)=\delta_{ij}$ (the Kronecker delta) so the character and cocharacter lattices of $\GL(V)$ are given by $\Z^n$.
    Then we have $e_1+\cdots e_n=0$ for characters on $\SL(V)$ so $M$ is isomorphic to $\Z^{n-1}$ and can be identified with the quotient of $\Z^n$ by the vector $(1,\ldots 1)$. The roots $\Phi$ are given by the character $\alpha_{ij}=e_i-e_j$ for $i\neq j$ and $M$ is generated by the roots. We have $N$ is isomorphic to$\Z^{n-1}$ and can be identified with the subspace of $\Z^n$ where $e^*_1+\cdots e^*_n=0$. The coroots $\Phi^*$ are given by the cocharacters $\alpha_{ij}^*=e_i^*-e_j^*$. Note that $e_i^*$ is a cocharacter of the maximal torus of $\GL(V)$ but $e^*_i$ does not define a cocharacter on $T$, but $\alpha_{ij}^*$ does. From a tropical perspective, if we identify $N_{\R}=T^{\trop}$ with the image of its section in $T^{\an}\subseteq \SL(V)^{\an}$, then the cocharacter lattice corresponds to points in $t\in T(k[\![ s ]\!])$ of the form $t_{ii}=s^{n_i}$ where $n_i\in \Z$ the cocharacters are $\alpha_{ij}$ are given by the points $t$ such that $t_{ii}=s$, $t_{jj}=s^{-1}$ and $t_{kk}=1$ otherwise. Notice that the lattice $M$ in $M_{\R}$ is smooth but the lattice $N$ in $N_{\R}$ is not smooth.
    The subgroup of upper triangular matrices $B\subseteq G$ is a Borel subgroup containing $T$ and the corresponding set of positive roots is $\alpha_{ij}$ such that $i>j$. The Weyl group is $S_3$ and the Weyl chambers are 6 strictly convex cones. The opposite Borel is given by lower triangular matrices and the big Bruhat cell is given by matrices such that $g_{nn}\neq 0$. When $G$ is $\PGL(V)$ the root datum is the same as above, but $M$ and $N$ are exchanged, because $\PGL(V)$ is the Langland's dual of $\SL(V)$. One can then see the Weyl chamber decomposition of of $N_{\R}$ for $\PGL(V)$ by looking at the Weyl chamber decomposition for the character lattice of $\SL(V)$.

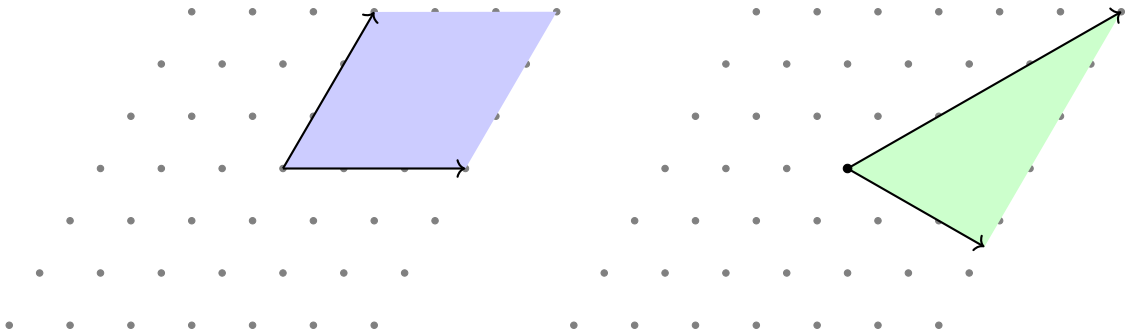
\begin{figure}[h]
    \centering
        \centering
        \begin{tikzpicture}[scale=0.8]
            \foreach \x in {-3,-2,-1,0,1,2,3}
                \foreach \y in {-3,-2,-1,0,1,2,3}
                    \filldraw[gray] (\x + 0.5*\y, {sqrt(3)/2*\y}) circle (1.5pt);


            \fill[blue!20, opacity=0.5] (0,0) -- (3,0) -- (4.5,2.59807621135) -- (1.5,2.595) -- cycle;
           \draw[->, thick] (0,0) -- (3,0);
            \draw[->, thick] (0,0) -- (1.5,2.595);

            

        \end{tikzpicture}
        \centering
        \begin{tikzpicture}[scale=0.8]
            \foreach \x in {-3,-2,-1,0,1,2,3}
                \foreach \y in {-3,-2,-1,0,1,2,3}
                    \filldraw[gray] (\x + 0.5*\y, {sqrt(3)/2*\y}) circle (1.5pt);


            \fill[green!20, opacity=0.5] (0,0) -- (4.5,2.59807621135) -- (2.25,-1.29903810568) -- cycle;


            \filldraw[black] (0,0) circle (2pt);

            \draw[->, thick] (0,0) -- (4.5,2.59807621135);
            \draw[->, thick] (0,0) -- (2.25,-1.29903810568);
            

        \end{tikzpicture}

    \caption{The lattice $N$ and one Weyl chamber for $\PGL(3)$ (on the left) and $\SL(3)$ (on the right). Notice that the Weyl chamber for $\PGL(3)$ is a unimodular cone, i.e. the spanning vectors form a basis for the lattice, but the Weyl chamber for $\SL(3)$ is not unimodular. 
    }
    \label{fig_thefigure}
\end{figure}

\end{example}

\subsection{Bruhat-Tits buildings}\label{subsec_BTbuildings}

In \cite{BT1} and \cite{BT2}, the authors give a detailed account of the structure of the rational points of a reductive groups over a nonarchimedean normed field, and in particular introduce the notion of a building. In our work that follows we will consider the  extended affine Bruhat-Tits building of $G$ over $k$ and here we give a description of this space. Note that our definition of the building differs somewhat from the mainstream approach to Bruhat-Tits buildings, as we will explain below.
We refer the reader to \cite{KalethaPrasad} for a thorough and modern discussion of Bruhat-Tits buildings and to  \cite{RemyThuillierWerner_survey} for a survey on the relationship between Bruhat-Tits buildings and Berkovich analytic geometry.

As in the introduction, we define the extended affine building of $G$ to be the set-theoretic colimit:
\[
G^{\build}:=\varinjlim_{T\le G} T^{\trop}.
\]
For each maximal torus $T$ we refer to the subset $T^{\trop}$ in $G^{\build}$ as an \textit{apartment}. There are two topologies we can give to $G^{\build}$. First, we can topologize it by taking the \emph{weak} topology given by the above colimit where tropical tori have the Euclidean topology. To denote this topological space we will write $G^{\build}_{\weak}$. In this topology a set is open if and only if its intersection with each apartment is open. The colimit topology is much finer that the topology that arises in Berkovich spaces, so we will now introduce another topology, and we will call this space $G^{\build}_{\Berk}$

We define the topology on $G^{\build}_{\Berk}$ by first giving an alternative description of the building.
We can define an equivalence relation on $G(k)\times T^{\trop}$ by $(g,\lambda) \sim (h,{\mu})$ are equal if and only if $g\bft_{\lambda}g^{-1}=h\bft_{\mu}h^{-1}$ in $G^{\an}$. Given another maximal torus $S$ recall that $S$ is conjugate to $T$, so let $S=gTg^{-1}$. Then $gT^{\trop}g^{-1}=S^{\trop}$ in $G^{\an}$ and so any point in $S^{\trop}$ can be written as $g\bft_{\lambda}g^{-1}$. The map $g\bft_{\lambda}g^{-1} \mapsto (g,\bft_{\lambda})$ defines a bijection between $G(k)\times T^{\trop}$ and $G^{\build}$ because all tori are contained in a maximal torus and all maximal tori are conjugate. This map is independent of the choice of $T$. In \cite[Theorem 5.4.2]{Berkovich_book} Berkovich gives an injection $\Theta \colon G^{\build}\rightarrow G^{\an}$ given by $(g,\lambda)\mapsto g\bft_{\lambda} \ast \bfg$ and we define the topology on $G^{\build}_{\Berk}$ such that this map is a homeomorphism onto its image. When $G$ is defined over a local field Berkovich's work was generalized to arbitrary reductive groups and was used to construct compactifications of the building, see \cite{RemyThuillierWernerI}, \cite{RemyThuillierWernerII}, and \cite{RemyThuillierWerner_wonderful}. Given a choice of Borel subgroup $B$, maximal torus $T$, and positive roots we can write any regular function on $G$ as:

\[
\sum\limits_{m,v}c_{m,v}\chi^mX^{v}
\]
where $m$ is an element of the character lattice of $M$ and $v\in \Phi^{\Z_{\geq 0}}$ is a multiindex. Then $g\bft_{\lambda} \ast \bfg$ is given by
\[
\nor{f}_{g\bft_{\lambda} \ast \bfg}=\max_{m,v}\nor{c_{m,v}}e^{-\lambda(m)} \prod_{\alpha \in \Phi^-} e^{-\lambda (\alpha^{v_{\alpha}})} \qquad\textrm{ when } \qquad g\cdot f=\sum\limits_{m,v}c_{m,v}\chi^mX^{v}.
\]
Then the topology on $G^{\build}_{\Berk}$ is such that for any  regular function $f$ on $G$ the map given by $(g,\lambda)\mapsto \nor{f}_{g\bft_{\lambda} \ast \bfg}$ is continuous. For each apartment in $G^{\build}_{\Berk}$ the topology on the apartment agrees with the Euclidian topology on $T^{\trop}$ but this topology is coarser than $G^{\build}_{\weak}$ in general.

We can describe the combinatorics of intersections of apartments in $G^{\build}$ according the following lemma.

\begin{lemma}
    Let $T$ and $S$ be maximal tori of $G$. Then the intersection $T\cap S$ is of the form
    \[
    \bigcap_{\alpha\in I}\ker(\alpha)
    \]
    where $I$ is some subset of the roots $\Phi$ associated to the torus $T$. Furthermore given any $I\subseteq \Phi$ there is some maximal torus $S$ such that the intersection $S \cap T$ is equal to $\cap_{\alpha\in I}\ker(\alpha)$.
\end{lemma}
The following argument is adapted from the one in \cite{SpeyerOverflow}.
\begin{proof} 
    Let $H$ be the subgroup generated by $T$ and $S$, then $H$ is connected as both $S$ and $T$ are. So then $H$ is determined by its Lie algebra as it is connected and $k$ is characteristic 0. Because $H$ contain $T$ we can write:
    \[
    \mathfrak{h}=\mathfrak{t} \oplus \bigoplus_{\alpha \in I} \mathfrak{g}_{\alpha}
    \] for some $I\subseteq \Phi$.  It follows then that the center of $H$, which is equal to $S\cap T$, is the intersection $    \bigcap_{\alpha\in I}\ker(\alpha)$.

    On the other hand let $I\subseteq \Phi$. Let $\overline{I}$ be the intersection of the lattice generated by $I$ and $\Phi$.  Notice then that $  \bigcap_{\alpha\in I}\ker(\alpha)=   \bigcap_{\alpha\in \overline{I}}\ker(\alpha)$. Let $H_I$ be the subgroup of $G$ generated by the groups $G_{\alpha}$ for $\alpha \in \overline{I}$ and $T$. 
    The group $H_{I}$ is reductive. 
    The Lie algebra of $H_{I}$ is 
     \[
    \mathfrak{h}=\mathfrak{t} \oplus \bigoplus_{\alpha \in \overline{I}} \mathfrak{g}_{\alpha}.
    \]
    The union of all subgroups $H_J$, such that $J\subset I$ and $\overline{J}\neq \overline{I}$, is a closed proper subset of $H_I$ so its complement is nonempty and open. Take a regular semisimple element $s$ in this nonempty open, it is then contained in a unique maximal torus $S\subseteq H_{\overline{I}}$ and we claim that $S\cap T= Z(H_{\overline{I}})$, and thus $S\cap T= \bigcap_{\alpha\in I}\ker(\alpha)$.  We have that $S\cap T \supseteq Z(H_{I})$.
    Assume for a contradiction that $t$ is in $S \cap T$ and not in $Z(H_{\overline{I}})$. Let $Z(t)$ be the centralizer of $t$ and let $Z(t)_0$ be the connected component, then this is a connected proper subgroup of $H_I$ which contains $T$, but then its Lie algebra has the form
    \[
    \mathfrak{t} \oplus \bigoplus_{\alpha \in J} \mathfrak{g}_{\alpha}
    \]
    so $Z(t)_0=H_J$ for some $J\subseteq I$. Furthermore $\overline{J}\neq \overline{I}$ because we assumed that $t\notin Z(H_{I})$.
    But since $t\in S$ it follows that $S\subseteq Z(t)_0$, but then $s\in Z(t)_0$ and this contradicts our choice of $s$.
 \end{proof}

Because the kernel of a root $\alpha$ is the coroot hyperplane $H_{\alpha^*}$ it follows that given two apartments $T^{\trop}$ and $S^{\trop}$ the intersection $T^{\trop}\cap S^{\trop}$ is an intersection of coroot hyperplanes in $T^{\trop}$, and any intersection of coroot hyperplanes is the intersection of $T^{\trop}$ with some apartment $S^{\trop}$. 

A homomorphism of (connected reductive) algebraic groups $ G_1\rightarrow G_2$ will map tori to tori, and thus induces continuous maps $G^{\build}_{1,\weak}\rightarrow G^{\build}_{2,\weak}$ and $G^{\build}_{1,\Berk}\rightarrow G_{2,\Berk}^{\build}$.

\begin{example}
Let $G=\textrm{SL}_2$, then $G^{\build}_{\Berk}$ is homeomorphic to the open cone over $\P^1(k)$, i.e.$[0,\infty)\times \P^1(k)/\sim$ where the equivalence relation is defined by identifiying all points in the subset $\{0\}\times \P^1(k)$. The space $G^{\build}_{\weak}$ is given by gluing together a copy of $[0,\infty)$ at $0$ for each element of $P^1(k)$, and a set is open if and only if its intersection with each cone is open.
\end{example}

A more typical approach to defining the extended affine building for a reductive group is to define an equivalence relation $\sim$ on $G(k)\times T^{\trop}$ using certain subgroups of $G(k)$, this is the definition given in \cite{KalethaPrasad}, \cite[Chapter 5]{Berkovich_book}, and \cite{RemyThuillierWernerI}. 
When $k$ has a discrete valuation the space $G(k)\times T^{\trop}/\sim$ is a simplicial complex, which again is a more standard object to study in Bruhat-Tits theory. For trivially valued fields it is more common to consider the spherical building of $G$, rather than the affine or extended buildings. Neither of these notions of building fit our purposes as they do not agree with the skeletons of our toroidal embeddings or Berkovich's realization in $G^{\an}$.
When working over a trivially valued field the space $G^{\build}_{\weak}$ will reflect the combinatorics of Bruhat-Tits theory but also reflect the skeletons of our toroidal embeddings. The topology on $G^{\build}_{\Berk}$ is designed to agree with the realization of $G^{\build}$ in $G^{\an}$. This is why we have introduced the space $G^{\build}$ in the above manner.

\bigskip

\part{Toroidal bordifications and tropicalization}
\bigskip


\section{Toric varieties, toroidal embeddings, and tropicalization}

\subsection{Toric geometry and tropicalization}

In this subsection we will give an overview of toric varieties, toroidal embeddings, and their relevant tropicalization constructions.

\subsubsection{Rational polyhedral cones and affine toric varieties} In this section we give a quick recap of the well-known equivalence between rational polyhedral cones and affine toric varieties, as e.g. detailed in \cite{KKMSD_toroidal}, \cite{Fulton_toric} or \cite{CoxLittleSchenk_toric}.

Let $T$ be an (automatically split) algebraic torus over $k$. Write $M=\Hom(T,\G_m)$ for the character lattice of $T$ and $N=\Hom(\G_m,T)$ for its cocharacter lattice; then $N$ is naturally dual to $M$, i.e. we have $N=\Hom(M,\Z)$. Note that we may write $T=\Spec k[M]$ where $k[M]$ denotes the group algebra generated by $M$. The \emph{tropicalization} $T^{\trop}$ of $T$ is the $\R$-linear space $N_\R=N\otimes_{\Z} \R=\Hom(M,\R)$ and we set $M_\R=M\otimes_{\Z} \R$. 

A \emph{rational polyhedral cone} $\sigma$ in $T^{\trop}=N_\R$ is a strictly convex intersection of finitely many half spaces of the form
\begin{equation*}
H_m=\big\{\lambda\in N_\R \mid \lambda(m)\geq 0\big\}
\end{equation*}
for an element $m\in M$. Write $\sigma^\vee$ for the dual cone of $\sigma$ given by 
\begin{equation*}
\sigma^\vee=\big\{m\in M_\R\mid\lambda(m)\geq 0 \textrm{ for any $\lambda \in \sigma$}\big\} \ .
\end{equation*}
By Gordan's Lemma the dual monoid $S_\sigma=\sigma^\vee \cap M$, which is automatically an integral, saturated, and torsion-free commutative monoid with a zero element, is also finitely generated. A monoid with all of these properties is also called a \emph{toric monoid} and we may recover the date of $M$ and $\sigma$ by setting $M$ to be the Grothendieck group
\begin{equation*}
S_\sigma^{\gp}=\big\{m_1-m_2\mid m_1,m_2\in S_\sigma\big\}
\end{equation*} 
associated to $S_\sigma$, and 
\begin{equation*}\sigma=\Hom(S_\sigma,\R_{\geq 0})\subseteq\Hom(M,\R)=N_\R.\end{equation*}

A rational polyhedral cone naturally gives rise to an \emph{affine toric variety} $U_\sigma$, i.e. an affine $T$-equivariant partial completion of $T$. It is given by $U_\sigma=\Spec k[S_\sigma]$, where $k[S_\sigma]$ denotes the monoid algebra generated by $S_\sigma$ and the natural operation $T\times U_\sigma\rightarrow U_\sigma$ by the torus $T$ is given by the comultiplication
\begin{equation*}\begin{split}
 k[S_\sigma]&\longrightarrow k[M]\otimes_k k[S_\sigma]\\
\chi^m &\longmapsto \chi^m\otimes \chi^m \ . 
\end{split}\end{equation*}

In the following we will refer to a pair $(N,\sigma)$ consisting of a finitely generated free abelian group and a rational polyhedral cone $\sigma$ in $N_\R$ as an \emph{abstract (rational polyhedral) cone}. We call $(N,\sigma)$ \emph{sharp}, if $N_\R$ is generated by the vectors in $\sigma$. A \emph{morphism $(N,\sigma)\rightarrow (N',\sigma')$ of abstract cones} is a homomorphism $f\colon N\rightarrow N'$ such that the induced map $f_\R\colon N_\R\rightarrow N'_\R$ fulfils $f_\R(\sigma)\subseteq \sigma'$. Such a morphism induces a homomorphism $S_{\sigma'}\rightarrow S_{\sigma}$ of toric monoids, which, in turn, induces a homomorphism $k[S_{\sigma'}]\rightarrow k[S_{\sigma}]$ of monoid algebras. The induced morphism $U_\sigma\rightarrow U_{\sigma'}$ is a \emph{toric morphism} of affine toric varieties, i.e. a morphism $U_\sigma\rightarrow U_{\sigma'}$ that restricts to a group homomorphism $T\rightarrow T'$ making $U_\sigma\rightarrow U_{\sigma'}$ equivariant. In fact, there are natural equivalences
\begin{equation*}
\begin{tikzcd}
\mathbf{RPC} \arrow[r,"\sim"] & \mathbf{Mon_{tor}} \arrow[r,"\sim"] & \mathbf{AffTor}
\end{tikzcd}
\end{equation*}
between the category $\mathbf{RPC}$ of abstract rational polyhedral cones, the category $\mathbf{Mon_{tor}}$ of toric monoids, and the category $\mathbf{AffTor}$ of affine toric varieties. 

A face of a rational polyhedral cone $\sigma$ is any subset $\tau$ of the form $\sigma\cap u^{\perp}$ for $u\in \sigma^{\vee}$. Let $\Delta$ be a finite set of rational polyhedral cones in $N_{\R}$. Then we say $\Delta$ is a fan if the following holds:
\begin{itemize}
    \item for any $\sigma \in \Delta$ and $\tau$ a face of $\sigma$ we have that $\tau \in \Delta$
    \item given $\sigma, \sigma' \in \Delta$, the intersection of $\sigma$ and $\sigma'$ is a face of each.
\end{itemize}
Given a rational polyhedral cone $\sigma$ and a face $\tau$ the inclusion map $(N,\tau )\hookrightarrow (N,\sigma)$ induces an open inclusion of toric varieties $U_{\tau}\hookrightarrow U_{\sigma}$. It follows from the above equivalence between affine toric varieties and rational polyhedral cones that there is an order preserving bijection between $T$-invariant open affine subsets of $U_{\sigma}$ and faces of $\sigma$.
Given a fan $(N,\Delta)$ we can define a toric variety $V(\Delta)$ by gluing the toric varieties $U_{\sigma}$ for $\sigma\in \Delta$ along open subvarieties defined by common faces of the cones in $\Delta$. 
We will say an abstract fan is a pair $(N,\Delta)$ where $N$ is a finitely generated abelian group and $\Delta$ is a fan in $N_{\R}$.
A \textit{morphism of abstract fans} $(N,\Delta)\rightarrow (N',\Delta')$ is a homomorphism of groups $f:N\rightarrow N'$ such that for each $\sigma\in\Delta$ there is $\sigma'\in\Delta'$ such that $f$ induces a morphism of abstract cones $(N,\sigma)\rightarrow (N',\sigma')$. Any map of fans $\Delta\rightarrow \Delta'$ will define a group homomorphism of tori $T\rightarrow T'$ and a toric morphism of toric varieties $V(\Delta)\rightarrow V(\Delta')$. The functor taking a fan $\Delta$ to $V(\Delta)$ and morphisms of abstract fans to their associatied morphism toric varieties induces an equivalence of categories between abstract fans and toric varieties. 

Given a fan $\Delta\subseteq N_{\R}$ the $T$-orbits of $X(\Delta)$ are in order reserving bijection with elements of $\Delta$ as follows. Let $\tau\in \Delta$ and define $N(\tau)=N/\Span (\tau)$, if the dimension of $\tau$ is $d$ then $N(\tau)$ corresponds to a torus of dimension $n-d$, $T(\tau)$, which is a quotient of $T$. Let the character lattice of $T(\tau)$ be $M(\tau)$, note that $M(\tau)=\tau^{\perp}$.
Define the star $\textrm{star}(\tau)$ of $\tau$ in $\Delta$ to be the image of the set 
\begin{equation*}
    \big\{\sigma\in \Delta \mid \textrm{$\tau$ is a face of $\sigma$}\big\} 
\end{equation*}
of cones in $\Delta$ containing $\tau$ under the projection map $\textrm{pr}_{\tau}:N\rightarrow N(\tau)$. The star of $\tau$ is fan in $N_{\R}(\tau)$ and it defines a toric variety $V\big(\textrm{star}(\tau)\big)$. The toric variety $V\big(\textrm{star}(\tau)\big)$ is covered by open affines of the form
$U_{\sigma}(\tau)=\Spec \big(k[\textrm{pr}_{\tau}(\sigma)^{\vee}\cap M(\tau)]\big)$
where $\sigma$ is a cone in $\Delta$ containing $\tau$ as a face. The inclusion of monoids $\textrm{pr}_{\tau}(\sigma)^{\vee}\cap M(\tau)\hookrightarrow \sigma^{\vee}\cap M$ that induces a closed embedding $U_{\sigma}(\tau)\hookrightarrow U_{\sigma}$, and these glue to give a toric closed embedding $V\big(\textrm{star}(\tau)\big)\hookrightarrow V(\Delta)$. Under this inclusion the image of the big torus in $V\big(\textrm{star}(\tau)\big)$, which is $T(\tau)$, will be a $T$-orbit of $V(\Delta)$ and this gives an order reversing correspondence between cones in $\Delta$ and $T$-orbits of $V(\Delta)$. 
Finally, note that the orbits of $V(\Delta)$ define a stratification of $V(\Delta)$, and the boundary divisors of $V(\Delta)$ correspond to closures of $(n-1)$-dimensional orbits, which correspond to $1$-dimensional cones in $\Delta$. Note that two boundary divisors intersect if and only if they are contained in some common cone of $\Delta$, and the intersection of a finite number of the divisors is the closure of the orbit that corresponds to the cone the corresponding $1$-dimensional cones collectively generate.

\subsubsection{Tropicalization of a toric variety}\label{subsubsection_TropicalizingTori}
Here we first recall the tropicalization of toric varieties as it appears in \cite{Payne_anallimittrop} (also see \cite{Kajiwara_troptoric}), and we will also introduce the tropicalization map introduced essentially in \cite{Thuillier_toroidal}.
Let $T$ be an $n$-dimensional torus over $k$. The \textit{tropicalization} $T^{\trop}$ of $T$ is defined to be the space $N\otimes_{\Z} \R$, where $N$ is the cocharacter lattice of $T$. There is a tropicalization map:
\[
\trop_T:T^{\an}\longrightarrow T^{\trop}
\]
given by:
\[
p\longmapsto \big(m\mapsto -\log\nor{ \chi^m}_p \big).
\]
Note that if $f:T\rightarrow T'$ is a morphism of algebraic tori then $f$ induces a continuous homomorphism of groups $T^{\trop}\rightarrow T'^{\trop}$. Furthermore there is a continuous section of the tropcalization map $J\colon T^{\trop}\rightarrow T^{\an}$, where $\lambda \mapsto \bft_{\lambda}$ and $\bft_{\lambda}$ is given by:
\[
\nor{\sum\limits_{m\in M} c_m \chi^m}_{\bft_{u}}=\max_{m\in M}\nor{c_m}e^{-\lambda(m)}.
\] Note that the image of the section $T^{\trop}\rightarrow T^{\an}$ is exactly the image of the map $T^{\an}\rightarrow T^{\an}$ given by $p\mapsto \bft_0 \ast p$ where $\bft=\bft_0$ is the point in $T^{\an}$ corresponding to the trivial norm (see Subsubsection \ref{subsubsec_starmult}).

Let $U_{\sigma}$ be an affine toric variety with big torus $T$ given by a rational polyhedral cone $\sigma \subseteq N_{\R}$. Then we define the tropicalization $U_{\sigma}^{\trop}$ of $U_{\sigma}$ to be the topological monoid $\Hom(S_\sigma, \Rbar)$ where the homomorphisms are monoid homomorphisms that preserve scaling by elements of $\R_{\geq 0}$. 
There is a bijection
\[
\bigsqcup_{\textrm{$\tau$ a face of $\sigma$}} N_{\R}(\tau)\xlongrightarrow{\sim} \Hom\big(S_\sigma,\Rbar\big)
\]
given as follows: let $\lambda \in N_{\R}(\tau)$, we have that $N_{\R}(\tau) = \Hom(\tau^{\perp},\R)$, so define $\lambda \mapsto \tilde{\lambda}$ where
\[
\tilde{\lambda}(m)=  \begin{cases} 
          \lambda(m)& \textrm{ if } m\in \tau^{\perp}\cap \sigma^{\vee}\\
          \infty& \textrm{ else}.
       \end{cases}
\] Given this bijection, we will identify $U_{\sigma}^{\trop}=\Hom(S_\sigma,\Rbar)$ and the disjoint union of all the spaces $N(\tau)$.
By writing $U_{\sigma}$ as a union of orbits
\[
U_{\sigma}=\bigsqcup_{\textrm{$\tau$ a face of $\sigma$}} T(\tau)
\]
we have a map
\begin{equation*}\begin{split}
\trop_T: U_{\sigma}^{\an}&\longrightarrow U_{\sigma}^{\trop}\\
p&\longmapsto \big(m\mapsto -\log\nor{\chi^m}_p\big)
\end{split}\end{equation*}
that restricts to the tropicalization maps $T(\tau)\rightarrow T(\tau)^{\trop}=N(\tau)$ on all torus orbits. This is again a continuous map with a continuous section $J\colon U_{\sigma}^{\trop}\rightarrow U_{\sigma}^{\an}$ given by mapping $\lambda\in U_{\sigma}^{\trop}$ to $\bft_\lambda$, where
\[
\nor{\sum\limits_{m\in S_{\sigma}} c_m \chi^m}_{\bft_{\lambda}}=\max_{m\in S_{\sigma}}\nor{c_m}e^{-u(m)}.
\] 
Note that the image of the section  is again exactly the image of the map $U_{\sigma}^{\an}\rightarrow U_{\sigma}^{\an}$ given by $p\mapsto \bft_0 \ast p$. Given a point $p\in U_{\sigma}^{\an}$ the norm $\bft_0\ast p$ is given by:
\[
\nor{\sum\limits_{m\in S_{\sigma}} c_m \chi^m}_{\bft\ast p=}\max_{m\in S_{\sigma}} \nor{c_m}\cdot\nor{\chi^m}_p.
\] 

Given a fan $\Delta$ in $N_{\R}$ and two cones $\sigma,\sigma'$ intersecting in face $\tau$, we find that $U_{\tau}^{\trop}$ will be an open subset of both $U_{\sigma}^{\trop}$ and $U_{\sigma'}^{\trop}$. We can then glue the tropicalizations of $T$-stable open affine subvarieties and the tropicalization to obtain a tropicalization  $V(\Delta)^{\trop}$ of $V(\Delta)$ and tropicalization map $\trop_T:V(\Delta)^{\an}\rightarrow V(\Delta)^{\trop}$. This map will also have a section $V(\Delta)^{\trop}\rightarrow V(\Delta)^{\an}$ whose image is equal to the image of $m_{\bft}$.

We now recall a variant of the tropicalization map presented in \cite[Section 2]{Thuillier_toroidal}. For this we need to assume that $k$ is endowed with the trivial absolute value. Let $\sigma$ be a rational polyhedral cone in $N_{\R}$ and let $U_{\sigma}$ the associated toric variety. We say that $\Hom\big(S_\sigma, \Rbar_{\geq 0}\big)$ is the \textit{canonical compactification of $\sigma$} and denote it by $\overline{\sigma}$. We observe that the tropicalization map $\trop_{U_\sigma}\colon U_\sigma^{\an}\rightarrow U_\sigma^{\trop}$ restricts to a map $U_\sigma^\beth\rightarrow \sigmabar$. 

Let $\Delta$ be a rational polyhedral fan in $N_\R$. Given a cone $\sigma$ in $\Delta$ and a face $\tau$ of $\sigma$, the canonical compactification $\overline{\tau}$ naturally injects into $\overline{\sigma}$ as homomorphisms which are $0$ away from $\tau^{\perp}\cap\sigma^{\vee}$. So it makes sense to define the canonical compactification $\Deltabar$ of $\Delta$ as  the union
\[
\Deltabar=\bigcup_{\sigma \in \Delta}\overline{\sigma} \ ,
\] 
where the polyhedra $\overline{\sigma}$ are glued along shared faces according to how the cones intersect. 

For each face $\tau \subseteq \sigma$ recall that we have an inclusion $N(\tau)\hookrightarrow \overline{\sigma}$ as monoid homomorphisms which are $\infty$ away from $\tau^{\perp}\cap \sigma^{\vee}$. The fan $\textrm{star}(\tau)$ is actually just a cone in $N(\tau)$, and $\overline{\Sigma}_{U_{\sigma}}\cap \rho^{-1}(\eta_{T(\tau)})$ is the cone $\textrm{star}(\tau)$, and $N(\tau)\cap \overline{\Sigma}_{U_{\sigma}}$ is the canonical compactification of ${\textrm{star}(\tau)}$. So  the stratification of the boundary of $U_{\sigma}$ by toric subvarieties corresponds to a stratification of $\overline{\sigma}$ by the cones $\textrm{star}(\tau)$, for $\tau$ a face of $\sigma$. Furthermore this observation about the stratification of the boundary extends from affine toric varieties to any toric variety.

We refer to Figure \ref{Figure_CanonicalCompactificationFan} for a depiction of a canonically compactified fan.

\begin{figure}[!ht]
\centering
\tikzset{every picture/.style={line width=0.75pt}} 

\begin{tikzpicture}[x=0.75pt,y=0.75pt,yscale=-1,xscale=1]

\draw    (105.2,248.6) -- (184.2,248.6) ;
\draw    (67.2,170.6) -- (105.2,248.6) ;
\draw    (105.2,248.6) -- (146.2,170.6) ;
\draw    (250.2,250.6) -- (329.2,250.6) ;
\draw    (212.2,172.6) -- (250.2,250.6) ;
\draw    (250.2,250.6) -- (291.2,172.6) ;
\draw    (251.2,156.6) -- (212.2,172.6) ;
\draw    (251.2,156.6) -- (291.2,172.6) ;
\draw    (326.2,201.6) -- (329.2,250.6) ;
\draw    (326.2,201.6) -- (291.2,172.6) ;
\draw  [draw opacity=0][fill={rgb, 255:red, 155; green, 155; blue, 155 }  ,fill opacity=0.56 ] (251.2,156.6) -- (291.2,172.6) -- (280,193.91) -- (276.08,201.37) -- (250.2,250.6) -- (212.2,172.6) -- (212.2,172.6) -- cycle ;
\draw  [draw opacity=0][fill={rgb, 255:red, 155; green, 155; blue, 155 }  ,fill opacity=0.56 ] (326.2,201.6) -- (329.2,250.6) -- (250.2,250.6) -- (291.2,172.6) -- (291.2,172.6) -- cycle ;
\draw  [draw opacity=0][fill={rgb, 255:red, 155; green, 155; blue, 155 }  ,fill opacity=0.56 ] (146.2,170.6) -- (184.2,248.6) -- (105.2,248.6) -- (67.2,170.6) -- (104.2,170.6) -- cycle ;
\end{tikzpicture}
\caption{An example of a fan $\Delta$ and its  canonical compactification $\overline{\Delta}$.} \label{Figure_CanonicalCompactificationFan}
\end{figure}
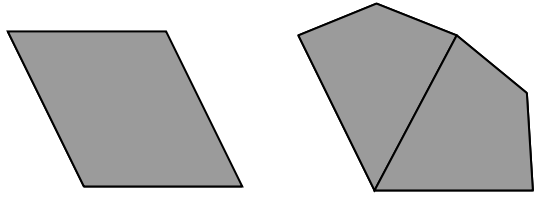

Finally, we observe that the tropicalization map $\trop_{U_\sigma}$ naturally restricts to a map $U_\sigma^\beth\rightarrow \overline{\sigma}$ and, subsequently, that the tropicalization map $\trop_{X(\Delta)}$ naturally restricts to a map $X(\Delta)^\beth\rightarrow \overline{\Delta}$.


\subsection{Toroidal embeddings and their tropicalization}\label{subsection_SkelOfTorEmbed}

Let $X$ be a normal and irreducible scheme that is locally of finite type over $k$ and $X_0\subseteq X$ a non-empty open subset making the embedding $X_0\hookrightarrow X$ into a \emph{toroidal embedding} (see e.g. \cite{KKMSD_toroidal, Thuillier_toroidal} for background on this notion). This means that around every point $p\in X$ there is an \'etale neighborhood $\delta\colon U\rightarrow X$ as well as an \'etale morphism $\gamma\colon U\rightarrow Z$ into a $T$-toric variety $Z$ with big torus $T$ such that $\gamma^{-1}(T)=\delta^{-1}(X_0)$. We refer to the pair $(\delta,\gamma)$ and, in a slight abuse of notation, to $\gamma$ itself as a \emph{toric chart} around $p$.

We will start by considering the special case that around every point the neighborhood $U$ can be chosen to be a Zariski open subset of $X$. In this case the toroidal embedding $X_0\hookrightarrow X$ is said to be \emph{without monodromy}. 

For $k\geq 1$ we inductively define $X_k$ to be the locally closed subset of $X$, whose points are the regular points of $X-(X_0\cup\cdots\cup X_{k-1})$. The connected components of $X_k$ form a stratification of $X$ by connected locally closed subsets, which we refer to as the \emph{toroidal strata} of $X$. We say that a toroidal embedding without monodromy is \emph{small}, if it has a unique closed stratum. Note that, by the orbit-cone correspondence, a toric variety is small if and only if it is affine. 

Let $U$ be a small open subset of $X$ (with the induced structure as a toroidal embedding $U_0:=U\cap X_0\hookrightarrow U$) and suppose $\gamma\colon U\rightarrow Z$ is a toric chart. By \cite[Lemme 3.6]{Thuillier_toroidal}, there is an open affine subset $Z'$ of $Z$ such that $\gamma$ induces a bijection between the strata of $U$ and $Z'$. This bijection preserves the order given by one stratum being in the closure of another (which is the usual order on torus-orbits on toric varieties). A toric chart $\gamma\colon U\rightarrow Z$ into an affine toric variety, where $U$ is small and $\gamma$ induces a bijection of strata is called a \emph{small toric chart}.

\begin{definition}
    A \emph{(rational polyhedral) cone complex} $\Sigma$ is a topological space $\vert\Sigma\vert$ together with a collection of closed subsets $\sigma_\alpha$, each of which is endowed with the structure of a (always sharp) rational polyhedral cone $\sigma_\alpha$ such that the following holds:
    \begin{enumerate}[(i)]
        \item The $\sigma_\alpha$ cover $\vert\Sigma\vert$, i.e. we have $\bigcup_{\alpha}\sigma_\alpha=\vert\Sigma\vert$.
        \item Given a cone $\sigma_\alpha$ in $\Sigma$, every face $\tau$ of $\sigma_\alpha$ is again a cone in $\Sigma$.
        \item The intersection $\sigma_\alpha\cap\sigma_\beta$ of two cones $\sigma_\alpha$ and $\sigma_\beta$ is a finite union of cones in $\Sigma$.
    \end{enumerate}
\end{definition}

A \emph{morphism} $f\colon\Sigma\rightarrow\Sigma'$ of rational polyhedral cone complexes is a continuous map $\vert f\vert\colon \vert\Sigma\vert\rightarrow\vert\Sigma'\vert$ such that for every cone $\sigma_\alpha$ in $\Sigma$ there is a cone $\sigma'_{\alpha'}$ in $\Sigma'$ such that $f(\sigma_\alpha)\subseteq \sigma'_{\alpha'}$ and the restriction of $f$ to $\sigma_\alpha\rightarrow \sigma'_{\alpha'}$ is a morphism of rational polyhedral cones. Then $f$ is said to be a \emph{strict morphism}, if for all $\sigma_\alpha$ the cone $\sigma'_{\alpha'}$ may chosen such that $f$ induces an isomorphism $\sigma_\alpha\xrightarrow{\sim} \sigma'_{\alpha'}$. 

Given a rational polyhedral cone complex $\Sigma$, its \emph{canonical compactification} is the (non-disjoint) union 
\begin{equation*}
    \Sigmabar=\bigcup_{\alpha}\sigmabar_\alpha \ , 
\end{equation*}
where $\sigmabar_\alpha=\Hom(S_\alpha,\Rbar_{\geq 0})$ is the canonical compactification of $\sigma_\alpha$. In this union we naturally identify $\overline{\tau}$ with a subset of $\sigmabar$, whenever $\tau$ is a face of $\sigma$.

To every toroidal embedding $X_0\hookrightarrow X$ we may associated cone complex $\Sigma_X$, whose cones are in an order-reversing one-to-one correspondence with the strata of $X$. To understand this construction we note that in \cite[Section 3]{Thuillier_toroidal} it is shown that for two small toric charts $\gamma_i\colon U_i\rightarrow Z_{\sigma_i}$ (with $i=1,2$), for which the two closed strata of $U_i$ generically agree, we have a natural isomorphism $\sigma_1\simeq \sigma_2$. Therefore there is a unique (sharp) rational polyhedral cone $\sigma_E$ associated to every toroidal stratum $E$ of $X$ and, whenever there is another stratum $E'$, whose closure contains $E$, we have that $\sigma_{E'}$ is a face of $\sigma_E$. The cone complex $\Sigma_X$ is then defined to be the union of the cones $\sigma_E$ taken over all strata $E$ of $X$.

There is a natural and continuous \emph{tropicalization map} $\trop_X\colon X^\beth\rightarrow \Sigmabar_X$ such that for every small toric chart $\gamma\colon U\rightarrow Z_{\sigma}$ the diagram
\begin{equation*}
    \begin{tikzcd}
        X^\beth \arrow[rrr,"\trop_X"]\arrow[rd,"\gamma^\beth"] &&& \Sigmabar_X\\
        & Z^\beth\arrow[rr,"\trop_Z"] && \sigmabar \arrow[u]
    \end{tikzcd}
\end{equation*}
commutes. Furthermore, \cite[Cor. 3.13]{Thuillier_toroidal} (also see \cite[Theorem 1.2]{Ulirsch_functroplogsch}) tells us that $\trop_X$ has a continuous section $J_X\colon \Sigmabar_X\rightarrow X^\beth$ making the composition $\bfp_X=J_X\circ \trop_X$ into a strong deformation retraction onto a closed subset $\frakS{}(X)$, the \emph{toroidal skeleton} of $X^\beth$.

\begin{remark}
We explicitly point out that the above-mentioned construction(s) from \cite{Ulirsch_functroplogsch} and \cite{Thuillier_toroidal} also work  in the case when $X$ is not quasi-paracompact using the workaround described in Remark \ref{remark_bethvsquasiparacompact} above. 
\end{remark}

\begin{remark}
The above story generalizes to general toroidal embeddings, possibly with self-intersection. Since this article does not strictly require these developments, we refer the reader to \cite{Thuillier_toroidal, ACP} for further details as well as to \cite{Ulirsch_functroplogsch} and \cite{Ulirsch_nonArchArtin} to a more general version of this story from the perspective of logarithmic geometry.
\end{remark}

There is an alternative point of view on the tropicalization of toroidal embeddings using \emph{Artin fans} (see \cite{Ulirsch_nonArchArtin} for further details). To a rational polyhedral cone $\sigma$ we associate its \emph{Artin cone} $\calA_\sigma$, which is given as the toric quotient stack
\begin{equation*}
\calA_\sigma=\big[U_\sigma/T\big], 
\end{equation*}
where $T=\Spec k[S_\sigma^{\gp}]$ is the big torus of the affine toric variety $U_\sigma=\Spec k[S_\sigma]$ defined by $\sigma$. Note that the Artin cone $\calA_\sigma$ only depends on the sharp monoid $\overline{S}_\sigma=S_\sigma/S_\sigma^\ast$. Given a face $\tau$ of $\sigma$, the natural open immersion $U_\tau\hookrightarrow U_\sigma$ induces an open immersion $\calA_\tau\hookrightarrow \calA_\sigma$. Let $\Sigma$ be a cone complex. Then we associate to $\Sigma$ the algebraic stack
\begin{equation*}
    \calA_\Sigma=\bigcup_{\sigma\in\Sigma}\calA_\sigma
\end{equation*}
glued according to the incidences provided by $\Sigma$. 

Now let $X_0\hookrightarrow X$ be a toroidal embedding and write $\calA_X=\calA_{\Sigma_X}$. Then there is a natural morphism $X\rightarrow \calA_X$ such that for every small chart $\gamma\colon U\rightarrow Z_\sigma$ the diagram
\begin{equation*}
    \begin{tikzcd}
        X\arrow[rrr] & & & \calA_X\\
        U\arrow[r]\arrow[u,"\subseteq"] &  Z_\sigma\arrow[r,"\quot"]&\big[Z_\sigma/T\big]\arrow[r,equal] &\calA_\sigma \arrow[u,"\subseteq"]
    \end{tikzcd}
\end{equation*}
commutes. This way, the toroidal strata of $X$ are precisely the preimages of the points of $\calA_\Sigma$ and the tropicalization map $\trop_X$ can be recovered by applying the $\beth$-functor to the morphism $X\rightarrow \calA_X$ by \cite[Theorem 1.1]{Ulirsch_nonArchArtin}.


\section{Stacky cone complexes and fans in the building}\label{section_stackyfans}

\subsection{Kummer homomorphisms and root stacks} There are several competing ways to formalize stacky fans as a combinatorial gadget to describe toric stacks. We refer the reader to \cite{GeraschenkoSatrianoI} for an overview of the different approaches. Here we present an approach using Kummer homomorphisms that is closely related to \cite{FantechiMannNeroni} and \cite{GillamMolcho}, and, in fact, a special case of the logarithmic root stack construction in \cite[Section 4]{BorneVistoli}.

Let $T$ be an algebraic torus over $k$ with character lattice $M$ and cocharacter lattice $N$ as above. 

\begin{definition}
Let $S_\sigma$ be the toric monoid associated to a rational polyhedral cone $\sigma$ in $T^{\trop}$. A \emph{Kummer homomorphism} is an injective homomorphism $\phi\colon S_\sigma\rightarrow \widetilde{S}_\sigma$ into a toric monoid $\widetilde{S}_\sigma$ that fulfils the property that for every $\widetilde{m}\in \widetilde{S}_\sigma$ there is an $n\in\N_{>0}$ such that $n\cdot \widetilde{m}=\phi_\sigma(m)$ for some $m\in S_\sigma$. 
\end{definition}

Let $\phi\colon S_\sigma\rightarrow\widetilde{S}_\sigma$ be a Kummer homomorphism. Set $\widetilde{M}=\widetilde{S}_\sigma^{\gp}$ and write $\widetilde{N}=\Hom(M,\Z)$ for its dual. Then $\phi_\sigma$ induces a natural homomorphism $\widetilde{N}\rightarrow N$. The preimage of $\sigma$ under this map is again a rational polyhedral cone $\widetilde{\sigma}$ that maps bijectively onto $\sigma$ and can also be characterized as $\widetilde{\sigma}=\Hom(\widetilde{S}_\sigma,\R_{\geq 0})$; note that in this notation we tautologically have $S_{\widetilde{\sigma}}=\widetilde{S}_\sigma$. 

The homomorphism $\phi_\sigma: S_\sigma\rightarrow \widetilde{S}_\sigma$ induces a toric morphism $U_{\widetilde{\sigma}}\rightarrow U_\sigma$, which, in turn, descends to a representable morphism $\calA_{\widetilde{\sigma}} \longrightarrow \calA_\sigma$
of toric quotient stacks 
\begin{equation*}\calA_{\widetilde{\sigma}}=\calA_{S_{\widetilde{\sigma}}}=\big[U_{\widetilde{\sigma}}\big/\widetilde{T}\big]\qquad \textrm{ and } \qquad \calA_{\sigma}=\calA_{S_\sigma}=\big[U_\sigma\big/T\big]\ .
\end{equation*} 
Here we write $\widetilde{T}$ for the big torus $\Spec \C[\widetilde{M}]$ of $U_{\widetilde{\sigma}}$. 

\begin{definition}
The \emph{affine toric root stack} $\calU_{(\sigma,\phi)}$ associated to a Kummer homomorphism $\phi\colon S_{\sigma}\rightarrow\widetilde{S}_\sigma$ is defined by the $2$-fibered product 
\begin{equation*}\begin{tikzcd}
\calU_{(\sigma,\phi)} \arrow[r]\arrow[d] & \calA_{\widetilde{\sigma}}\arrow[d]\\
U_\sigma \arrow[r] & \calA_\sigma 
\end{tikzcd}\ .\end{equation*}
\end{definition}

The affine toric root stack $\calU_{(\sigma,\phi)}$ naturally fits in between $U_{\widetilde{\sigma}}$ and $U_{\sigma}$ in the sense that there is a factorization 
\begin{equation*}
U_{\widetilde{\sigma}}\longrightarrow \calU_{(\sigma,\phi)}\longrightarrow U_\sigma
\end{equation*} 
where the first arrow is a surjective \'etale morphism and the second arrow makes $U_\sigma$ into a coarse moduli space of $\calU_{(\sigma,\phi)}$. In fact, the kernel $A=\ker (\widetilde{T}\rightarrow T)$ is a finite abelian group given by 
\begin{equation*}
A=\Spec k\big[\coker(M\rightarrow \widetilde{M})\big]
\end{equation*} and we have  
\begin{equation*}
\calU_{(\sigma,\phi)}=\big[U_{\widetilde{\sigma}}\big/A\big] \qquad \textrm{ as well as }\qquad U_\sigma=U_{\widetilde{\sigma}}/A \ . 
\end{equation*}
So there is natural operation of the torus $T=\widetilde{T}/A$ on $\calU_{\sigma, \phi_{\sigma}}$ with a unique dense open orbit that is isomorphic to $T$.

\begin{lemma}\label{lemma_Kummerface}
Let $\sigma$ be rational polyhedral cone in $N_\R$, let $\phi\colon S_\sigma\rightarrow \widetilde{S}_\sigma$ be a Kummer homomorphism and $\tau$ a face of $\sigma$. Then there is a unique Kummer homomorphism $\phi\vert_\tau\colon S_\tau\rightarrow \widetilde{S}_\tau$ that makes the diagram
\begin{equation*}
    \begin{tikzcd}
        \widetilde{S}_\tau & &\arrow[ll] \widetilde{S}_\sigma\\
        S_\tau \arrow[u,"\phi\vert_\tau"]& & \arrow[ll] S_\sigma\arrow[u,"\phi"']
    \end{tikzcd}
\end{equation*}
commute. 
\end{lemma}

\begin{proof}
We define $\widetilde{S}_\tau=\widetilde{M}\cap S_\tau$. Since the inclusion homomorphism $M\rightarrow \widetilde{M}$ is of finite index, the inclusion $\phi\vert_\tau\colon S_\tau\hookrightarrow \widetilde{S}_\tau$ is a Kummer homomorphism. The commutativity of the diagram and the uniqueness of the Kummer homomorphism $\phi\vert_\tau\colon S_\tau\rightarrow \widetilde{S}_\tau$ is then automatic.
\end{proof}

For a rational polyhedral cone $\sigma$ in $N_\R$ we write $S_\sigma^\ast$ for its subgroup of units and $\overline{S}_\sigma$ for the quotient monoid $S_\sigma/S_\sigma^\ast$. The monoid $\overline{S}_\sigma$ is again a toric monoid associated to the cone $\sigma$, now as a rational polyhedral cone in $\Span(\sigma)\subseteq N_\R$. On the level of affine toric varieties the operation $S_\sigma\leadsto \overline{S}_\sigma$ corresponds to removing any torus factors from $U_\sigma$. We write $\overline{U}_\sigma$ for the resulting affine toric varieties with big torus $\overline{T}=\Spec \C[\overline{M}]$ with character lattice $\overline{M}=(\overline{S}_\sigma)^{\gp}$. 

\begin{lemma}
Let $\sigma$ be rational polyhedral cone in $N_\R$ and let $\phi\colon S_\sigma\rightarrow\widetilde{S}_\sigma$ be a Kummer homomorphism into $\widetilde{S}_\sigma=S_{\widetilde{\sigma}}$. The induced homomorphism $\overline{\phi}\colon \overline{S}_\sigma\rightarrow \overline{S}_{\widetilde{\sigma}}$ is a again a Kummer homomorphism and, for any choice of section of $S_\sigma\rightarrow\overline{S}_\sigma$, there is a $2$-cartesian square 
\begin{equation*}\begin{tikzcd}
\calU_{(\sigma,\phi)} \arrow[r]\arrow[d] & \calA_{\overline{S}_{\widetilde{\sigma}}}\arrow[d]\\
U_\sigma \arrow[r] & \calA_{\overline{S}_\sigma} \ .
\end{tikzcd}\end{equation*}
\end{lemma}

So, in order to describe the affine toric root stack $\calU_{(\sigma,\phi)}$ it is enough to give a Kummer homomorphism $\overline{S}_\sigma\rightarrow \overline{S}_{\widetilde{\sigma}}$. 

\begin{proof}
This is a consequence of \cite[Proposition 5.17]{Olsson_loggeo&algstacks}. 
\end{proof}


\subsection{Stacky cone complexes and toroidal embeddings}\label{section_stackycc&toroidal}
\begin{definition}
A \emph{stacky (rational polyhedral) cone complex} is pair $(\Sigma, \Phi)$ consisting of a (rational polyhedral) cone complex $\Sigma$ and family of Kummer homomorphisms $\phi_\sigma\colon  S_\sigma\rightarrow \widetilde{S}_\sigma$, one for each $\sigma$ in $\Sigma$, such that, whenever $\tau$ is a face of $\sigma$, we have $\phi_\sigma\vert_\tau=\phi_\tau$. 
\end{definition}

The datum of a stacky cone complex $(\Sigma,\Phi)$ gives rise to another cone complex $\widetilde{\Sigma}$, whose underlying topological space is equal to $\vert\Sigma\vert$ and in which we replace every cone $\sigma$ by $\widetilde{\sigma}$ as above. This way we obtain a natural morphism $\widetilde{\Sigma}\rightarrow \Sigma$ of cone complexes that is a bijection on the underlying topological spaces. 

Let $X_0\hookrightarrow X$ be a toroidal embedding without self-intersection and denote by $\Sigma_X$ its associated cone complex. Suppose we are also given a collection $\Phi$ of Kummer homomorphisms making $(\Sigma_X,\Phi)$ into a stacky cone complex. 

Denote by $\calA_{\widetilde{\Sigma}}\rightarrow \calA_X$ morphism of Artin fans induced by $\widetilde{\Sigma}\rightarrow\Sigma$. 

\begin{definition}\label{def_stackytoroidalembeddings}
    The \emph{(stacky) toroidal embedding} $\calX(\Phi)$ associated to $(X,\Phi)$ is defined by the $2$-fibered product
        \begin{equation*}\begin{tikzcd}
            \calX(\Phi) \arrow[r]\arrow[d] & \calA_{\widetilde{\Sigma}}\arrow[d]\\
            X \arrow[r] & \calA_\Sigma \ .
        \end{tikzcd}\ .\end{equation*}
\end{definition} 

By \cite[Proposition 3.8]{Ulirsch_StackQuotient}, the underlying topological space of $\calX(\Phi)^\beth$ is equal to $X^\beth$. Hence we have a natural and continuous tropicalization map $\trop_{\calX(\Phi)}\colon \calX(\Phi)^\beth\rightarrow \widetilde{\Sigma}_{\calX(\Phi)}$ that makes the diagram
\begin{equation*}
    \begin{tikzcd}
        \calX(\phi)^\beth \arrow[rr,"\trop_{\calX(\phi)}"]\arrow[d] & & \overline{\widetilde{\Sigma}}\arrow[d]\\
        X^\beth \arrow[rr,"\trop_{X}"] & & \Sigmabar
    \end{tikzcd}
\end{equation*}
commute. Note that this tropicalization only agrees with $\trop_X\colon X^\beth\rightarrow \Sigmabar_X$ on the underlying topological spaces up to rescaling. In particular, there is a section of $\trop_{\calX(\Phi)}$ that identifies $\trop_{\calX(\Phi)}$ with a strong deformation retraction onto a closed subset $\frakS(\calX(\Phi))$ of $\calX(\Phi)^\beth$, the \emph{non-Archimedean skeleton} of $\calX(\Phi)^\beth$. From the perspective of Artin fans, \cite[Theorem 1.1]{Ulirsch_nonArchArtin} implies that we may identify $\trop_{\calX(\Phi)}$ with the underlying continuous map obtained by applying the $\beth$-functor to the morphism $\calX(\Phi)\rightarrow \calA_{\widetilde{\Sigma}}$.

\begin{remark}
    Stacky toroidal embeddings in the sense of Definition \ref{def_stackytoroidalembeddings} are special cases of toroidal embeddings of Deligne--Mumford stacks, as introduced in \cite[Definition 6.1.1]{ACP}, and the retraction/tropicalization map is a special case of the retraction/tropicalization map of these toroidal embeddings. We refer the reader to \cite[Section 6]{ACP} for more details as well as to \cite[Section 6]{Ulirsch_nonArchArtin} for an alternative perspective on this using methods from logarithmic geometry. 
\end{remark}


\subsection{Stacky fans in $G^{\build}$} 
Let $G$ be a reductive algebraic group over $k$. As in Section \ref{subsec_BTbuildings}, we define the \emph{extended affine building $G^{\build}$} as the colimit 
\begin{equation*}
G^{\build}:=\varinjlim_{T\le G} T^{\trop}
\end{equation*}
in the category of sets, taken over all tori $T$ in $G$. As explained in \cite[Section 5.4]{Berkovich_book}, the extended affine building $G^{\build}$ admits a natural injection into the Berkovich analytification $G^{\an}$ of $G$ (see Subsection \ref{subsec_BTbuildings} for details). The topology on $G^{\an}$ is induced from the one on $G^{\an}$. There is another finer topology that we also work with, the topology induced by the colimit, in either case the topolgoy of $G^{\build}$ agrees with the Euclidian topology when restricted to the image of $T^{\trop}$ in $G^{\build}$ for a given torus $T$.  
Since for two tori $T$ and $T'$ in $G$ with $T\le T'$, the induced $\R$-linear homomorphism $T^{\trop}\rightarrow (T')^{\trop}$ is injective, we find that that each $T^{\trop}$ maps homeomorphically onto its image in $G^{\build}$. So we may identify $T^{\trop}$ with a closed subset of $G^{\build
}$, to which we refer as an \emph{apartment} of $G^{\build}$.

\begin{definition}\label{Defintion_ConeInGBuilding}
A \emph{(rational polyhedral) cone} in $G^{\build
}$ is a subset $\sigma\subseteq G^{\build}$ such that there is a torus $T$ in $G$, for which $\sigma\subseteq T^{\trop}$ and $\sigma$ is a rational polyhedral cone in $T^{\trop}$. 
\end{definition}

\begin{definition}\label{Definition_BuildingFan}
A \emph{fan} $\Delta$ in $G^{\build}$ is a collection of cones $\sigma$ in $G^{\build}$ that fulfil: 
\begin{enumerate}[(i)]
\item Given a cone $\sigma\in \Delta$ and a face $\tau$ of $\sigma$, then $\tau\in \Delta$.
\item For two cones $\sigma$ and $\sigma'$ in $\Delta$, the intersection $\sigma\cap\sigma'$ is a face of both $\sigma$ and $\sigma'$.
\end{enumerate}
\end{definition}

When $G$ already is a torus $T$, a fan in $T^{\build}=T^{\trop}$ is nothing but a fan in the usual sense (as in the theory of toric varieties), with the exception that $\Delta$ is usually assumed to be a finite set. This motivates us to call a fan $\Delta$ in $G^{\build}$ \emph{apartment-wise finite} if, for every torus $T$ in $G$, the subfan
\begin{equation*}
\Delta_T=\big\{\sigma\in\Delta\big\vert\sigma \subseteq T^{\trop}\neq\emptyset\big\}
\end{equation*}
is finite.

\begin{definition}\label{def_stackyfan}
A \emph{stacky fan} in $G^{\trop}$ is a fan $\Delta$ together with a collection $\Phi$ of Kummer homomorphism $\phi_\sigma\colon \overline{S}_\sigma\rightarrow \overline{S}_{\widetilde{\sigma}}$ for every cone $\sigma\in \Delta$ such that for every face $\tau$ of $\sigma$ we have 
\begin{equation*}
\phi_\tau=\overline{\phi_\sigma\vert_\tau}
\end{equation*}
for the induced Kummer homomorphism $\phi_{\sigma}\vert_\tau$ from Lemma \ref{lemma_Kummerface}.\end{definition}

One may think of both $\sigma\longmapsto \overline{S}_\sigma$ and $\sigma\longmapsto \overline{S}_{\widetilde{\sigma}}$ as as sheaves $\calS$ and $\widetilde{\calS}$ of (sharp) monoids on $\Delta$ (endowed with the poset topology). Then the condition in Definition \ref{def_stackyfan} says that the collection $\phi_{(.)}$ of Kummer homomorphisms is a morphism $\calS\rightarrow \widetilde{\calS}$ of sheaves of monoids on $\Delta$.

If $(\Delta,\Phi)$ is a fan in $G^{\build}$ then the union of the cones in $\Delta$ forms a cone complex (and in fact a stacky cone complex when equipped with the Kummer homomorphisms $\Phi$). We denote this cone complex $\Sigma(\Delta,\Phi)$.

\section{Toroidal bordifications of reductive groups}\label{section_toroidalbord}

Let $G$ be a reductive group and fix a stacky fan $(\Delta,\Phi)$ in $G^{\build}$. In this section we construct a toroidal bordification $\calX_G(\Delta, \Phi)$ associated to $(\Delta,\Phi)$ expanding on \cite[Section IV.2]{KKMSD_toroidal} and prove its basic properties, as stated in Theorem \ref{mainthm_stackyfan=toroidalbordification}.

Let $B$ be a Borel subgroup of $G$ and $T\subseteq B$ a maximal torus. Consider a rational polyhedral cone $\sigma$ in $T^{\trop}$ and a Kummer homomorphism $\phi_\sigma\colon S_\sigma\rightarrow S_{\widetilde{\sigma}}$. 

\begin{lemma}\label{lemma_Borel}
The pushout 
\begin{equation*}
\calB(\sigma,\phi_\sigma)=\calU_{\sigma,\phi_\sigma} \times^T B
\end{equation*}
 exists as a Deligne-Mumford stack and is a pushout for all subtori $T'\leq T$ such that $\sigma\subseteq (T')^{\trop}$.
\end{lemma}

\begin{proof}
By \cite[Theorem 10.6]{Borel} the Borel subgroup $B$ is a semidirect product
\begin{equation*}
B\cong U\rtimes T
\end{equation*}
for a unipotent subgroup $U$ of $B$. Thus (not taking the group structure into account), we find
\begin{equation*}
\calB(\sigma,\phi_\sigma)\cong U\times  \calU_{\sigma,\phi_\sigma}
\end{equation*}
and this proves the existence of $\calU_{\sigma,\phi_\sigma}\times^{T} B$.

Now let $T'\subseteq T$ be a subtorus such that $\sigma\subseteq (T')^{\trop}$ and write $\calU'_{\sigma,\phi_\sigma}$ for the affine toric root stack associated to $(\sigma,\phi_\sigma)$ with big open torus $T'$. Then we have natural isomorphisms
\begin{equation*}
\calU_{\sigma,\phi_\sigma} \times^T B\cong \calU'_{\sigma,\phi_\sigma}\times^{T'}T\times^T B\cong \calU'_{\sigma,\phi_\sigma} \times^{T'} B \ ,
\end{equation*}
since we may choose a splitting $T\cong T'\oplus T''$ of $T$ and find $\calU_{\sigma,\phi_\sigma}\cong T'' \times \calU'_{\sigma,\phi_\sigma}$, as long as $\sigma\subseteq T'$. But this shows $\calU_{\sigma,\phi_\sigma}\cong \calU'_{\sigma,\phi_\sigma}\times^{T'}T$.
\end{proof}

We now define a partial bordification of $G$ itself by setting
\begin{equation*}\label{eq_BordificationOfG}
\calX_G(\sigma,\phi_\sigma)= \calB(\sigma,\phi_\sigma)\times^B G \ ,
\end{equation*}
which exists, since $G\rightarrow B\backslash G$ is a Zariski-locally trivial principal $B$-bundle over the projective variety $B\backslash G$ and $\calX_G(\sigma,\phi_\sigma)\times^B G$ is the associated fiber bundle with fiber $\calB(\sigma,\phi_\sigma)$. 

Note, however, that we also have natural isomorphisms
\begin{equation*}
\calX_G(\sigma,\phi_\sigma)= \calB(\sigma,\phi_\sigma)\times^B G \cong \big(B\times^T \calU_{\sigma,\phi_\sigma}\big)\times^B G\cong   \calU_{\sigma,\phi_\sigma} \times^T G\ ,
\end{equation*}
where the right-hand side does not depend on the choice of $B$ anymore. So, for a face $\tau$ of $\sigma$, we may choose $\tau$ and $\sigma$ to lie in the same $T^{\trop}$ by Lemma \ref{lemma_Borel} and then we have a natural open immersion 
\begin{equation*}
\calX_G(\tau,\phi_\tau)\hooklongrightarrow\calX_G(\sigma,\phi_\sigma)
\end{equation*}

\begin{definition}\label{Defintion_ToroidalBordification}
The \emph{toroidal bordification} $\calX_G(\Delta,\Phi)$ of $G$ associated to a stacky fan $(\Delta,\Phi)$ is defined to be the colimit
\begin{equation*}
\calX_G(\Delta,\Phi)=\varinjlim_{\sigma\in \Delta} \calX_G(\sigma,\phi_\sigma) \ .
\end{equation*}
\end{definition}

In other words, we glue the $\calX_G(\sigma,\phi_\sigma)$ along open substacks associated to common faces and consider $\calX_G(\Delta,\Phi)$ to be the union 
\begin{equation*}
\calX_G(\Delta,\Phi)=\bigcup_{\sigma \in \Delta }\calX_G(\sigma,\phi_\sigma) \ .
\end{equation*}

Whenever $G$ is itself a torus $T$ and $\Delta$ is finite, this is nothing but the well-known construction of a toric variety associated to a fan (or, to be precise, a toric root stack associated to a stacky fan). We now prove the basic properties of $\calX_G(\Delta,\Phi)$ stated in Theorem \ref{mainthm_stackyfan=toroidalbordification} from the introduction.

\begin{proposition}\label{prop_toroidalstructure}
    The toroidal bordification $\calX_G(\Delta,\Phi)$ is a separated scheme that is locally of finite type over $k$ making the embedding $G\hookrightarrow \calX_G(\Delta)$ into a toroidal embedding, whose associated cone complex is $\Sigma(\Delta,\Phi)$.
\end{proposition}

\begin{proof}
We consider first the case that $\Phi$ is trivial. Then the union 
\begin{equation*}
\calX_G(\Delta)=\bigcup_{\sigma \in \Delta }\calX_G(\sigma) 
\end{equation*}
is a covering of $\calX_G(\Delta)$ by open affine schemes that are of finite type over $k$. Hence $\calX_G(\Delta)$ is locally of finite type over $k$. 

Given two cones $\sigma$ and $\sigma'$ in $\Delta$, we consider the cone $\tau=\sigma\cap\sigma'$. Then we automatically have 
\begin{equation*}
    \calX_G(\sigma)\cap \calX_G(\sigma')=\calX_G(\tau) \ ,
\end{equation*}
which is again an affine scheme of finite type over $k$. This implies that $\calX_G(\Delta)$ is separated. 

Taking root stacks does not change these properties and, therefore, also $\calX(\Delta, \Phi)$ is locally of finite type and separated. 

Let $B$ be a Borel subgroup of $G$ and $T\subseteq B$ a maximal torus. Write $B=U\rtimes T$ for a unipotent subgroup $U$ of $B$. The both $U$ and $B\backslash G$ are smooth varieties. Let $\sigma$ be a cone in $\Delta_B$. Then, Zariski-locally on $\calX_B(\sigma,\phi_\sigma)$, there is a toric chart $\calX_B(\sigma,\phi_\sigma)\rightarrow \G_m^{\dim U} \times \calX_T(\sigma,\phi_\sigma) $. Since $\calX_G(\sigma,\phi_\sigma)\rightarrow B\backslash G$ is a Zariski-locally trivial principal $\calX_B(\sigma, \phi_\sigma)$-bundle, then, Zariski-locally on $\calX_G$, there is a toric \'etale chart $\calX_G(\sigma,\phi_\sigma)\rightarrow \G_m^{\dim U+\dim G/B} \times  \calX_T(\sigma,\phi_\sigma)$. Since $\calX_G(\sigma,\phi_\sigma)$ is an open substack of $\calX_G(\Delta,\Phi)$, this proves that $G\hookrightarrow \calX_G(\Delta,\Phi)$ is a toroidal embedding. 

The toric charts constructed above are in fact already small and, thus, we also find that the cone complex associated to this toroidal embedding is $\Sigma(\Delta, \Phi)$.
\end{proof}

\begin{corollary}
    There is a natural order-reversing one-to-one correspondence between the toroidal strata of $\calX_G(\Delta)$ and the cones in $\Delta$.
\end{corollary}

\begin{proof}
    This is an immediate consequence of Proposition \ref{prop_toroidalstructure} and the strata-cone correspondence for toroidal embeddings in \cite[Section II.1, page 71]{KKMSD_toroidal}.
\end{proof}

\begin{corollary}
    The Deligne-Mumford stack $\calX_G(\Delta,\Phi)$ is smooth if and only if every cone in $\Delta$ is generated by a basis of the lattice dual to $\widetilde{S}_\sigma$, for every $\sigma\in\Delta$.
\end{corollary}

\begin{proof}
    This is an immediate consequence of Proposition \ref{prop_toroidalstructure}, the characterization of smooth toric stacks via their stacky fans (see e.g. \cite[Theorem 3.10.7]{GillamMolcho}), and the fact that smoothness of algebraic varieties may be checked \'etale locally. 
\end{proof}

\begin{proposition}
    For a one-parameter subgroup $u\colon \GG_m\hookrightarrow G$ we have
\begin{equation*}
\lim_{t\rightarrow 0}u(t)\in \calX_G(\Delta,\Phi)
\end{equation*}
if and only if the $u\in \Delta$.
\end{proposition}

\begin{proof}
    There is a maximal torus $T\leq G$ such that $u$ lands in $T$. This means that $u$ is also a one-parameter subgroup of $T$. Let $T$ be an arbitrary such torus $T$. By construction, the one-parameter subgroup $u$ has a limit in $\calX_G(\Delta,\Phi)$ if and only it has a limit in the toric root stack $\calX_G(\Delta_T,\Phi_T)$. This is the case if and only if $u$ has a limit in the coarse moduli space $\calX_G(\Delta_T)$. But $u$ has a limit in $\calX_G(\Delta_T)$ if and only if $u\in \Delta_T$ by \cite[Chapter 1]{CoxLittleSchenk_toric}.
\end{proof}

\begin{proposition}
    The multiplication $G\times G\rightarrow G$ extends to a natural right operation of $G$ on $\calX_G(\Delta,\Phi)$
\end{proposition}

\begin{proof}
By construction there is a natural right operation of $G$ on $\calX_G(\sigma,\phi_\sigma)$ for all $\sigma\in\Delta$. Given a face $\tau$ of $\sigma$, the natural diagram
\begin{equation*}
    \begin{tikzcd}
         \calX_G(\tau,\phi_\tau)\times G \arrow[rr]\arrow[d,"\id\times\subseteq"] && \calX_G(\tau,\phi_\tau)\arrow[d,"\subseteq"]\\  
        \calX_G(\sigma,\phi_\sigma) \arrow[rr] \times G && \calX_G(\sigma,\phi_\sigma)
    \end{tikzcd}
\end{equation*}
commutes. Hence there is a natural $G$ operation from the left on all of $\calX(\Delta,\Phi)$.
\end{proof}

\begin{proposition}
    \label{prop_PermutationOfPiecesOfCover}
    Let $(\Delta,\Phi)$ be a stacky fan such that the action of $G(k)$ on $G^{\build}$ lifts to an action of $G(k)$ on $(\Delta,\Phi)$. Let $g\in G(k)$, $(\sigma,\phi_{\sigma})\in (\Delta, \Phi)$, and let $g\cdot (\sigma,\phi_{\sigma})= (\sigma, \phi_{\sigma'})$. Then the automorphism of $G$ given by conjugation by $g$ extends to an isomorphism
    \[
L_g: \calX_G(\sigma,\phi_{\sigma}) \rightarrow \calX_G (\sigma, \phi_{\sigma'}).
    \]
    Given a face $\tau$ of $\sigma$, we write $\tau'$ for the corresponding face of $\sigma'$. The isomorphism $L_{g}^\sigma$ restricts the isomorphism 
\begin{equation*}
L_{g}^\tau \colon\calX_G(\tau,\phi_\tau)\xlongrightarrow{\sim} \colon\calX_G(\tau',\phi_{\tau'}) \ .
\end{equation*}
\end{proposition}

Proposition \ref{prop_PermutationOfPiecesOfCover} immediately implies that $L_{g}$ extends to an automorphism of $\calX(\Delta, \Phi)$, which proves Property (vi) in Theorem \ref{mainthm_stackyfan=toroidalbordification}.

\begin{proof}[Proof of Proposition \ref{prop_PermutationOfPiecesOfCover}]
    Since the action of $G(k)$ on $G^{\build}$ lifts to action of $G(k)$ on $(\Delta,\Phi)$, the cone $\sigma$ has to be a subset of the Weyl chamber associated to a Borel subgroup $B$ with maximal torus $T$. Findi $B'$ and $T'$ analogously for $\sigma'$. Since $g$ maps $\sigma$ to $\sigma'$, we automatically have that $B'=g^{-1}Bg$ and $g^{-1}Tg$. This implies that the conjugation extends to an isomorphism $\calX_B(\sigma,\phi_\sigma)\xrightarrow{\sim}\calX_{B'}(\sigma',\phi_{\sigma'})$. Furthermore, the conjugation by $g$ induces an isomorphism $B\backslash G\xrightarrow{\sim}B'\backslash G$. Hence it induces an isomorphism  
    \begin{equation*}
        \calX_G(\sigma,\phi_{\sigma}) \xlongrightarrow{\sim} \calX_G (\sigma', \phi_{\sigma'})
    \end{equation*}
    as claimed. The second claim is an immediate consequence of the above construction.
\end{proof}

Given a homomorphism of algebraic groups $\xi \colon G_1\rightarrow G_2$ there is an induced map $\xi^{\build}\colon G_1^{\build}\rightarrow G_2^{\build}$ that maps apartments to apartments. Given stacky fans, $(\Delta_i,\Phi_i)$ in $G_i^{\build}$, we say that $\xi^{\build}$ maps $(\Delta_1,\Phi_1)$ to $(\Delta_2,\Phi_2)$ if for each cone $\sigma_1\in \Delta_1$:
\begin{itemize}
    \item there is some cone $\sigma_2$ in $\Delta_2$ such that $\xi^{\build}(\sigma_1)\subseteq \sigma_2$ and
    \item  there is semigroup homomorphism $\tilde{\xi}\colon\tilde{S}_{\sigma_2} \rightarrow \tilde{S}_{\sigma_1}$ such that the following diagram commutes
    \begin{equation*}\begin{tikzcd}
\tilde{S}_{\sigma_2} \arrow[r, "\tilde{\xi}^{\build}"] & \tilde{S}_{\sigma_1}  \\
{S}_{\sigma_2} \arrow[u,"\phi_{\sigma_2}"]  \arrow[r, "\xi^{\build}"] & {S}_{\sigma_1} \arrow[u,"\phi_{\sigma_1}"].  
\end{tikzcd}\ \end{equation*}
where we use $\xi$ in an abuse of notation to denote the morphism $S_{\sigma_2}\rightarrow S_{\sigma_1}$ induced by the map of $\xi^{\build}\colon \sigma_2 \rightarrow \sigma_1$.
\end{itemize}

\begin{lemma}\label{lemma_uniquenessofextensionofkummer}
    With the setup above, if the morphism $\tilde{\xi}^{\build}$ exists, then it is unique.
\end{lemma}

\begin{proof}
    Let $\tilde{s}\in \tilde{S}_{\sigma_2},$ then because $\phi_{\sigma_2}$ is a Kummer homomorphism these is some $n\in \N$ and $s\in S_{\sigma_2}$ such that $n\cdot \tilde{s}=\phi_{\sigma_1}(s)$. So then by the commutative diagram we have that $n\cdot \tilde{\xi}^{\build}(\tilde{s})=\phi_{\sigma_1}(\xi(s))$. Because $\tilde{S}_{\sigma_1}$ is a toric monoid this means that $\tilde{\xi}^{\build}(\tilde{s})$ is uniquely determined.
\end{proof}

\begin{proposition}\label{Prop_FunctorialityOf}
    Let $\xi \colon G_1\rightarrow G_2$ be a group homomorphism such that $\xi^{\build}$ maps $(\Delta_1,\Phi_1)$ to $(\Delta_2,\Phi_2)$. Then $\xi$ extends to an equivariant morphism $\calX(\Delta_1,\Phi_1)\rightarrow \calX(\Delta_2,\Phi_2)$.
\end{proposition}
\begin{proof}
    Given a maximal torus $T_1\subseteq G_1$ the morphism $\xi$ maps $T_1$ to another maximal torus $T_2$ in $G_2$.
    Let $\sigma_1\subseteq T_1^{\trop}$ be a cone in $\Delta_1$, and let $\sigma_2\supseteq \xi^{\build}(\sigma_1)$. By assumption $\xi$ induces a morphism $U_{\sigma_1}\rightarrow U_{\sigma_2}$, which then extends to a morphism $\calU_{(\sigma_1,\phi_{\sigma_1})}\rightarrow \calU_{(\sigma_2,\phi_{\sigma_2})}$, and this extension is unique by Lemma \ref{lemma_uniquenessofextensionofkummer}.
    But then this extends to a a right $G$-equivariant morphsim $\calX_{G_1}(\sigma_1,\phi_1)\rightarrow \calX_{G_2}(\sigma_2,\phi_2)$. So for each cone $\sigma_1 \in \Delta$ we have a morphism $\calX_{G_1}(\sigma_1,\phi_{\sigma_1})\rightarrow \calX_{G_2}(\Delta_2,\Phi_2)$. and these maps glue.
\end{proof}

\section{Toroidal skeletons and tropicalization}\label{Sec_SkeletonAndBuilding}

In Section \ref{section_toroidalbord} we introduced a toroidal bordification of $G$ associated to a fan $\Delta$ in $G^{\build}$, which we denoted by $\calX_G(\Delta)$. In this section we study the connections between the non-Archimedean analytification of $\calX_G(\Delta)^{\beth}$ and $G^{\build}$. In particular we will prove Theorem \ref{mainthm_skeletonbuilding}.

At this point we recall once more that the scheme $\calX_G(\Delta)$ is not quasi-paracompact and so $\calX_G(\Delta)^{\beth}$ does not have the strucure of a Berkovich analytic space. The issue with giving $\calX_G(\Delta)^{\beth}$ the structure of a Berkovich space is that in general Berkovich spaces can only be glued along closed subsets if the covering by closed subsets is locally finite, which, in general, is not true for $\calX_G(\Delta)$. Thus we invoke the workaround described in Remark \ref{remark_bethvsquasiparacompact} above. For each cone $\sigma\in \Delta$ the space $\calX_G(\sigma)^{\beth}$ has tropicalization $\Sigmabar_{\sigma}$ and there is a tropicalization map $\trop_{\calX_G(\sigma)}\colon\calX_G(\sigma)^{\beth}\rightarrow \Sigmabar_{\sigma}$. Note that these maps glue to define a map $\trop_{\calX_G(\Delta)}\colon \calX_G(\Delta)^{\beth}\rightarrow \Sigmabar_{\Delta} $, where $\Sigmabar_{\Delta}$ is the colimit of the topological spaces $\Sigmabar_{\sigma}$, furthermore the sections $J_{\calX_G(\sigma)}\colon \Sigmabar_{\sigma}\hookrightarrow \calX_G(\Delta)^{\beth}$ glue to a section $J_{\calX_G(\Delta)}\colon \Sigmabar_{\Delta} \hookrightarrow \calX_G(\Delta)^{\beth}$. 
The composition of the tropicalization map with the section gives a retraction $\bfp_{\calX_G(\Delta)}\colon \calX_G(\Delta)^{\beth}\rightarrow \calX_G(\Delta)^{\beth}$ whose image is equal to the image of the section $J_{\calX_G(\Delta)}$. Recall that there is another retraction $\bfq\colon \calX_G(\Delta)^{\beth}\rightarrow \calX_G(\Delta)^{\beth}$ given by $p\mapsto p\ast \bfg$ where $\bfg$ is the Shilov boundary point of $G^{\beth}$.

Furthermore, recall that the universal property of the colimit gives a continuous map $c \colon \calX_G(\Delta)^{\beth}\rightarrow \calX_G(\Delta)^{\an}$. For each $\sigma$ we have that the restriction of $c$ to $\calX_G(\sigma)^{\beth}$ is the natural inclusion of $\calX_G(\sigma)^{\beth}$ into $\calX_G(\Delta)^{\an}$.  The space $\calX_G(\Delta)^{\an}$ contains $G^{\an}$ and thus the space $\fD=G(k)T^{\an}G^{\beth}$. We also have that there is an embedding $\Theta\colon G^{\build}_{\Berk}\hookrightarrow G^{\an}$, by construction we have also an embedding $\iota_{\Delta\subseteq G^{\build}}\colon \Sigma_{\Delta}\hookrightarrow G^{\build}_{\weak} $ and there is a continuous bijection given by the identity $G^{\build}_{\weak}\rightarrow G^{\build}_{\Berk}$. Finally recall that we have a tropicalization map $\trop_{\build}\colon \fD\rightarrow G^{\build}_{\Berk}$ given by mapping a point of the form $gsh$ to $(g,\trop(s))\in G^{\build}_{Berk}$ where $g\in G(k)$, $s\in T^{\an}$, and $h\in G^{\beth}$. The composition $\Theta \circ \trop_{\build}$ is given by $\bfq$.

\begin{proof}[Proof of the Theorem \ref{mainthm_skeletonbuilding}]
To see that $c^{-1}(G^{\an})$ is equal to $\trop_{\calX_G(\Delta)}^{-1}(\Sigma_{\Delta})$ it suffices to observe that $G^{\an}$ is the open stratum of the toroidal embedding $G^{\an}\hookrightarrow \calX_G(\Delta)$. To see that $c$ maps $c^{-1}(G^{\an})$ into $\fD$ one can fix a maximal torus $T$ in $G$, then for $\sigma\in \Delta$ where $\sigma\in T^{\trop}$ we observe that $c$ defines a homeomorphism when restricted to $\calX_G(\sigma)^{\beth}$, and $\trop_{\calX_G(\Delta)}^{-1}(\Sigma_{\sigma})$ is contained in $T^{\an}G^{\beth}$. 
We have that $\bfq=\bfp_{\calX_G(\Delta)}$ by Lemma \ref{lemma_bfpisbfq} below.
That the diagram commutes follows from the fact that $\Theta\circ \trop_{\build}$ is given by $p \mapsto p \ast \bfg$ and the fact that $\bfq=\bfp_{\calX_G(\Delta)}$. This proves Part \eqref{item_mainthmb1}.

To show Part \eqref{item_mainthmb2} we assume now that $\Delta$ covers $G^{\build}$. It follows immediately that $\iota_{\Delta \subseteq G^{\build}}$ defines a homeomorphism onto $G^{\build}_{\weak}$ and thus $c$ defines a continuous bijection from $\Sigma_{\Delta}$ onto $\Theta(G^{\build})$. To see that $c$ defines a continuous bijection from $c^{-1}(G^{\an})$ onto $\fD$ fix a maximal torus $T$. Let $\Delta_T$ be the restriction of $\Delta$ to $T^{\trop}$. Then $\Delta_T$ is a finite fan in $T^{\trop}$ and $c^{-1}(G^{\an})\cap \calX_G(\Delta_T)^{\beth}$ is equal to $T^{\an}G^{\beth}$ because $c^{-1}(G^{\an})\cap \calX_G(\Delta_T)^{\beth}=\trop_{\calX_G(\Delta)}^{-1}(\Sigma_{\Delta_T})$. Thus the image of $c^{-1}(G^{\an})$ under $c$ is $\fD$. Furthermore we have that $\calX_G(\Delta_T)$ is locally finite type so $c$ defines a homeomorphism onto its image and so the restriction of $c$ to $T^{\an}G^{\beth}$ is an injection, therefore $c$ defines a continuous bijection from $c^{-1}(G^{\an})$ to $\fD$.
\end{proof}

\begin{lemma}\label{lemma_bfpisbfq}
    The map $\bfp_{\calX_G(\Delta)}\colon \calX_G(\sigma)^{\beth}\rightarrow J_{\calX_G(\sigma)}(\Sigmabar_{\sigma})$ is equal to $\bfq$. 
\end{lemma}

\begin{proof}
    Notice that there is a dominant surjective morphism of toroidal embeddings $\calX_T(\sigma)\times G \rightarrow \calX_G(\sigma)$. From this it follows that we have a commutative diagram:
    \begin{equation*}
    \begin{tikzcd}
    \calX_T(\sigma)^{\beth }\times G^{\beth} \arrow[rr] \arrow[d,"\bfp_{\calX_T(\sigma)\times G}"] && \calX_G (\sigma)^{\beth} \arrow[d,"\bfp_{\calX_G(\sigma)}"] \\
    J_{\calX_T(\sigma)\times G}\left(\Sigma_{\calX_T(\sigma)\times G}\right) \arrow [rr] && J_{\calX_G(\sigma)}\left(\Sigmabar_{\sigma}\right).
        \end{tikzcd}
\end{equation*}
For a point $p\in \calX_T(\sigma)^{\beth }\times G^{\beth}$ of the form $(s,g)$ we have that $\bfp_{{\calX_T(\sigma)\times G}}(p)=\omega_{\bft_{\trop(s)}, \bfg}$ where $\omega_{\bft_{\trop(s)}}$ is the Shilov boundary of $\calM(\bft_{\trop(s)})\times \calM(\bfg)$. Thus we have that $\bfp_{\calX_G(\sigma)}(sg)= \bft_{\trop(s)}\ast\bfg = (sg)\ast \bfg$.
\end{proof}

The following commutative diagram illustrates some of the finer relations between all the various topological spaces in play.

\begin{equation*}
    \begin{tikzcd}
        \calX_G(\Delta)^\beth \arrow[dddd,"\trop_{\calX_G(\Delta)}"] \arrow[rrrrr, "c"]& & & & &\calX_G(\Delta)^{\an}\arrow[dddd,"\mathbf{q}"]\\ \\ & \trop_{\calX_G(\Delta)}^{-1}(\Sigma_{\Delta})\arrow[luu,"\supseteq"']\arrow[dd,"\trop_{\calX_G(\Delta)}"]\arrow[rr,"c"]& &  \fD \arrow[rruu,"\subseteq"]\arrow[dd,"\trop_{\build}"']\arrow[r,"\subseteq"]& G^{\an}\arrow[dd,"\mathbf{q}"]\arrow[ruu,"\subseteq"']& \\ \\
        \Sigmabar_{\Delta} \arrow[uuuu,bend left, "J_{\calX_G(\Delta)}"]& \Sigma_{\Delta}\arrow[l,"\supseteq"']\arrow[r,hookrightarrow,"i_{\Delta\subseteq G^{\build}}"]\arrow[uu,bend left,"J_{\calX_G(\Delta)}"] &  G_{\weak}^{\build} \arrow[r,"\id_{G^{\build}}"] & G^{\build}_{\Berk} \arrow[uu,bend right,"\Theta"']\arrow[r,hookrightarrow,"\Theta"] & G^{\an}\arrow[r,"\subseteq"]&\calX_G(\Delta)^{\an}
    \end{tikzcd}
\end{equation*}

\bigskip

\part{Spherical tropicalization and wonderful compactifications}
\bigskip

\section{Spherical varieties and toroidal embeddings} 
Let $X$ be a spherical $G$-variety, that is, is a normal $G$-variety such that some Borel subgroup $B\subseteq G$ has an open orbit in $X$.
Tevelev and Vogiannou introduced a tropicalization map for spherical varieties in \cite{TevelevVogiannou}, and
 in \cite{Coles} the spherical tropcialization map was extended to the Berkovich analytification of the spherical variety. As a result there is a retraction $\bfp_G\colon X^{\beth}\rightarrow \overline{\Sigma}^{G}_{X}$, where the image $\Sigmabar^{G}_{X}$ is a closed subspace of $X^{\beth}$. The spherical variety $X$ has an open $G$-orbit, which is isomorphic to a spherical homogenous space $G/H$. When the open embedding $G/H\hookrightarrow X$ is toroidal, we also have a toroidal tropicalization map together with a section giving rise to the retraction map $\bfp \colon X^{\beth}\rightarrow \overline{\Sigma}_X$.
 A priori the retractions $\bfp$ and $\bfp_G$ are distinct. In Section \ref{Sec_SpherTrop} we will show that these two maps agree. Before getting there, in this section we will review some facts about spherical varieties and collect some necessary results about the particular case of spherical varieties $X$ where the open embedding of the $G$-orbit $G/H\hookrightarrow X$ is toroidal.

\subsection{The Luna-Vust theory of spherical varieties}\label{subsubsec_LVtheory}

We will now recall the Luna-Vust theory of spherical varieties. Introduced in \cite{LunaVust}, this theory gives a combinatorial classification of $G$-equivariant open embeddings of a spherical homogeneous space $G/H$ into normal $G$-varieties. In the case when $G=G/H=T$ where $T$ is a torus the Luna-Vust theory gives the classification of toric varieities in terms of fans. The results of this subsection are not original. We follow the exposition given in \cite{Knop}, to which we also refer to for the proofs of the results quoted below.

\begin{definition}
    Let $X$ be a normal $G$-variety. We say $X$ is \textit{spherical} if there is an open $B$-orbit for some Borel subgroup $B\subseteq G$. If $G/H$ is a spherical homogeneous space, and $X$ is a normal $G$ variety equipped with a $G$-equivariant open embedding $G/H\hookrightarrow X$ then we say $X$ is a $G/H$-\textit{embedding}.
\end{definition}

To state the classification of  $G/H$-embeddings $G/H\hookrightarrow X$ we need to introduce some data associated to $G/H$. Fix a Borel subgroup of $B\subseteq G$ with open orbit in $G/H$. Let $M$ be the character lattice of $B$, and define the $B$ semi-invariant rational functions to be:
\[
k(G/H)^{(B)}\coloneq\left\{
  f\in k(G/H) \;\middle|\;
  \begin{aligned}
  & \textrm{There exists $\chi_f\in M$ such that for any } \\
  & b\in B(k) \textrm{ we have } b\cdot f= \chi_f(b)\cdot f
  \end{aligned}
\right\}
\]
Because the $B$-orbit is open we have a well defined map $k(G/H)^{(B)}\rightarrow M$ given by $f\mapsto \chi_f$ the image of this map is called the \textit{weight lattice of }$G/H$, we will denote it $M(G/H)$, and the kernel of this map is exactly $k^{\times}$. Let $N(G/H)\coloneq \Hom(M(G/H), \ZZ)$, we call $N(G/H)$ the \textit{coweight lattice of} $G/H$. Define $N_{\R}(G/H)\coloneq N(G/H)\otimes_{\ZZ} \R$. Given any valuation $\val:k(G/H)^{\times} \rightarrow \R$ we can define an element $\varrho(\val)\in N(G/H)$ by $\varrho(\val)(\chi_f)=\val(f)$. Let $\calV$ be the set of $G(k)$-invariant valuations on $k(G/H)^{\times}$, then $\varrho:\calV\rightarrow N(G/H)$ is an injection and the image is a finitely generated convex cone that spans $N(G/H)$ by \cite[Corollary 5.3]{Knop}. We will identify $\calV$ with its image in $N(G/H)$. The cone $\calV$ is also cosimplicial, meaning that there are linearly independent characters $\chi_1,\ldots \chi_k$ such that 
\begin{equation*}
\calV=\big\{\lambda \in N(G/H) \mid \lambda(\chi_i)\geq 0 \textrm{ for any $i=1,\ldots k$}\big\} \ .
\end{equation*}
Note that for valuations which are \textit{not} necessarily $G(k)$-invariant the map $\varrho$ won't necessarily be injective. One particularly important source of elements in $N(G/H)$ are $B$-invariant divisors. For any $X\supseteq G/H$ let 
\begin{equation*}
\calD(X)=\{D\subset X \mid \textrm{$D$ a $B$-invariant irreducible divisor}\} \ .
\end{equation*}
For an irreducible divisor $D\subseteq X$ let $\val_D$ be the valuation on $k(G/H)$ define by $D$.

We can now define colored cones and colored fans.

\begin{definition}
    A \textit{colored cone} is a pair $(\sigma, \calF)$, where $\sigma \subseteq N_{\R}(G/H)$ is a strictly convex finitely generated rational cone, and $\calF\subseteq \calD(G/H)$ is a finite subset, such that:
    \begin{itemize}
        \item $\sigma$ is generated by $\{\varrho(\val_D) \mid D\in \calF \}$, and finitely many elements of $\calV$,
        \item The intersection of the relative interior of $\sigma$ with $\calV$ is nonempty,
        \item $0\notin \big\{\varrho(\val_D) \mid D\in \calF \big\}$.
    \end{itemize}
\end{definition}

Let $Y\subseteq X$ be a $G$-orbit,  define 
\begin{equation*}\calD_Y(X):=\big\{D\in \calD(X) \mid D\supseteq Y\big\}
\end{equation*}
and 
\begin{equation*}
\calF_Y(X):= \big\{D \in \calD_Y(X)\big\} \ .
\end{equation*}
Let $\sigma_Y(X)$ be the cone in $N_{\R}N(G/H)$ generated by $\big\{\varrho(\val_D) \mid D\in \calD_Y(X)\big\}$. Then $(\sigma_Y(X),\calF_Y(X))$ is a colored cone. We say that $X$ is a \textit{simple} $G/H$-embedding if it has a unique closed $G$-orbit.

\begin{proposition}[{\cite[Theorem 3.1]{Knop}}]
    The map $X\mapsto \big(\sigma_Y(X),\calF_Y(X) \big)$, where $Y$ is the unique closed $G$-orbit of $X$, defines a bijection between isomorphism classes of simple spherical embeddings $G/H \hookrightarrow X$ and colored cones in $N_{\R}(G/H)$.
\end{proposition}

This correspondence can in fact be stated as an equivalence of categories but for the purposes of this article we do not need the details.
In the case when $G=G/H=T$ is a torus, a simple embedding is the same as an affine variety, and in analogy with toric varieties, a spherical variety is glued from simple varieties, with the gluing specified by a colored fan.

\begin{definition} A \textit{colored face} of a colored cone $(\sigma,\calF)$ is a colored cone $(\tau, \calE)$ such that $\tau$ is a face of $\sigma$ and $\calE=\{D\in \calF \mid \varrho(\val_D)\in \tau \}$.
\end{definition}

\begin{proposition}[{\cite[Lemma 3.2]{Knop}}]
    Let $Y$ be a $G$-orbit in $X$, the map $Z\mapsto \big(\sigma_Z(X), \calF_Z(X)\big)$ defines a bijection between colored faces of $\big(\sigma_Y(X),\calF_Y(X)\big)$ and $G$-orbits $Z\supseteq Y$.
\end{proposition}

\begin{definition}
    A \textit{colored fan} $\frakF$ is a finite collection of colored cones such that:
    \begin{itemize}
        \item   for any $(\sigma,\calF)\in \frakF$ and any colored face $(\tau,\calE)$ of $(\sigma,\calF)$, then $(\tau,\calE)\in \frakF$
        \item for any $\val\in \calV$ there is at most one colored cone $(\sigma,\calF)\in\frakF$ containing $\val$ in its relative interior.\footnote{This second condition implies that for any two colored cones $(\sigma,\calF)$ and  $(\tau,\calE)$ the intersection of $\sigma$ and $\tau$ is  face of both $\sigma $ and $\tau$ if the intersection is contained in $\calV$. It is not true that $\sigma$ and $\tau$ intersect in a face if the intersection lies outside of $\calV$, i.e. the underlying collection of cones in $\frakF$ is not necessarily a fan unless the cones are contained in $\calV$. We refer the reader to \cite[Theorem 1.4]{Gagliardi} for a geometric interpretation of when the underlying cones fail to form a fan.}.
    \end{itemize}
\end{definition}

\begin{theorem}[{\cite[Theorem 3.3]{Knop}}]
    The map $X\mapsto \frakF(X)$ is a bijection between isomorphism classes of $G/H$-embeddings and colored fans in $N_{\R}(G/H)$.
\end{theorem}

A $G$-orbit $Y\subseteq X$ is a spherical homogenous space (\cite[Corollary 2.2.]{Knop}) and the embedding $Y\hookrightarrow \overline Y$, where $\overline{Y}$ is the closure of $Y$ in $X$, is a $Y$-embedding. We can describe the colored fan $\mathfrak{F}(\overline{Y})$ in terms of $\mathfrak{F}(X)$ as follows. 
Let $\tau$ be the cone in $\mathfrak{F}(X)$ that corresponds to $Y$, then consider $M(\tau)=\tau^{\perp}\cap M(G/H)$. This consists of $B$ semi-invariant rational functions that have no poles or zeros on $Y$, and so we can restrict these functions to $Y$, thus giving a map $M(\tau)\rightarrow M(Y)$. This map is an isomorphism by Theorem 6.3 of \cite{Knop}, so we have an inclusion $\iota: M(Y)\hookrightarrow M(G/H)$ and thus an induced surjection $\iota^*:N_{\R}(G/H)\rightarrow N_{\R}(Y)$. If $N_{\R}(\tau)=N_{\R}(G/H)/\vspan(\tau)$ then this surjection is just the quotient map $N_{\R}(G/H)\rightarrow N_{\R}(\tau)$. Define
\[
\calD_{\iota}:=\big\{D\in \calD(G/H) \mid \textrm{$D\cap Y$ is a divisor of $Y$.}\big\}
\]
and
\[
\Star(\tau,\calE):=\big\{(\iota^*(\sigma),\calF\cap \calD_{\iota}) \mid (\sigma,\calF)\in \mathfrak{F}(X) \textrm{ and $(\tau,\calE)$ is a colored face of $(\sigma,\calF)$}\big\}.
\]

\begin{theorem}\label{thm_OrbitFan}
    Let $Y$ be a $G$-orbit of $X$ corresponding to the colored cone $(\tau, \calE)\in \mathfrak{F}(X)$. Then $M(Y)=M(\tau)$, $N_{\R}(Y)=N_{\R}(\tau)$, and $\calD(Y)=\calD_{\iota}$. Furthermore if $\overline{Y}$ is the closure of $Y$ in $X$ then $\mathfrak{F}(\overline{Y})$ is $\Star(\tau,\calE)$. 
\end{theorem}

\subsection{Toroidal spherical embeddings}

We will recall some results on when an $G/H$-embedding is also a toroidal embedding, and discuss properties satisfied by such embeddings.

\begin{definition}
    We say that a $G/H$-embedding $G/H\hookrightarrow X$ is a \textit{toroidal spherical embedding} if the associated colored fan $\mathfrak{F}(X)$ contains no colors.
\end{definition}

\begin{remark}
    Notice that the Luna-Vust theory gives a bijection between fans in $\calV$ and toroidal spherical embeddings.
\end{remark}

\begin{example}\label{Ex_ToroidalEmbed}
Fix a reductive group $G$, a Borel $B$, and maximal torus $T$.
The group $G\times G$ acts on $G$ via $(g,h)\cdot x\mapsto gxh^{-1}$. This makes the group $G$ a spherical $G\times G$-variety as the group $B\times B^{\opp}$ has an open orbit in $G$. Given our choice of Borel and torus the coweight lattice of $G/H$, $N(G/H)$, is identified with the cocharacter lattice of $T$, and the valuation cone $\calV \subseteq N_{\R}(G/H)$ is identified with the negative Weyl chamber defined by $B$ (i.e. the Weyl chamber defined by $B^{\opp}$). Toroidal $G\times G$-equivariant embeddings $G\hookrightarrow X$ are classified by fans supported in $\calV$. If $G$ is of adjoint type, then the Weyl chamber $\calV$ is a strictly convex cone and the fan defined by $\calV$ gives the wonderful compactification of $G$ (in the sense of \cite{deConciniProcesi}).
\end{example}

Let $X$ be a toroidal spherical embedding, let $X_0=X\setminus \left(\cup_{D\in \calD(G/H)} \overline{D}\right)$ and let $P$ be the stabilizer of $X_0$. Then $P$ is a parabolic subgroup containing $B$. Let $L$ be a Levy subgroup of $P$ and let $P_u$ be the unipotent radical of $P$, so $L\cap P_u = {e}$ and $P=LP_u$. The following theorem then gives a description of the local structure of $X$.

\begin{theorem}[\cite{PerrinSurvey} Proposition 3.3.2]\label{thm_LocalStructureTheorem} Let $G/H\hookrightarrow X$ be a toroidal spherical embedding.
There is a closed $L$-stable subvariety $Z\subseteq X_0$ whose $L$-orbits are in bijection with with $G$-orbits of $X$ via $Lx\mapsto Gx$. Furthermore there is an isomorphism $P_u\times Z \mapsto X_0$ given by $(p,z)\mapsto pz$. The variety $Z$ is in fact a toric variety whose open torus $S$ is a quotient of $L/[L,L]$.
\end{theorem}

A particular result of the above theorem is that $GX_0=X$, so for any point $x\in X$ there is an open neighborhood that is isomorphic to $P_u\times Z$. Recall that any unipotent subgroup, in particular $P_u$, is isomorphic as a variety to $\A ^r_k$ for some $r\geq 0$. It also follows from the local structure theorem that the open $B$-orbit $BH/H$ is isomorphic to $P_u\times S$. The following diagram illustrates this:
\begin{equation*}\begin{tikzcd}
 && P_u\times Z \arrow[rrd, "\sim"] &&\\
P_u\times S \arrow[d,hookrightarrow] \arrow[urr,hookrightarrow]\arrow[rr, "\sim"] & &PH/H\arrow[d,hookrightarrow] \arrow[rr, hookrightarrow] & & X_0\arrow[d, hookrightarrow] \\
G\times S  \arrow[rr, twoheadrightarrow]\arrow[rrd,hookrightarrow] & & G/H  \arrow[rr, hookrightarrow] & & X. \\
&&G\times Z \arrow[rru] &&
\end{tikzcd}\end{equation*}
So the embedding  of $BH/H$ into $X_0$ is the same as the embedding of the torus $S$ into the toric variety $Z$, up to taking the product with a copy of affne space. Furthermore the embedding of $G/H$ into $X$ is Zariski locally described by the embedding of $S$ into $Z$, up to taking the product with affine space. This leads us to the following theorem.

\begin{proposition}
    Given a $G/H$-embedding $X$, the open embedding $G/H\hookrightarrow X$ is a toroidal embedding if and only if $X$ is a toroidal spherical embedding.
\end{proposition}
\begin{proof}
    As was discussed above, it follows from \ref{thm_LocalStructureTheorem} that if $X$ is spherically toroidal then  $G/H\hookrightarrow X$ is Zarski locally of the form $P_u\times S \hookrightarrow P_u\times Z$ and $Z$ is a toric variety with open torus $S$ so the embedding $S\hookrightarrow Z$ is a torus embedding.
    
    On the other hand, assume $X$ is \textit{not} spherically toroidal, we will show $X$ is \textit{not} toroidal. Let $D \in \calD$ be a color such that $D$ is contained in the colors of some cone of $\mathfrak{F}(X)$. Let $\rho$ be the ray generated by $\val_D$ in $\N_{\R}(G/H)$, and let $(\rho, \calF)$ be the colored cone of $\mathfrak{F}(X)$ whose underlying cone is given by $\rho$. Then $(\rho, \calF)$ defines a simple spherical embedding $G/H\hookrightarrow X'$, where $X'$ has exactly two $G$-orbits, the open $G$-orbit and a closed $G$-orbit which we denote $Y$. By assumption $\overline{D}\subseteq Y$ and $D$ have codimension 1 in $G/H$ as well as $Y\cap G/H=\emptyset$ and so $Y$ must have codimension 2. But $Y=X'\setminus G/H$ and so $G/H\hookrightarrow X'$ is not toroidal. Thus, $G/H\hookrightarrow X$ is not toroidal as being toroidal is an étale local condition and $X'$ is an open subembedding of $X$.
\end{proof}


\section{Spherical and toroidal tropicalization}\label{Sec_SpherTrop}

In this subsection we will review spherical tropicalization and use our results from the previous section to prove the following theorem from the introduction.

\begin{manualtheorem}{C}\label{thm_sphertroplogtrop}
Let $G$ be a connected reductive group and let $X$ be a spherical $G$-variety with open $G$-orbit $G/H$. Assume the embedding $G/H\hookrightarrow X$ is a toroidal embedding. Then we have that the spherical tropicalization map and the tropicalization map for toroidal embeddings agree. To be precise, $J_G(\Sigmabar_X^G)=J(\Sigmabar_X)$ and the retraction maps $\bfp$ and $\bfp_G$  are the same.
\end{manualtheorem}

In \cite{TevelevVogiannou} the tropicalization map for spherical homogeneous spaces was introduced, building on work of \cite{LunaVust}. Let $G/H$ be a spherical homogeneous space and let $\calV$ be the valuation cone. We will define the tropicalization map as a continuous map from the Berkovich analytification (see \cite[Theorem A]{Coles}). Recall that the first step in tropicalizing a toric varietiy is to tropicalize a torus, this extends to the toric variety as all torus orbits are isomorphic to tori. In a spherical variety all $G$-orbits are a homogeneous space, a variety isomorphic to $G/H$ for some $H$ a closed subgroup of $G$. Then the spherical tropicalization a continuous map
\[
\trop_G\colon (G/H)^{\an}\rightarrow \calV
\]
defined as follows. Let $K/k$ be a valued extension and let $\gamma\in G/H(K)$. For each $f\in k(G/H)$ there is a Zariski open $U_f\subseteq G(k)$ such that $\val_K\big(g\cdot f(\gamma)\big)$ is well-defined and constant (\cite[Subsection 3.2]{Coles}). Then define $\trop_G(\gamma)=\textrm{val}_{\gamma}$ where:
\[
\val_{\gamma}(f)=\val_K\big(g\cdot f(\gamma)\big) \,\,\,\,\, \textrm{   $g\in U_f$.}
\] These maps defined on each valued field $K$ agree with respect to valued extensions and thus define a continuous map from $G/H^{\an}$ \cite[Corollary 3.16]{Coles}. In \cite{Nashpreprint} it was shown that for a $G/H$-embedding $X$ the tropicalization map could be extended to $X$ in roughly the same the tropicalization map for a torus extends to a toric variety.
Each spherical variety $X$ has finitely many orbits and each orbit is spherical (\cite[Corollary 2.2]{Knop}). Each orbit $Y$ has an associated valuation cone $\calV(Y)$, and we will define the spherical tropicalization of $X$, denoted as $X^{\trop_G},$ to be the union of these valuations cones, topologized as follows. Let $X$ be a simple spherical embedding. Then $X$ corresponds to some colored cone $(\sigma, \calF)$ in $N_{\R}(G/H)$. Consider the space $\Hom(S_\sigma,\Rbar)$,
here $S_\sigma$ is the set of vectors in $M$ which pair with any element of $\sigma$ to a nonnegative number.
 Given $\lambda\in \Hom(S_\sigma,\Rbar)$ the set where $u$ is not equal to $\infty$ is given by $\tau^{\perp}\cap S_\sigma$ for some $\tau$ a face of $\sigma$, and so $\lambda$ defines an element of $\Hom_\R(N_{\R}(\tau),\R)$. This defines a bijection:
\[
\Hom(S_\sigma,\Rbar)\rightarrow \bigsqcup_{\textrm{$\tau$ a face of $\sigma$}}N_{\R}(\tau).
\]
Note that this is a partial compactification of $N_{\R}(G/H)$ because $N_{\R}(G/H)=N_{\R}({0})$.
By Theorem \ref{thm_OrbitFan} the face vector space $N_{\R}(\tau)$ is equal to $N_{\R}(Y)$ where $Y$ is the $G$-orbit corresponding to $Y$, and so $\calV(Y)\subseteq N_{\R}(\tau)$. So $X^{\trop_G}\subseteq \Hom(S_\sigma,\Rbar)$ and we give $X^{\trop_G}$ the subspace topology. When $X$ is not necessarily simple, we can write it as a finite union of simple subembeddings, the intersection of which is another simple subembedding, and we glue the tropicalizations accordingly. So $X^{\trop_G}$ is a collection of cones glued to give a partial compactification of $\calV$.

As with tori there is a section of the tropicalization map $\calV \hookrightarrow G/H^{\an}$, given by sending a valuation $\val$ to the points $(\eta_G,\val)\in G/H^{\an}$, when $X$ has multiple $G$-orbits these maps glue to a section $J_G\colon X^{\trop_G}\hookrightarrow X^{\an}$. The composition $J_g\circ \trop_G$ defines a retraction map $\bfp_G\colon X^{\an}\rightarrow X^{\an}$ and by \cite[Theorem A]{Coles} this map is given by $p\mapsto \bfg \ast p$ where $\bfg$ is the Shilov bpundary point of $G^{\beth}$. This description of tropicalization will be central to our later computations.
We refer the reader to \cite{NashExtendedVIaToric} and \cite{Coles} for a more in-depth discussion.

\begin{remark}\label{Rem_SpherTropFunctorial}
    Note that if $\xi\colon G_1 \rightarrow G_2$ is a group homomorphism and $X_i$ is a spherical $G_i$-variety, then spherical tropicalization is functorial in the sense that there is a continuous map $\trop(\xi)\colon \Trop_{G_1}(X_1)\rightarrow \Trop_{G_2}(X_2)$. The map $\trop(\phi)$ is induced by the map $X_1^{\an}\rightarrow X_2^{\an}$ given by $p\mapsto \bfg^ 2 \ast \xi^{\an}(p)$, here $\bfg^2$ is the Shilov boundary point of $G_2^{\beth}$. Note that this agrees with tropicalization of morphisms of spherical varieties as introduced in \cite[Section 3]{NashExtendedVIaToric}.
\end{remark}

Again in direct analogy with toric varieties, we can restrict the spherical tropicalization map from $X^{\an}$ to $X^{\beth}$ and obtain a map onto a subset of $X^{\trop_G}$ which we denote $\Sigmabar_X^G$. The composition $J_G\circ \trop_G\colon X^{\beth}\rightarrow X^{\beth}$ gives the retraction $\bfp_G$ and the image is a compact subset of $X^{\beth}$. The space $\Sigmabar_X^G$ is related to the canonical compactification of the colored fan of $X$.
We will describe $\Sigmabar_X^G$ in the case where the open $G$-orbit of $X$ is $G/H$ and the embedding $G/H\hookrightarrow X$ is toroidal, as this is our main case of interest. First, consider the case when $X$ is simple, i.e. there is one closed $G$-orbit. Then $X$ as above corresponds to a strictly convex cone $\sigma\subseteq N_{\R}(G/H)$, that is supported in $\calV$. In this case we will have that the image of $\trop_G$ is the set $\lambda\in \Hom(S_\sigma, \Rbar )$ such that $\lambda(m)\geq 0$ for any $m\in S_\sigma$, in fact the image of $\trop_G$ is exactly $\sigmabar= \Hom(S_\sigma,\Rbar_{\geq 0})$ by \cite[Theorem B]{Coles}. We then glue the cones $\sigmabar$ along their intersections as in the case of a toroidal embedding. To be precise, if $X_1$ and $X_2$ are two simple subembeddings then they intersect along a third simple subembedding $X_3$. If the cone corresponding to $X_i$ is $\sigma_i$ then $\overline{\sigma}_3$ is contained in $\overline{\sigma}_1$ and $\overline{\sigma}_2$, and in $\Sigma^G_X$ we have that $\overline{\sigma}_1$ and $\overline{\sigma}_3$ are glued along $\overline{\sigma}_3$.

We will now prove the main theorem of this Section.
Let $G/H\hookrightarrow X$ be a toroidal $G/H$-embedding. Let $X_0$, $L$, $P$, $Z$, and $S$ be as in Theorem \ref{thm_LocalStructureTheorem}, and let $T$ be a maximal torus of $B$.

\begin{proof}[Proof of Theorem \ref{thm_sphertroplogtrop}]
It suffices to show that the map $\bfp$ is equal to the map $\bfp_G$. Recall that $\bfp_G(p)=\bfg\ast p$.
The embedding $G\times S \hookrightarrow G\times Z$ is a toroidal embedding and there is a dominant surjective morphism of toroidal embeddings $G\times Z\rightarrow X$ given by left multiplication by $G$. From this it follows that we have a commutative diagram:
    \begin{equation*}
    \begin{tikzcd}
    G^{\beth}\times Z^{\beth} \arrow[rr,"m^{\an}"] \arrow[d,"\bfp_{G\times Z}"] && X^{\beth} \arrow[d,"\bfp"] \\
    J(\Sigmabar_{G\times Z}) \arrow [rr, "m^{\an}"] && J(\Sigmabar_{X}).
        \end{tikzcd}
\end{equation*}
Here $m$ is the mutliplcation map $G\times Z\rightarrow X$. Let $p\in X^{\beth}$ and let $q$ be a point in $G^{\beth}\times Z^{\beth}$ such that $m^{\an}(q)=p$ and $q$ is of the form $(g,s)$, where $g\in G^{\beth}$ and $s\in Z^{\beth}$.
Then $\bfp_{G\times Z}(q)=\omega_{\bfg, \bft_{\trop_S(s)}}$ where $\omega_{\bfg, \bft_{\trop(s)}}$ is the Shilov boundary of $\calM(\bfg) \times \calM(\bft_{\trop(s)})$ and $\trop_S$ is the tropicalization map associated to the toric variety $Z$. Thus we have that $\bfp(p)= m^{\an}\big(\bfp_{G\times Z}(q)\big)= \bfg \ast \bft_{\trop_S(s)}=\bfg \ast s= \bfg \ast p$. 
\end{proof}


\section{Affine buildings, wonderful compactifications, and spherical tropicalization}

The reductive group $G$ is a spherical $G\times G$-variety with respect to the action given by $(g,h)\cdot x=gxh^{-1}$. We can describe the Luna-Vust theory as follows. For the remainder of this section fix a maximal torus $T$ in $G$ and a Borel subgroup containg $T$. Let $d\colon G\rightarrow G\times G$ be the diagonal embedding and let $d^{-}\colon G\rightarrow G\times G$ be given by $ g\mapsto (g,g^{-1})$. Then $B\times B^{\opp}$ has an open orbit in $G$, and it follows from \cite[Theorem 3.4.12]{PerrinSurvey} that the lattice $M(G)$ can be identified with characters on $T\times T$ that are trivial on $d^{-}(T)$. So $M(G)$ can be canonically identified with characters on $d(T)$. So $N_{\R}(G)$ is $T^{\trop}$ and it also follows from \cite[Theorem 3.4.12]{PerrinSurvey} that the valuation cone $\calV$ in $N_{\R}(G)$ can be identified with the negative of the  Weyl chamber in that corresponds to $B$. So wonderful compactifications (up to $G\times G$-equivariant isomorphism) are classified by fans in $T^{\trop}$ whose support is contained the Weyl chamber determined by $B$.

 As discussed in the previous subsection we then have a spherical tropicalization map $G^{\an}\rightarrow \calV$ and there is a retraction map $\bfp_{G \times G}\colon X^{\an}\rightarrow J_{G\times G}(\calV)$ given by $p\mapsto \bfg \ast p \ast \bfg $. Before we discuss the relationship with the building we will pause to mention how this map relates to the \emph{wonderful compactification} of $G$. 

In \cite{deConciniProcesi} de Concini and Procesi use representation-theoretic methods to construct a compactification of a semisimple algebraic group with trivial center (an adjoint group) which has particularly nice properties. This compactifications is called the \emph{wonderful compactification} of $G$. We refer the reader to \cite{BrionKumar} for more background and generalizations of this construction.

In this artcile, a \emph{wonderful compactification} of $G$ is a toroidal compactification of $G$ that is equivariant with respect to the action of $G\times G$ on $G$ defined by $(g,h)\cdot x=gxh^{-1}$. 

\begin{remark}
Our terminology is a more general use of the term `wonderful compactification'. In particular a wonderful compactification will neither be smooth nor unique.
Most authors assume the group is semisimple so there is a unique final wonderful compactification, or even more they assume that the group is of adjoint type, so there that the unique final wonderful compactification is smooth.
\end{remark}

Let $\Delta$ be a fan in $\calV$ which covers $\calV$ and let $X$ be the associated Weyl compactification. The compactification $X$ gives rise to a tropicalization map for $G$. There is a tropicalization map $\trop_{won} \colon X^{\beth}\rightarrow \overline{\Sigma}_{X}$ induced by the fact that $G\hookrightarrow X$ is toroidal. In particular there is a map $\trop_{\won} \colon G^{\an}\rightarrow \Sigma_{X}$, but by Theorem \ref{thm_sphertroplogtrop} this map agrees with the spherical tropical map for $G$. The induced retraction $\bfp\colon G^{\an}\rightarrow G^{\an}$ is equal to the map $\bfp_{G\times G}$ which is given by $p\mapsto \bfg \ast p \ast \bfg$ and $\Sigma_{X}$ is the image of the valuation cone $\calV$ in $G^{\an}$. In particular, while the construction of particular wonderful embedding depends on a choice of maximal torus, the tropicalization map induced by a wonderful compactification only depends on the action of $G\times G$ on $G$.

The cone $\calV$ is a fundamental domain for the action of the Weyl group $W$ on $T^{\trop}$ so one could describe $\calV$ as a quotient of $T^{\trop}$ by $W$. Recall that we can also write $G^{\build}$ as a quotient of the product $G(k)\times T^{\trop}$. 
Then there is a map $\pi \colon G^{\build}\rightarrow \calV$ given by $(g,\lambda )\mapsto \lambda \textrm{ mod $W$}$. 
Recall that we also have a map $\trop_{\build}\colon \fD \rightarrow G^{\build}$ where $\fD\subseteq G^{\an}$.

\begin{manualtheorem}{D}\label{thm_tropbuildvswonderfultropsec9}
The map $\pi\colon G^{\build}\rightarrow \calV$ is well-defined, independent of the choice of $T$, and its restriction to the Weyl chambers in $G^{\build}$ is a linear isomorphism. Moreover, the diagram
        \begin{equation}\label{eq_factorizationsec9}\begin{tikzcd}
            & \fD \arrow[ld, "\trop_{\build}"'] \arrow[r,"\subseteq"]&G^{\an}\arrow[rd,"\trop_{G\times G}"]& \\
            G^{\build}_{\Berk}\arrow[rrr, "\pi"]  & && \calV 
\end{tikzcd}\end{equation} 
commutes and the factorization \eqref{eq_factorization} is functorial with respect to homomorphisms of reductive groups. To be precise, given $\xi: G_1\rightarrow G_2$ be a morphism of algebraic groups then the following diagram commutes:
     \begin{equation*}
\begin{tikzcd}
G_1^{\build }\arrow[d,"\xi^{\build}"'] \arrow[rr,"\pi"] & & \calV_1\arrow[d,"\trop(\xi)"]\\
G_2^{\build} \arrow[rr,"\pi"] & & \calV_2.
\end{tikzcd}
\end{equation*}
\end{manualtheorem}

\begin{proof}
We have that $\pi$ is well-defined because if $(g,\lambda)\sim (h,\mu)$ then $\lambda \textrm{ mod $W$} = \mu \textrm{ mod $W$} $. To see that $\pi$ is independent of the choice of $T$ consider that we have sections of $G^{\build}_{\Berk}$ and $\calV$ in $G^{\an}$. In fact we have a commutative diagram: 
     \begin{equation*}
\begin{tikzcd}
\Theta(G^{\build}_{\Berk})\arrow[rr,"\tilde{\pi}"] & & J(\calV )\\
G_{\Berk}^{\build}\arrow[u,"\Theta"]  \arrow[rr,"\pi"] & & \calV \arrow[u,"J"].
\end{tikzcd}
\end{equation*}
and the map $\tilde{\pi}$ is given by the map $p\mapsto \bfg \ast p$. Furthermore, this shows that Diagram \ref{eq_factorizationsec9} commutes, because $\Theta\circ \trop_{\build}$ is given by $p \mapsto p \ast \bfg$ and $J\circ \trop_{G\times G}$ is given by $p\mapsto \bfg \ast p \ast \bfg$.

    For the second diagram, we fix maximal tori $T_i\subseteq G_i$ and let $W_i$ be the corresponding Weyl group. Let $(g,\lambda)\in G_1^{\build}$ then $\trop(\xi)(\pi(g,\lambda))=\trop(\xi)(\lambda \textrm{ mod $W_1$})=\xi(\lambda) \textrm{ mod $W_2$}$, here $\xi(\lambda) $ is the image of $\lambda$ under the map $T_1^{\trop}\rightarrow T_2^{\trop}$. On the other hand we have $\pi\big(\xi^{\build}(g,\lambda)\big)=\pi\big(\xi(g),\xi(\lambda)\big)=\xi(\lambda) \textrm{ mod $W_2$}$.
\end{proof}

To end this section we would like to display how the maps introduced differ from the tropicalization map $\trop\colon T^{\an}\rightarrow T^{\trop}$. For a torus we have that $T^{\an}/T^{\beth}=T^{\build}=T^{\beth}\backslash T^{\an}/T^{\beth}=T^{\trop}$, for a general connected reductive group we have the following commutative diagram:

\begin{equation*}\begin{tikzcd}
G^{\an}\arrow[rr, "p\mapsto pG^{\beth}"] \arrow[bend right=80, rrdddd, "\trop_{G\times G}",swap] & & G^{\an}/ G^{\beth} \\
 & & \\
\fD \arrow[uu,hookrightarrow] \arrow[ddrr, "\trop_{G\times G}"] \arrow[rr,"\trop_{\build}"] & & G^{\build} \arrow[dd, "\pi"] \arrow[uu,"(g \textrm{,} \lambda)\mapsto gt_{\lambda}G^{\beth}",swap, hookrightarrow]\\ 
 & & \\
& & \calV \\
\end{tikzcd}\end{equation*} 
when $G=T$ all vertical maps are the identity map and  the maps left to right are $\trop$.


\bigskip

\part{Complements and examples}
\bigskip

\section{Bundles on chains of projective lines and their tropicalization}\label{section_bundles}

In \cite{MartensThaddeus}, Martens and Thaddeus propose a family of modular compactifications of an arbitrary connected reductive group $G$ that parametrize certain twice framed equivariant bundles on chain of projective lines. Special instances of this construction include the wonderful compactification, when $G$ is of adjoint type, all toric orbifolds, when $G=T$, and Kausz' compactification of $\GL_n$ \cite{Kausz}. In this section we expand \cite{MartensThaddeus} and study a tropicalization procedures for $G$ employing this moduli-theoretic perspective.

\subsection{Compactification of reductive as moduli stacks of bundles}
Before we recall the Martens--Thaddeus construction from \cite{MartensThaddeus}, we remind the reader of the following two facts as background:
\begin{enumerate}
    \item Let $X=X_1\cup\cdots \cup X_{k+1}$ be a nodal chain of $k+1$ projective lines starting at a $\G_m$-invariant smooth point $p_0$ of $X_1$ and ending in a $G_m$-invariant smooth point $p_\infty$ of $X_{k+1}$. Denote the nodes by $p_1,\ldots, p_k$. A principal $G$-bundle $E$ on $X$ is said to be \emph{rationally trivial}, if it is generically trivial on each component. Note that every principal $G$-bundle is rationally trivial, e.g. when $k$ is of characteristic zero. By \cite[Theorem 4.3]{MartensThaddeus_BirkhoffGrothendieck}, every rationally trivial principal $G$-bundle $E$ admits a reduction to the torus $T$. When $E$ is $\G_m$-equivariant, so is its reduction to $T$.
    \item A $\G_m$-equivariant principal $T$-bundle on $X$ is uniquely determined by cocharacters in $N$, one at each $p_0, p_1,\ldots, p_k,p_\infty$. In particular, by setting $\alpha_0=0$ and $\alpha_\infty=0$, there is a natural one-to-one correspondence between $\G_m$-equivariant principal $T$-bundle on $X$ with trivial $\G_m$-equivariant structure at $p_0$ and $p_\infty$ and $k$-tuples $(\alpha_1,\ldots, \alpha_k)$ consisting of cocharacters $\alpha_i\in N$ for $i=1,\ldots, k$. 
\end{enumerate}

Let $(\Delta_C,\Phi)$ be a simplicial stacky fan in the Weyl chamber $C$, where $\Phi$ is given by the choice of an integral element $\beta_\rho\in \rho\cap N$ for every ray $\rho$ of $\Delta$. Choose an order of the one-dimensional cones of $\Delta$, so that we may write $\big\{\sigma(1),\ldots, \sigma(k)\big\}$ for the ordered set of one-dimensional face of $k$-dimensional cone $\sigma\in\Delta_C$. 

Denote by $\calM_G(\Delta_C,\Phi)$ the category fibered in groupoids, whose fiber over a scheme $S$ consists of the following data:
\begin{enumerate}
    \item A flat and proper morphism $X\rightarrow S$, whose fibers are nodal chains of projective lines with two smooth $\G_m$-invariant sections $p_0$ and $p_\infty$.
    \item A $\G_m$-equivariant principal $G$-bundle $E$ that is rationally trivial (i.e. generically trivial on every component) and has trivial $\G_m$-equivariant structure at $p_0$ and $p_\infty$.
    \item A frame of $E_0$ and $E_\infty$ at the two $\G_m$-invariant smooth sections $p_0$ and $p_\infty$ of $X$.
\end{enumerate}
This data is assumed to be \emph{$(\Delta_C,\Phi)$-stable}, meaning that in every fiber there is a unique $k$-dimensional cone $\sigma\in\Delta$ such that $X$ is a nodal chain of $k+1$ projective lines and $E$ admits a reduction to $T$ that is isomorphic to $E\big(\beta_{\sigma(1)},\ldots, \beta_{\sigma(k)}\big)$. 

Then, \cite[Theorem 4.2]{MartensThaddeus} tells us that $\calM_G(\Delta_C,\Phi)$ is a separated Deligne-Mumford stack of finite type over $k$ that admits a $G\times G$-operation with a dense orbit isomorphic to $G$. Moreover, it is proper over $k$ if and only if $\vert\Delta_C\vert=C$ and the complement of the open and dense orbit $G$ has normal crossings by \cite[Proposition 3.7]{MartensThaddeus}.
We denote by $M_G(\Delta_C)$ the coarse moduli space of $\calM_G(\Delta_C,\Phi)$ (which does not depend on $\Phi$). By \cite[Corollary 8.2]{MartensThaddeus} this $M_{G}(\Delta_C)$ is projective if and only if the support of $\Delta_C$ is $C$ and $\Delta_C$ is a normal fan (see \cite{CoxLittleSchenk_toric}). 

\begin{example}\label{example_wonderfultropicalization}
    Suppose that $G$ is semisimple and with trivial center, i.e. of adjoint type. Then, if we choose the Weyl chamber $C$ itself as the unique maximal cone a fan, the associated moduli space $\calM_G(C)$ is the \emph{wonderful compactification} of de Concini--Procesi \cite{deConciniProcesi}. See \cite[Section 9]{MartensThaddeus} for further details on this relationship.
\end{example}

\begin{example}
    Suppose now that $G=T$ is an algebraic torus. Then the associated root system is trivial and the Weyl chamber $C$ is equal to $N_\R$. For a stacky fan $(\Delta,\Phi)$ the compacitification $\calM_G(\Delta,\Phi)$ the toric orbifold associated to $(\Delta,\Phi)$. 
\end{example}

\subsection{Toroidal tropicalization}
Let again $(\Delta_C,\Phi)$ be a simplicial stacky fan in the Weyl chamber $C$, where $\Phi$ is given by the choice of an integral element $\beta_{\rho}\in \rho\cap N$ for every one-dimensional face $\rho$ of $\Delta_C$. By \cite[Proposition 10.1]{MartensThaddeus}, the coarse moduli space $M_{G}(\Delta_C)$ is a spherical toroidal embedding with cone complex $\Sigma$. Thus, we have a natural toroidal tropicalization map 
\begin{equation*}
    \trop_{G,\Delta_C}\colon M_G(\Delta_C)^{\beth}\longrightarrow\overline{\Sigma} \ .
\end{equation*}
When we choose $\Delta_C$ such that it covers $C$ this tropicalization map agrees with the spherical tropicalization map on the dense open subset $G^{\an}$ by Theorem \ref{mainthm_sphertroplogtrop}. 

We now use the data given by the Kummer morphisms $\Phi$ to construct a root stack $M_G(\Sigma,\Phi)$ that makes the diagram
\begin{equation*}
    \begin{tikzcd}
        \calM_G(\Delta_C,\Phi)\arrow[rr] \arrow[d]& & M_G(\Delta_C,\Phi)\arrow[d]\\
        \calM_G(\Delta_C)\arrow[rr] & & M_G(\Delta_C)
    \end{tikzcd}
\end{equation*}
commute. We obtain a stacky toroidal tropicalization map 
\begin{equation*}
    \trop_{G,\Delta_C,\Phi}\colon \calM_G(\Delta_C,\Phi)^{\beth}\longrightarrow \overline{\widetilde{\Sigma}}
\end{equation*}
as explained in Section \ref{section_stackycc&toroidal} above. In view of Example \ref{example_wonderfultropicalization}, we sometimes refer to $\trop_{G,\Delta_C,\Phi}$ as the \emph{wonderful tropicalization}. Note that this map agrees with the spherical tropicalization map (in the non-stacky case) by Theorem \ref{thm_sphertroplogtrop}.

\subsection{Modular tropicalization}
In the spirit of \cite{ACP} we may give a moduli-theoretic interpretation of the wonderful tropicalization map, in particular this is a moduli-theoretic example of spherical tropicalization.

\begin{definition}
    A \emph{metric chain} $\Gamma$ consists of a chain graph $H=(V,E)$ starting at a vertex $v_1$ and ending at a vertex $v_{k+1}$ connected by edges $e_{1},\ldots, e_{k}$ and $k-2$ inner vertices $v_2,\ldots, v_{k-1}$ together with an edge length function $\ell\colon E(H)\rightarrow \R_{\geq 0}$.
\end{definition}

If we allow the edge lengths $\ell$ to take values in $\Rbar=\R\sqcup\{\infty\}$, we refer to $\Gamma$ as a \emph{pseudo-metric chain}. 

\begin{definition}
    A \emph{$(\Delta_C,\Phi)$-decoration} on a (pseudo)-metric chain $\Gamma$ is given by associating to every vertex $e$ of $\Gamma$ a vector $\beta(e)\in\{\beta_\rho\mid \rho\in\Delta_C(1)\}$ that preserves the order on the rays of $\Delta_C$ and such that there is a (automatically unique) cone $\sigma\in \Delta_C$ such that $\beta(e_i)=\beta_{\sigma(i)}$ for all $1\leq i\leq k$. 
\end{definition}

We refer to the datum $\big(H,\beta(e_1),\ldots, \beta(e_k)\big)$ as the \emph{combinatorial type} of a $(\Delta_C,\Phi)$-decorated (pseudo)-metric chain. 

Given $(\Delta_C,\Phi)$ as above, we may interpret $\Delta_C$ as a parameter space for metric $(\Delta,\Phi)$-decorated metric chains. To be precise, the point in $\Delta_C$ associated to $\beta(\Gamma,\beta(e_1),\ldots, \beta(e_k)\beta)$ is given by $\ell(e_1)\cdot \beta(e_1) + \cdots + \ell(e_k)\cdot \beta(e_k)\in \sigma$, where $\sigma\in \Delta_C$ is the unique cone that corresponds to the combinatorial type of $\big(\Gamma,\beta(e_1),\ldots, \beta(e_n)\big)$. This identification immediately extends to an identification of the canonically compactified $\overline{\Delta}_C$ as a parameter space for $(\Delta_C,\Phi)$-decorated pseudo-metric chains. 

Under this identification the wonderful tropicalization admits a modular interpretation as follows: A point $x\in \calM_G(\Delta_C,\Phi)^\beth$ can be represented by a point in $\calM_G(\Delta_C,\Phi)(R)$ where $R$ is the valuation ring of a non-Archimedean extension $L$ of $K$, which we may assume to be algebraically closed. This datum in particular gives rise to a flat and proper family $X$ of chains of nodal curves over $R$ together with a $\G_m$-equivariant principal $G$-bundle $E$ that is rationally trivial and $(\Delta_C,\Phi)$-stable. This means that the restriction $E_s$ of $E$ to the special fiber $X_s$ of $X$ admits a reduction to $E(\beta_{\sigma(1)}, \ldots, \beta_{\sigma(k)})$ and, thus, determines a cone $\sigma\in\Delta$ (generated by $\sigma(1),\ldots, \sigma(k)$). \'Etale locally around every node $p_i$ of $X_s$ (for $i=1,\ldots, k$) the family $X$ may be written as $xy=r_i$ and we set $l(e_i)=\val_R(r_i)$. Then $\trop_{G,\Delta_C,\Phi}(x)$ is precisely the point in $\sigma$ corresponding to the metric chain with edge lengths $\ell(e_1),\ldots, \ell(e_k)$.

\begin{remark}
Suppose that $G=\GL_n$. Then the datum of a $(\Delta_C,\Phi)$-decoration on a metric chain $\Gamma$ is the same as a vector bundle on $\Gamma$ (in the sense of \cite{GrossUlirschZakharov}) together with a fixed choice of transition maps on the star cover of $\Gamma$. We refer the reader to \cite[Section 2.3]{GrossUlirschZakharov} for further details on this terminology.
\end{remark}

\section{The \texorpdfstring{$A_n$}{An} case and Goldman-Iwahori space}
\label{section_examples}
  
\subsection{Goldman-Iwahori space} Let $V$ be an $n$-dimensional $k$-vector space.  Recall that a \emph{non-Archimedean norm} on $V$ is a map 
 \begin{equation*}
\gnor \colon V\rightarrow{\R} 
 \end{equation*}
 subject to the following axioms:
 \begin{enumerate}
 \item $\nor{v}\geq 0$ for all $v\in V$ and $\nor{v}=0$ if and only if $v=0$;
 \item $\nor{\gamma v} =\nor{\gamma} \cdot \nor{v}$ for all $v\in V$ and $\gamma\in k$;
 \item $\nor{v+w}\leq \max\{\nor{v},\nor{w}\}$ for all $v,w\in V$. 
 \end{enumerate}

 \begin{definition}
The \emph{Goldman-Iwahori space} associated is the set $\calN(V)$ of all non-Archimedean norms on the dual space $V^\ast$, endowed with the weakest topology that makes all evaluation maps
  \begin{equation*}\begin{split}
  \ev_v\colon \calN(V)&\longrightarrow \R\\
  \gnor &\longmapsto \nor{v}
\end{split}  \end{equation*}
for $v\in V^\ast$ continuous. 
 \end{definition}

Given a surjective linear map $f\colon V\rightarrow W$ of finite-dimensional $k$-vector spaces, there is an induced continuous map $\calN(V)\rightarrow\calN(W)$ given by
\begin{equation*}
    \gnor \longmapsto \gnor \circ f^\ast \ ,
\end{equation*}
where $f^\ast$ denotes the injective linear map that is dual to $f$. In particular, there is a natural continuous operation of $\GL(V)(k)$ on $\calN(V)$.
Moreover, when $L$ is a non-Archimedean extension of $k$,  there is a natural continuous and surjective restriction map $\calN(V_L)\rightarrow \calN(V)$. 
 
 \begin{example}
 Let $\underline{e}=(e_1,\ldots, e_n)$ be a  
 basis of $V^\ast$ and $\lambda=(\lambda_1,\ldots, \lambda_n)\in\R^n$. Then the association 
 \begin{equation*}\begin{split}
 \gnor_{\underline{e},\lambda}\colon V^\ast&\longrightarrow \R^n\\
 u=a_1e_1+\ldots+a_ne_n&\longmapsto \max_{i=1,\ldots, n}{\nor{a_i} e^{-\lambda_i}}
 \end{split}\end{equation*}
 is a non-Archimedean norm on $V^\ast$. 
 \end{example}

 In fact, when $k$ is spherically complete (e.g.\ when it carries the trivial valuation), every non-Archimedean norm on $V^\ast$ is of the form $\gnor_{\underline{e},\lambda}$ (so that $\GL(V)$ acts transitively on $\calN(V)$). Note that a choice of a basis $\underline{e}$ modulo rescaling of the basis is equivalent to choosing a maximal torus $T$ and there is a map from $\alpha\colon \GL(V)^{\build} \rightarrow  \calN(V)$ given by mapping the point $\bft_{\lambda}\in T^{\trop}\subseteq \GL(V)^{\build}$ to the norm  $\gnor_{\underline{e},\lambda}$.

 \begin{proposition}\label{prop_buildingtoGoldmanIwahori}
 The map $\alpha$ is well-defined and determines a continuous bijection from $\GL(V)^{\build}_{\Berk}$ to the Goldman-Iwahori space $\calN(V)$.
 \end{proposition}
 
 \begin{proof}
Choose a non-zero vector $v\in V$. Note that $\alpha$ may be written as a composition of continuous maps
\begin{equation*}
    \GL(V)^{\build}\hooklongrightarrow \GL(V)^{\an}\longrightarrow (\GL(V)/P_v)^{\an}\subseteq \A(V)^{\an}\longrightarrow \calN(V) \ ,
\end{equation*}
given as follows: The first map is the injection of $G^{\build}$ into the Berkovich analytic space described in Section \ref{sec_buildings} above. The second arrow is given by taking the quotient by the closed subgroup fixing the vector $v$, identifying $(G/P_v)^{\an}$ with the affine space $\A(V)^{\an}$. The last arrow is given by restricting a seminorm on the coordinate ring $k[V^\ast]$ of $\A(V)$ to $V^\ast$. This shows that $\alpha$ is well-defined and continuous. The surjectivity of $\alpha$ follows from the fact that $k$ is spherically complete. 

It remains to show that $\alpha$ is injective. Let $t_\lambda\in T^{\trop}$ and $t_\mu\in (T')^{\trop}$ be two points in $G^{\build}$ such that $\alpha(t_\lambda)=\alpha(t_{\lambda'})$ in $\calN(V)$. In this case, $\alpha(t_\lambda)$ and $\alpha(t_{\lambda'})$  lie in the same apartment and, hence, we may choose $T=T'$. In other words, we may choose a basis $\underline{e}$ of $V^\ast$ such that $\gnor_{\underline{e},\lambda}=\gnor_{\underline{e},\mu}$. Evaluating the norm at the $e_i$, we therefore see that $\lambda=\mu$. 
 \end{proof}

\subsection{An explicit interpretation of our tropicalization maps} Under the above identification, the tropicalization map 
 \begin{equation*}
 \trop_{\build}\colon \fD\longrightarrow \GL(V)^{\build}
 \end{equation*} 
 has the following explicit description: A point $x$ in $\mathfrak{D}$ may be represented by a product $gh$ such that $g$ is an invertible linear map $g\colon V_L\rightarrow V_L$ for a non-Archimedean extension $L$ of $\C$ (e.g. $L=\C((t^{\R}))$ the field of Hahn series) for which there exists a $k$-linear basis $\underline{e}$ of $V$ with respect to which $g$ has the diagonal form
 \begin{equation*}
 \begin{bmatrix}
 \ell_1  & 0 & \cdots & 0\\
 0& \ddots & \ddots & \vdots\\
 \vdots & \ddots & \ddots & 0\\
 0 & \cdots &  0 & \ell_n 
 \end{bmatrix}
\end{equation*} 
with $\ell_i \in L^\ast$ for $i=1,\ldots, n$, and $h\in \GL(V)(\calO_L)$ where $\calO_L$ is the valuation ring of $L$. Writing $\vec{\lambda}=(\ell_1,\dots, \ell_n)$, the tropicalization $\trop_{\build}(x)$ of $x$ is  the non-Archimedean norm $\gnor_{\underline{e},\val({\vec{\ell}})}$ on  $V$.

Recall that the spherical tropicalization map 
 \begin{equation*}
 \trop_{\GL(V)\times \GL(V)}\colon \GL(V)^{\an}\longrightarrow \calV
 \end{equation*} 
 is given by first identifying $\calV$ with $T^{\trop}/W$ where $T$ is a fixed maximal torus in $\GL(V)$ and $W$ is the associated Weyl group. Recall that this can be identified with the cone given by $t_{\lambda}\in T^{\trop}$ such that $\lambda_1\geq \lambda_2\geq \ldots \lambda_n$.
 Then a point $x\in \GL(V)(L)$ can be written as $gth$ where $g,h\in \GL(V)(\calO_L)$ and $t\in T(L)$ is a matrix 
 \begin{equation*}
 \begin{bmatrix}
 \ell_1  & 0 & \cdots & 0\\
 0& \ddots & \ddots & \vdots\\
 \vdots & \ddots & \ddots & 0\\
 0 & \cdots &  0 & \ell_n 
 \end{bmatrix}
\end{equation*} 
where $\val(\ell_1)\geq \ldots \geq \val(\ell_n)$.
Set $\lambda_i=\val(\ell_i)$ then $\trop_{\GL(V)\times\GL(V)}(x)$ is equal to the point $t_{\lambda}\in T^{\trop}/W$ by \cite[Theorem 2]{TevelevVogiannou}. The map $\pi\colon \GL(V)^{\build}\rightarrow \calV$ is given as follows. Let $x=\gnor_{\underline{e},\lambda }$ be a point in $\GL(V)^{\build}$, then the coordinates of $\lambda$ can be permuted to form a new vector $\lambda'$ with $\lambda'_1\geq\ldots \geq \lambda'_n$ and $\pi(x)$ is the point $\bft_{\lambda'}\in T^{\trop}/W$.

\subsection{$\PGL$ and $\SL$}\label{subsection_PGLSL}
We say that two non-Archimedean norms $\gnor$ and $\gnor'$ on $V$ are \emph{equivalent} if there is a constant $c>0$ such that 
 \begin{equation*}
 \nor{u} =c\nor{u}'
 \end{equation*}
 for all $u\in V^\ast$. We write $\calP(V)$ for the quotient of $\calN(V)$ by this equivalence relation (endowed with the quotient topology). 
 
Under the bijection $\alpha$ the extended affine building $\PGL(V)^{\build}$ may be identified with $\calP(V)$ and the extended affine building $\SL(V)^{\build}\subseteq \GL(V)^{\build}$ with the subset of norm $\gnor_{\underline{e},\lambda}$ such that $\lambda_1+\cdots +\lambda_n=0$. We  visualize this situation  in the following commutative diagram
\begin{equation*}
\begin{tikzcd}
\SL(V)^{\build} \arrow[rr] \arrow[d,"\subseteq"]& & \calS(V)\arrow[d,"\subseteq"]\\
\GL(V)^{\build}\arrow[d,"\quot"] \arrow[rr,"\alpha"]&& \calN(V)\arrow[d,"\quot"]\\
\PGL(V)^{\build} \arrow[rr] && \calP(V)
\end{tikzcd}
\end{equation*}
Moreover, note that the two vertical compositions are, in fact, homoeomorphisms, since every norm $\gnor_{\underline{e},\lambda }$ is equivalent to a unique norm with $\lambda_1+\cdots +\lambda_n=0$. This homoemorphism, however, does not preserve the integral structures on each torus (see Figure \ref{fig_thefigure} for an example). This phenomenon was first observed in an analogous situation in \cite{MartensThaddeus} and, both in this article and in \cite{MartensThaddeus} motivates the use of stacky structures around the boundary.







\bibliographystyle{amsalpha}
\bibliography{biblio}{}

\appendix


\end{document}